\documentclass[10pt]{article}

\usepackage{amsmath,amssymb,amsfonts,amsthm,bm,mathabx}

\usepackage[english]{babel} 
\usepackage[utf8]{inputenc} 
\usepackage[T1]{fontenc}    
\usepackage{lmodern}

\usepackage[numbers]{natbib}
\usepackage{csquotes}

\usepackage{authblk}
\usepackage{graphicx}
\usepackage{multirow}
\usepackage{color}
\usepackage{hyperref}
\usepackage{cleveref}
\usepackage{dsfont}
\usepackage{times}
\usepackage{float}
\usepackage[a4paper, margin=2.5cm]{geometry}
\usepackage{algorithm2e}
\usepackage{appendix}
\usepackage{subfig}
\usepackage{listings}

\newtheorem{rmk}{Remark}[section]

\newtheorem{rmkAppndx}{Remark}
\newtheorem{prpAppndx}{Proposition}

\newcommand{\df}{{\rm d}}               

\newcommand{\R}{\mathbb{R}}
\newcommand{\Z}{\mathbb{Z}}

\newcommand{\N}{\mathbb{N}}
\newcommand{\Q}{\mathbb{Q}}

\newcommand{\Ms}{\mathsf{M}}
\newcommand{\Ns}{\mathsf{N}}

\newcommand{\as}{\mathsf{a}}

\newcommand{\gs}{\mathsf{g}}

\newcommand{\is}{\mathsf{i}} 

\newcommand{\ms}{\mathsf{m}}
\newcommand{\ns}{\mathsf{n}}

\newcommand{\vs}{\mathsf{v}}

\newcommand{\xs}{\mathsf{x}}

\renewcommand{\Im}{\operatorname{Im}}
\renewcommand{\Re}{\operatorname{Re}}

\begin{document}

\title{A recursive method for computing singular solutions in  corners with homogeneous Dirichlet-Robin boundary condition with power-law coefficient variation}

\author[1]{Nicolás Piña-León \thanks{Corresponding author. Email: jpinna@us.es}}
\author[2]{Vladislav Manti\v{c} \thanks{Email: mantic@us.es}}
\author[3,4]{Sara Jim\'enez-Alfaro \thanks{Email: 
 s.jimenez-alfaro@imperial.ac.uk}}

\affil[1]{\small Instituto de Matemáticas de la Universidad de Sevilla (IMUS)\\ Edificio Celestino Mutis, Avda. Reina Mercedes  s/n, 41012, Sevilla, Spain}
\affil[2]{\small Departamento de Mecánica de Medios Continuos y Teoría de Estructuras, Escuela Técnica Superior de Ingeniería, Universidad de Sevilla\\ Camino de los Descubrimientos s/n, 41092, Sevilla, Spain}
\affil[3]{Department of Civil and Environmental Engineering, Imperial College London\\ London, SW7 2AZ, UK}
\affil[4]{Department of Engineering Science, University of Oxford, Oxford, OX1 3PJ, UK}

\date{\today} 

\maketitle


\begin{abstract}
This study introduces a recursive method for computing asymptotic solutions of the Laplace equation in corner domains with the homogeneous Dirichlet boundary condition on one side and the Robin boundary condition with a power-law coefficient variation with exponent $\alpha\in\R$  on the other side (D-R corner problem). An asymptotic solution of this D-R corner problem is given as the sum of a main term, the solution of either a homogeneous Dirichlet-Neumann (D-N) or Dirichlet-Dirichlet (D-D) corner problem, and a finite or infinite series of the associated higher-order shadow terms by using harmonic basis functions with power-logarithmic terms.  To determine this series of shadow terms, it is shown that the recursive procedures based on recursive non-homogeneous D-N or D-D corner problems are always convergent for $\alpha>-1$ or $\alpha<-1$, respectively. For the critical case $\alpha=-1$, the closed form expression of the asymptotic solution is given. Asymptotic solutions for several relevant D-R corner problems are derived and analysed. Two of these examples are applied to the problem of bridged cracks in antiplane Mode III in linear elastic fracture mechanics. The results presented can be applied to many other physical and engineering applications, such as heat transfer with the thermal resistance condition, acoustics and electrostatics with the impedance condition, and elasticity and structural analysis with the Winkler spring boundary condition.
\end{abstract}



\section{Introduction}

Consider a linear elliptic boundary value problem (BVP) in a plane domain with one or several corners on the boundary, also called angular points. As follows from the seminal works by Kondratiev \cite{Kondratiev1967} and Costabel and Dauge \cite{CostabelDauge1993}, the asymptotic solution of such a BVP near the corner vertex includes a linear combination  of  singular eigensolutions, or simply singularities, of a homogeneous elliptic BVP  in an infinite corner (angular sector). See   \cite{Grisvard1985, KufnerSandig1987, Dauge1988, Nicaise1993,Grisvard1992, NazarovPlamenevsky1994, KozlovMazyaRossmann1997, KozlovMazyaRossmann2001} for a large number of  relevant mathematical results regarding elliptic BVPs with singularities on the domain boundary.
In general, the asymptotic solution of an elliptic BVP, denoted as  $u$, near a corner  can be decomposed into a singular part, denoted as $u_\text{S}$, given by a linear combination of a few most singular eigensolutions (often taking just one term), and a regular part, denoted as $u_\text{R}$, given by an infinite series of more regular (i.e. less singular) eigensolutions. These functions are harmonic and near the corner, the solution exhibits reduced regularity compared to the interior and belongs to certain fractional Sobolev spaces; see~\cite{Grisvard1985,Grisvard1992,Mghazli1992} for further details. 
One of the key issues in the analysis of singularities in elliptic BVPs and their applications is the derivation of the closed-form expressions of these singular eigensolutions, and then an accurate calculation of the coefficients that multiply the most singular eigensolutions included in $u_\text{S}$, usually referred to as flux intensity factors in scalar elliptic BVPs and stress intensity factors in elastic BVPs. These flux or stress intensity factors often represent key magnitudes for applications, e.g., in fracture mechanics, since they appear in fracture criteria for crack formation and propagation.


%

The analysis of the asymptotic singular behavior of solutions of elliptic BVPs near  corner  points on the boundary is of great importance for many engineering applications, e.g., fracture and contact mechanics, fluid mechanics and heat transfer, see  \cite{Leguillon1987,Yosibash2012} where many engineering problems in domains with corner singularities on the boundary were studied.

Numerical analysis of elliptic BVPs with boundary singularities is very relevant for applications because the presence of such singularities in the solution of a BVP can significantly reduce the convergence rate of the numerical solution, e.g., by Finite Element Method (FEM), as shown in the pioneering works by Babu\v{s}ka \cite{Babuska1970} and Strang and Fix \cite{StrangFix1973}, and reviewed in a very comprehensive way in \cite{SzaboBabuska1991,BrennerCarstensen2004}. A large number of methods have been proposed to deal with this difficulty and to achieve an optimal convergence rate despite the presence of a singularity in the BVP solution. Many of these approaches are based on knowledge of the structure of the singular part $u_\text{S}$, especially the closed form of the most singular eigensolutions. 

Recently,   BVPs with the so-called Robin boundary conditions  have attracted  a considerable attention in mathematical literature, see \cite{SayasVariational, Medkova2018} for comprehensive studies of such BVPs and their variational formulations, because this type of boundary condition is very relevant for many applications.    Therefore,   the present work deals with derivation of such singular eigensolutions for  elliptic BVPs with the  Robin boundary condition.  Note that, the nomenclature regarding Robin boundary condition is not unique, and different names are used for this type of boundary condition, mainly depending on the application, e.g., Fourier or Kapitza boundary condition or thermal boundary resistance condition are  used in heat transfer, impedance boundary condition  in electromagnetism and acoustic, and Winkler spring boundary condition in elasticity and structural analysis. 

Although the literature on singularities in corners with Robin boundary condition is rather scarce, several relevant results have been obtained in the last decades, showing that the problem of corner singularities with this boundary condition is actually very tricky due to a  linear combination of the solution $u$ and its first derivatives. This leads to a peculiar and much more complex structure of each singular eigensolution compared to corner singularities with standard Dirichlet and Neumann boundary conditions.

Mghazli~\cite{Mghazli1992} analysed the regularity of harmonic solutions in plane corners with nonhomogeneous Dirichlet-Robin (D-R) and Robin-Robin (R-R) boundary conditions and deduced rather complicated closed form expressions for singular eigenfunctions  for constant coefficients in the Robin condition.
Costabel and Dauge~\cite{CostabelDauge1996} analysed the asymptotic behaviour of the harmonic solutions near the singular point of change of the homogeneous Robin to the Neumann boundary condition on a smooth boundary, characterising the change of the nature of the local solution behaviour, from the logarithmic to the square root singularity,  when the coefficient multiplying the normal derivative vanishes.

Sinclair~\cite{Sinclair1996,Sinclair1999,Sinclair2004,Sinclair2009,Sinclair2015}   applied  Williams~\cite{Williams1952} approach and a recursive procedure to derive expressions of the double asymptotic series,  including the main and so-called shadow terms, for several relevant particular cases of corners with Robin boundary conditions,  for both the Laplace and plane-elastic BVPs. 
Mishuris~\cite{Mishuris1999,Mishuris2001,MishurisKuhn2001} applied the Mellin transform to characterise singular eigensolutions in some specific Laplace and plane-elastic BVPs with Neumann-Robin boundary conditions with power-law variation of the coefficient, corresponding to a crack either perpendicular to or located along a straight Robin-type interface.
Antipov et al.~\cite{Antipov2001} obtained exact solutions of Laplace equation for some  problems of a semi-infinite crack along a straight Robin  interface, deducing asymptotic behavior of the solutions at the crack tip. 
Lenci~\cite{Lenci2001} derived, for an elastic BVP in an infinite plane, an integral equation for a finite crack on a Robin interface, showing the presence logarithmic singularity at the crack tip observed also in~\cite{Antipov2001}.

Ueda et al.~\cite{Ueda2006} derived recursive formulas for a series representation of the singular eigensolutions of the Laplace equation with Dirichlet-Robin boundary condition on a straight boundary with the coefficient variation given by a linear combination of a constant term and a term  inversely proportional to the distance from the singular point. 
These authors have generalized their approach in~\cite{Watanabe2007} to derive recursive formulae for a series representation of the singular eigensolutions of the Laplace and elastic BVPs considering two dissimilar materials joined by a straight interface, one part of which is a perfect interface and the other part is of the Robin type, with the coefficient variation given by a linear combination of a constant term and a term inversely proportional to the distance from the singular point where these parts meet.
Wu~\cite{Wu2017} derived an exact solution using conformal mapping technique  of a BVP for the Laplace equation in a half-plane with Robin boundary condition on a finite segment  with the coefficient  inversely proportional to the square root of the distance from the segment end-points,  and the homogeneous Dirichlet boundary condition on the rest of the boundary.

Jim\'enez-Alfaro et al.~\cite{Jimenez-Alfaro2020} derived expressions for singular eigensolutions of homogeneous Dirichlet-Robin and Neumann-Robin BVPs for the Laplace equation in infinite corners in the form of asymptotic series given by a main term and a finite or infinite series of the so-called shadow terms. These power-logarithmic shadow terms are defined by solving recursive linear systems. In~\cite{Jimenez-Alfaro2023} this approach was applied to an especially relevant Neumann-Robin BVP in the half-plane domain, leading to the development of a new crack-tip finite element for logarithmic singularities to model cracks propagating along   Robin interfaces in~\cite{Mantic2024}.

In the present work, we address the derivation of  singular eigensolutions of the Laplace equation in an infinite corner domain (angular sector)   
  with mixed homogeneous boundary conditions: Dirichlet  boundary condition on one side, and Robin  boundary condition with power law variation of the coefficient on the other side of the corner. 
Inspired mainly by the approaches developed in~\cite{Mghazli1992,CostabelDauge1996} and especially that in~\cite{Jimenez-Alfaro2020}, we propose two original and general recursive procedures to derive expressions for the singular eigensolutions  in the form of finite or infinite asymptotic series.  In these recursive procedures, the Robin boundary condition is replaced by non-homogeneous recursive Neumann  or Dirichlet boundary conditions. An advantage of these recursive procedures is that they are very suitable for computational implementation in computer algebra software, which is substantial, as the complexity of the shadow terms is in general increasing with the shadow term order.
A complete study of these series is presented, proposing conditions under which the series is finite or infinite, depending on the values of the problem parameters: the inner angle of the corner and the exponent in power-law variation of the coefficient in the Robin boundary condition. The energy of the singular eigensolution is analyzed, and also the error of a truncated asymptotic series. 
The main results are summarized in  tables, with the mathematical arguments given in the appendices. Then,   several representative examples of the singular eigensolutions represented by the asymptotic series  are studied, showing the behavior of the eigensolution, its  derivatives and errors, and analyzing the solution energy. 
Finally, the results of the present study are comprehensively presented for the case of the half-plane domain, because it actually models a crack bridged by a linear elastic spring distribution with power-law variation of the spring stiffness, which could have relevant applications in computational fracture mechanics.

For the sake of brevity, this article does not address the similar case of homogeneous Neumann-Robin boundary conditions  with power-law variation of the coefficient, which will be a topic of a forthcoming work.

The article is organized as follows. In Section 2, the D-R corner problem is introduced, observing certain singularities at the corner tip, for   some power-law variations of the coefficient in the Robin boundary condition. Then, in Section 3, the asymptotic series expansions for singular eigensolutions in these corners are deduced using a Dirichlet-Neumann (D-N) approach and a Dirichlet-Dirichlet (D-D) approach, considering a D-N homogeneous problem for main terms and non-homogeneous D-N for shadow terms, and performing an analogous procedure for the D-D approach. In addition, we analyze the error of these procedures, their convergence as series finite or infinite, and the energy of the system, which is summarized in tables.  Within this section, we study a special case of the power-law variation  of the coefficient in the Robin boundary condition for which the previous approaches do not converge.  
In Section 4, several representative examples of the application of both approaches are analyzed, together with a special case.  Finally, an application of the present results in modeling bridged cracks in fracture mechanics is briefly presented in Section 5. In the Appendix, we introduce some proofs about our methods, and we deduce an explicit recursive method for the calculation of the coefficient of the asymptotic series.	

\section{Definition of the Dirichlet-Robin corner problem}\label{sec:DR_definition}

The original motivation of the present work is the elastic anti-plane strain  problem in an infinite corner with a power-law variation of spring stiffness on one corner side and restrained movement on the other side. The out-of-plane displacement function $u=u_z(x,y)$  solves the Laplace equation in a corner domain $\Omega$  defined in polar coordinates  $(r,\theta)$ as
  \begin{equation}
   \Omega =  \{(r\cos{\theta}, r\sin{\theta}), r>0, 0 <\theta< \omega\}, \label{Omegadef}  
  \end{equation}
 where $0<\omega \leq 2\pi$ is  the inner angle of the corner. Let  the (open) boundary parts be defined by radial lines in polar coordinates as $\Gamma_{1} = \{(r,0), r>0 \}$,  with $\theta = 0$, and $\Gamma_{2} = \{(r\cos(\omega),r\sin(\omega)), r>0 \}$, with $\theta =\omega$.  

Then, this elastic BVP can be defined as follows, see  Fig. \ref{fig:EsquemaDR} for a scheme,
\begin{equation}\label{DR1}
    \begin{aligned}
        \Delta u & = 0, \quad  & \text{ in } &\Omega, \\
		u & = 0, \quad & \text{ on } & \Gamma_{1},   \\
		 \sigma_{\theta z} + K(r) u & =  0, \quad & \text{ on }  & \Gamma_{2},  
    \end{aligned}
\end{equation}
where the function 
$K(r) = K_0 \left(\frac{r}{a}\right)^{\alpha}$ 
describes the spring stiffness variation with  $r$, where the parameter $K_0$ is a reference stiffness value, $a$ is a reference length, and $\alpha\in\mathbb{R}$ is the exponent in the power-law variation of spring stiffness. The shear stress $\sigma_{\theta z}=\frac{G}{r}\frac{\partial u}{\partial \theta}$, where $G$ is the shear modulus of the bulk. 

The above elastic problem can be rewritten as the Dirichlet-Robin (D-R) BVP for the Laplace equation in an infinite corner domain (angular sector) $\Omega\subset \R^{2}$,  
  with mixed homogeneous boundary conditions: Dirichlet  boundary condition on one side, and Robin  boundary condition with power law variation of the coefficient on the other side of the corner.  It is convenient to prescribe the Dirichlet boundary condition  at $\theta=0$, for the sake of simplicity of the following derivations.
 
 Then, the singular eigensolutions are defined  as solutions of the following BVP in  $\Omega$:  
\begin{align}\label{DR}
	\Delta u &= 0 \quad \text{ in } \Omega, \nonumber \\
	u & = 0 \quad \text{ on } \Gamma_{1},    \\
	\frac{1}{r} \frac{\partial u }{\partial \theta} +  \gamma\, r^\alpha u & =  0 \quad \text{ on }  \Gamma_{2}, \nonumber  
\end{align} 
where $\alpha, \gamma\in \R$   are given real constants, with $\gamma = \frac{K_0}{G a^\alpha}>0$.
Notably, there is no remote boundary condition prescribed at infinity in the above problem, leading to non-uniqueness in its solution, and resulting in an infinite sequence of  singular eigensolutions.
%

If $\alpha>-1$ the Robin boundary condition in \eqref{DR} can be expressed as 
\begin{equation}
    \frac{\partial u}{\partial \theta} + \gamma r^{\alpha+1} u = 0 \quad \text{ on }  \Gamma_{2}.  \label{eq:BR-DN}
\end{equation}
whereas if $\alpha<-1$, the condition \eqref{eq:BR-DN} leads to a singularity for $r\rightarrow 0$, which is avoided by expressing the Robin boundary condition as 
\begin{equation}
    \frac{1}{\gamma r^{\alpha+1}} \frac{\partial u}{\partial \theta} + u = 0 \quad \text{ on }  \Gamma_{2}.  \label{eq:BR-DD}
\end{equation}
A particular case is $\alpha=-1$, for which a closed-form solution is obtained, as will be shown in the following sections.

\begin{figure}[H]
	\centering
    \includegraphics[scale=0.25]{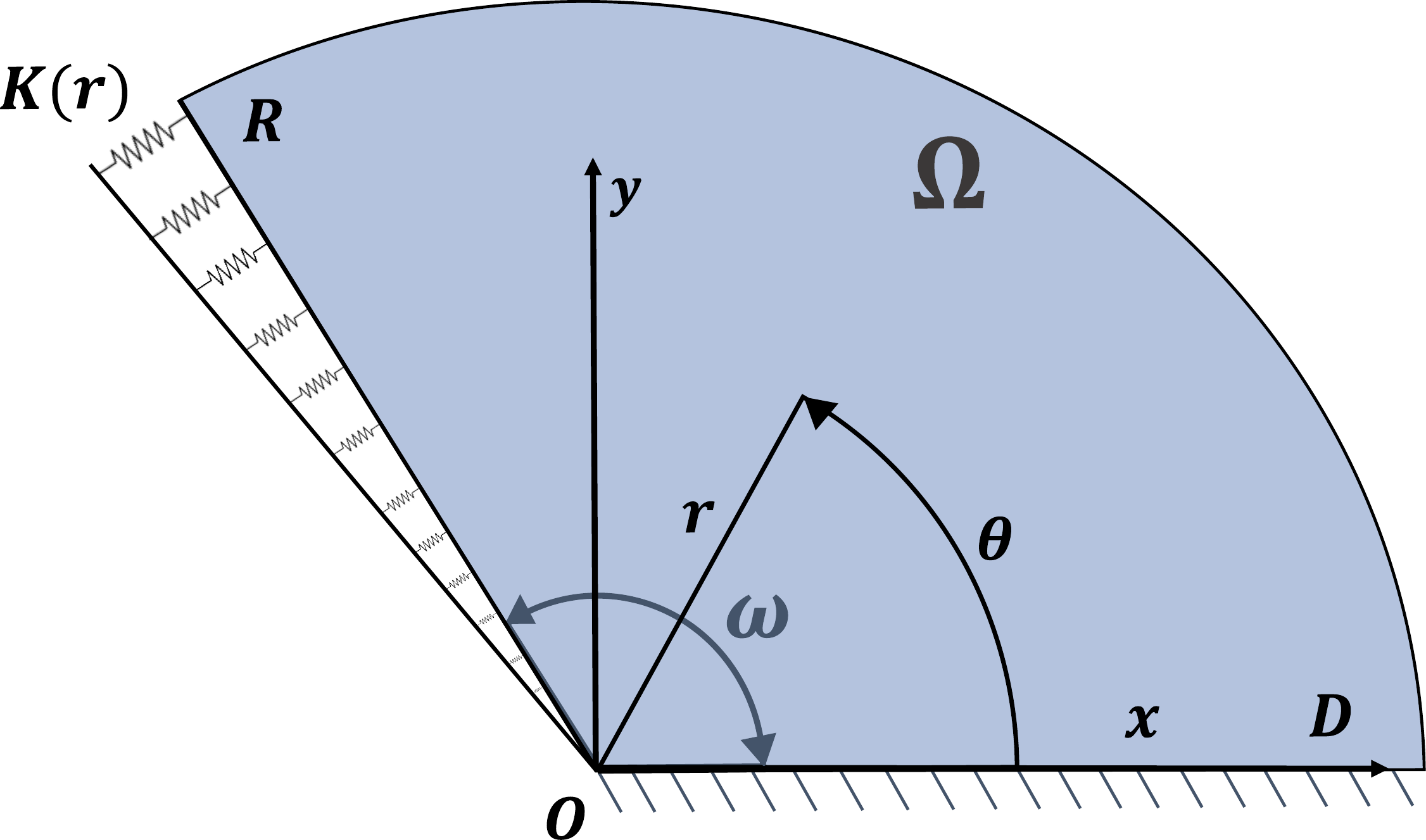}
	\caption{Scheme of a corner elastic problem with spring distribution of varying stiffness with $\alpha=-1$, where $D = \Gamma_{1}$ and $R = \Gamma_{2}$ denote the Dirichlet and Robin boundaries, respectively.}
	\label{fig:EsquemaDR}
\end{figure}


\section{Asymptotic series expansion for singular eigensolutions in Dirichlet-Robin corner problems}
\subsection{Corner BVPs}

A  $j$-th singular eigensolution $u_j$, for  $j\in \N$, of a   D-R    corner problem is given by the sum of a main term $u_j^{(0)}$ and a series of $S_{j}$ shadow terms  $u_{j}^{(k)}$,  for $k \in\{1,\ldots, S_{j}\}$, see  \cite{Mghazli1992,CostabelDauge1996,Jimenez-Alfaro2020},
\begin{equation}\label{eq:asympseries}
 u_j(r,\theta) = u_{j}^{(0)}(r,\theta) +  \sum_{k=1}^{S_{j}}u_{j}^{(k)}(r,\theta) . 
\end{equation}
%


The main term $u_{j}^{(0)}$ is obtained by solving the homogeneous Dirichlet-Neumann (D-N) problem obtained from \eqref{DR} by replacing the Robin boundary condition~\eqref{eq:BR-DN} by a homogeneous Neumann boundary condition. The eigensolutions of the D-N corner problem are easily obtained by the method of separation of variables   \cite{Yosibash2012}  giving 
\begin{equation}
	u_{j}^{(0)}(r,\theta) = r^{\lambda_{j}}\sin(\lambda_{j}\theta ),   \qquad\text{with } \lambda_{j} = (2j-1)\frac{\pi}{2\omega}, \quad j\in \N. 
\end{equation}
%
The shadow terms $u_{j}^{(k)}$, for $k\in\{1,2,\ldots,S_{j}\}$, are obtained by solving the following recursive non-homogeneous Dirichlet-Neumann (D-N) corner  problems
\begin{equation}
	\label{DN}
\begin{aligned}
\Delta u_{j}^{(k)} & = 0, & \qquad \text{in } & \Omega   \\
u_{j}^{(k)} & = 0, & \qquad \text{on } & \Gamma_{1} \\
\frac{1}{r}\frac{\partial u_{j}^{(k)}}{\partial \theta} & = -\gamma r^{\alpha}u_{j}^{(k-1)}, & \qquad \text{on } & \Gamma_{2},   
\end{aligned}
\end{equation}
where the condition on $\Gamma_{2}$ is a non-homogeneous Neumann condition, in which the right-hand side is defined either by the main term $u_{j}^{(0)}$, if $k =  1$, or by the previous shadow term $u_{j}^{(k-1)}$, if $k > 1$.

As will be seen, in the case of $\alpha > -1$, this error decreases in a neighborhood of the corner tip as we add shadow terms. However, in the case of $\alpha \leq -1$, it does not always converge; it depends on the inner angle of the corner $\omega$   and the power $\alpha$. 

To avoid this difficulty for $\alpha \leq -1$, the main term $u_{j}^{(0)}$ is obtained by solving the homogeneous Dirichlet-Dirichlet (D-D) problem obtained from \eqref{DR} by replacing the Robin boundary condition in the form \eqref{eq:BR-DD} by a homogeneous Dirichlet  boundary condition. The eigensolutions of the homogeneous D-D corner problem are obtained by the method of separation of variables%
\begin{equation}%
 u_{j}^{(0)}(r,\theta) = r^{\lambda_{j}}\sin(\lambda_{j}\theta), \quad \text{with } \lambda_{j} = j\frac{\pi}{\omega}, \quad j\in \N.   
\end{equation}%
Then, the shadow terms $u_{j}^{(k)}$, for $k\in\{1,2,\ldots,S_{j}\}$, are obtained by solving the following recursive non-homogeneous Dirichlet-Dirichlet (D-D) corner  problems, obtained from \eqref{DR} by  rewriting the Robin boundary condition in   the form \eqref{eq:BR-DD} 
\begin{equation}\label{DD}
    \begin{aligned}
    \Delta u_{j}^{(k)} & = 0, & \qquad \text{in } & \Omega\\
    u_{j}^{(k)} & = 0, & \qquad \text{on } & \Gamma_{1}\\
    \gamma r^{\alpha} u_{j}^{(k)}  & =  -\frac{1}{r}\frac{\partial u_{j}^{(k-1)}}{\partial \theta}, & \qquad \text{on } & \Gamma_{2}.
    \end{aligned}
\end{equation}
With this   recursive approach, similar to the previous one, the solution converges for the case $\alpha < -1$; but for the case $\alpha \geq -1$, it does not always converge and again depends on the values of $\omega$ and $\alpha$. 

It is important to highlight that expressions $u_{j}^{(0)}$ are valid for all $j\in \Z$. However, for $j\leq 0$ the   solution has infinite   energy in the neighborhood of the corner tip; therefore, we focus on $j\in \N$.

\begin{rmk}\label{RmkErr}
	It is easy to see that the main terms do not fulfill the Robin boundary condition. Thus, the shadow terms are added to diminish the error in this boundary condition on $\Gamma_{2}$.  
\end{rmk}

In summary, we have two recursive procedures to find the shadow terms given by equations~\eqref{DN} and~\eqref{DD} for  $\alpha > -1$  and $\alpha < -1$, see Tables~\ref{TblDN1} and~\ref{TblDD1}, respectively.  Nevertheless, there are also cases such that the convergence and finite energy hold for $\alpha < -1$ using D-N approach, and for $\alpha > - 1$ using D-D approach, see Tables~\ref{TblDN2} and~\ref{TblDD2}, respectively. Similar recursive procedures were previously developed in~\cite{Mghazli1992,CostabelDauge1996,Sinclair2004,Sinclair2015,Jimenez-Alfaro2020}. In the next sections, we will study both procedures with $\alpha \in \R$ analyzing the convergence and the local energy of the system, which will give us criteria under which to use one or another recursive procedure. But before that, we are going to study a special case in which neither of the two procedures converges.

\subsection{Special case \texorpdfstring{$\alpha = -1$}{a = -1} \label{sec:case_alpha-1}} 

As mentioned in Section~\ref{sec:DR_definition}, a particular and critical case is $\alpha = -1$, since the recursive procedures based on~\eqref{DN}  and~\eqref{DD} do not converge in this case, which will be deduced from the error analysis. However, we can easily find the exact closed-form expression for the singular eigensolutions by considering
\begin{equation} \label{eq:DR_alpha_1}
u_{j}(r,\theta) = r^{\lambda_{j}} \sin(\lambda_{j} \theta). 
\end{equation}
Since this $u_{j}(r,\theta)$    is a harmonic function and $u_j(r,0)=0$, the only condition to be fulfilled in \eqref{DR} is   the Robin boundary condition, leading to a  transcendental equation for $\lambda_j$,
\begin{equation}\label{trascendental}
	\tan(\lambda_{j}\omega) + \frac{\lambda_{j}}{\gamma} = 0,  \qquad \text{with }(2 j - 1) \frac{\pi}{2\omega} < \lambda_j < j\frac{\pi}{\omega}, \text{for all } j\in \N.
\end{equation}
This equation has infinite roots $\lambda_j$, but only approximate numerical solutions can be found.
It is worth mentioning that only $j\in \N$ are taken to avoid infinite strain energy, as  was explained before.

\subsection{General expressions for the shadow terms}

Taking into account the general results on singular solutions for elliptic BVPs developed in~\cite{Kondratiev1967,CostabelDauge1993}, see also \cite{Grisvard1985, KufnerSandig1987, Dauge1988, Grisvard1992, Nicaise1993, NazarovPlamenevsky1994, KozlovMazyaRossmann1997, KozlovMazyaRossmann2001}, and adapting  
   the     expression for the shadow terms used in \cite{Jimenez-Alfaro2020} for  $\alpha=0$
\begin{equation} \label{ExpBBasica}
	u_{j}^{(k)} (z) = \sum_{l=0}^{L_{j,k}} a_{j,k}^{(l)} \Im \left\{ z^{\lambda_{j} + k} \log^{l}(z)  \right\} + b_{j,k}^{(l)} \Re\left\{ z^{\lambda_{j} + k} \log^{l}(z)  \right\},  
\end{equation}
to the present case $\alpha\in\R$,  we propose the following general expression for the shadow terms  
\begin{equation} \label{ExpBasica}
	u_{j}^{(k)} (z) = \sum_{l=0}^{L_{j,k}} a_{j,k}^{(l)} \Im \left\{ z^{\lambda_{j} \pm k(\alpha + 1)} \log^{l}(z)  \right\} + b_{j,k}^{(l)} \Re\left\{ z^{\lambda_{j} \pm k(\alpha + 1)} \log^{l}(z)  \right\},  
\end{equation}
where $z(r,\theta) = r e^{\is \theta}$ is a complex number, $a_{j,k}^{(l)}$ and $b_{j,k}^{(l)}$ are coefficients to be determined. The symbol $\pm$ denotes `+' and `-', respectively, for the recursive procedures based on \eqref{DN} and~\eqref{DD}. It is convenient that this series also covers the main terms, thus, we  set $L_{j,0} = 0$, $a_{j,0}^{(0)}=1$ and $b_{j,0}^{(0)}= 0$. 

In terms of the polar coordinates, the series~\eqref{ExpBasica} can be rewritten as 
\begin{equation}\label{ExpIm}
	\begin{split}
        u_{j}^{(k)} (r,\theta) = \sum_{l=0}^{L_{j,k}} & a_{j,k}^{(l)}  \Im \left\{ r^{\lambda_{j} \pm k(\alpha + 1)}e^{\is(\lambda_{j} \pm k(\alpha + 1) )\theta } (\log(r) + \is \theta )^{l}\right\} 
		\\ 
		+ \ & b_{j,k}^{(l)} \Re\left\{ r^{\lambda_{j} \pm k(\alpha + 1)} e^{\is(\lambda_{j} \pm k(\alpha + 1))\theta }(\log(r) + \is \theta )^{l}\right\}.
	\end{split} 
\end{equation}
Using binomial expansion  and $\is =  e^{\is\frac{\pi}{2}}$,  we deduce that
\begin{equation}
	\label{ExpRe_a}
	\begin{aligned}
	u_{j}^{(k)} (r,\theta)  = & \sum_{l=0}^{L_{j,k}} \left[  a_{j,k}^{(l)}\ r^{\lambda_{j} \pm k(\alpha + 1)} \sum_{m=0}^{l} \binom{l}{m} \log^{m}(r) \theta^{l-m}  \Im \left\{e^{ \is\left[ (\lambda_{j} \pm k(\alpha + 1))\theta  + \frac{\pi}{2} (l-m)\right] }  \right\}  \right.  \\
	&  + \  \left. b_{j,k}^{(l)} \ r^{\lambda_{j} \pm k(\alpha + 1)} \sum_{m=0}^{l} \binom{l}{m} \log^{m}(r) \theta^{l-m}  \Re \left\{e^{ \is\left[ (\lambda_{j} \pm k(\alpha + 1))\theta  + \frac{\pi}{2} (l-m)\right] }  \right\}  \right]\\ 
	  = & \ r^{\lambda_{j} \pm k(\alpha + 1)}\sum_{l=0}^{L_{j,k}}  \left[ a_{j,k}^{(l)} \sum_{m=0}^{l} \binom{l}{m} \log^{m}(r) \theta^{l-m}  \sin\left( (\lambda_{j} \pm k(\alpha + 1))\theta  + \frac{\pi}{2} (l-m)\right)   \  \right. \\
	 &  +\  \left. b_{j,k}^{(l)} \sum_{m=0}^{l} \binom{l}{m} \log^{m}(r) \theta^{l-m}  \cos\left( (\lambda_{j} \pm k(\alpha + 1))\theta  + \frac{\pi}{2} (l-m)\right) \right].
\end{aligned}	
\end{equation}

Note that the expression in~\eqref{ExpRe_a} takes the form of the solutions proposed in~\cite[Proposition A1, A2 and A3]{Mghazli1992}, for the recursive procedure based on \eqref{DN}, while it  takes the form of the solutions proposed in Propositions~\ref{AppPrp1},~\ref{AppPrp2}, and~\ref{AppPrpn} in the Appendix~\ref{AppC}, for the recursive procedure based on \eqref{DD}.


Evaluating the expression in~\eqref{ExpIm} at  $\theta = 0$ we get
 \[u_{j}^{(k)}(r,0) = \sum_{l=0}^{L_{j,k}} b_{j,k}^{(l)} r^{\lambda_{j} \pm k(\alpha + 1)} \log^l(r). \]
Thus,  it can be concluded that all the coefficients $b_{j,k}^{(l)}$ are zero to verify the Dirichlet boundary condition on $\Gamma_{1}$, i.e., $u_{j}^{(k)}(r,0) = 0$ for all $r>0$.  Then, by interchanging the summation order, we get
\begin{equation}
	\begin{aligned}\label{ExpRe}
	u_{j}^{(k)} (r,\theta)  & =   r^{\lambda_{j} \pm k(\alpha + 1)}\sum_{l=0}^{L_{j,k}}  a_{j,k}^{(l)} \sum_{m=0}^{l} \binom{l}{m} \log^{m}(r) \theta^{l-m}  \sin\left( (\lambda_{j} \pm k(\alpha + 1))\theta  + \frac{\pi}{2} (l-m)\right)\\
	& = r^{\lambda_{j} \pm k(\alpha + 1)} \sum_{m=0}^{L_{j,k}}  \log^{m}(r) \sum_{l=m}^{L_{j,k}} a_{j,k}^{(l)}\binom{l}{m}  \theta^{l-m}  \sin\left( (\lambda_{j} \pm k(\alpha + 1))\theta  + \frac{\pi}{2} (l-m)\right).
\end{aligned}
\end{equation}


\subsection{Shadow terms computed by the recursive procedure based on  D-N BVPs}\label{SubsectCoeff_DN}

In this section we apply the recursive procedure defined by~\eqref{DN}. Considering~\eqref{ExpRe}, the Neumann boundary condition on $\Gamma_{2}$ leads to the following equation
\begin{align*}
&\frac{1}{r}\sum_{l=0}^{L_{j,k}} \frac{\partial}{\partial \theta}\left[ a_{j,k}^{(l)}\ r^{\lambda_{j} + k(\alpha + 1)} \sum_{m=0}^{l} \binom{l}{m} \log^{m}(r) \theta^{l-m}  \sin\left( (\lambda_{j} + k(\alpha + 1))\theta  + \frac{\pi}{2} (l-m)\right)  \right]_{\theta = \omega} \\ 
& =  -\gamma \sum_{l=0}^{L_{j,k-1}} a_{j,k-1}^{(l)} \ r^{\lambda_{j} + k(\alpha + 1) -1} \sum_{m=0}^{l} \binom{l}{m} \log^{m}(r) \omega^{l-m}  \sin\left( (\lambda_{j} + (k-1)(\alpha + 1))\omega  + \frac{\pi}{2} (l-m)\right)    
\end{align*}
or equivalently, interchanging the order of summation and simplifying the expression $r^{\lambda_{j} + k(\alpha + 1) -1}$  it holds 
\begin{align*}
	&\sum_{m=0}^{L_{j,k}} \log^{m}(r) \sum_{l=m}^{L_{j,k}} a_{j,k}^{(l)} \binom{l}{m} \omega^{l-m} \left[ \frac{(l-m)}{\omega} \sin\left( (\lambda_{j} + k(\alpha + 1))\omega + \frac{\pi}{2} (l-m)\right)  \right. \\ 
   & \hspace{4cm} +\  \left. (\lambda_{j}  + k(\alpha + 1))\cos\left( (\lambda_{j} + k(\alpha + 1))\omega + \frac{\pi}{2} (l-m)\right)\right]
   \nonumber\\ 
   & =   -\gamma \sum_{m=0}^{L_{j,k-1}} \log^{m}(r)  \sum_{l=m}^{L_{j,k-1}} a_{j,k-1}^{(l)} \binom{l}{m}  \omega^{l-m}  \sin\left( (\lambda_{j} + (k-1)(\alpha + 1))\omega  + \frac{\pi}{2} (l-m)\right).   
\end{align*}
Using angle sum identities we have 
\begin{equation}\label{SumSystDN}
\begin{split}
    &\sum_{m=0}^{L_{j,k}} \log^{m}(r) \sum_{l=m}^{L_{j,k}} a_{j,k}^{(l)} \binom{l}{m} \omega^{l-m} \left[ \frac{(l-m)}{\omega} \cos\left(  \omega k(\alpha + 1) + \frac{\pi}{2} (l-m)\right)  \right. \\ 
   & \hspace{4cm} - \left. (\lambda_{j}  + k(\alpha + 1))\sin\left( \omega k(\alpha + 1) + \frac{\pi}{2} (l-m)\right)\right] 
   \\ 
   &=  -\gamma \sum_{m=0}^{L_{j,k-1}} \log^{m}(r)  \sum_{l=m}^{L_{j,k-1}} a_{j,k-1}^{(l)} \binom{l}{m}  \omega^{l-m}  \cos\left( \omega (k-1)(\alpha + 1)  + \frac{\pi}{2} (l-m)\right).   
\end{split}
\end{equation}
Since, the previous polynomial expression should be verified for all $r>0$, coefficients $a_{j,k}^{(l)}$ are determined, obtaining one equation for each power of logarithmic terms. Hence, a system of $L_{j,k}+1$ equations is obtained,  
\begin{equation}\label{SystDN}
\bm{M}_{j,k} \bm{a}_{j,k} = \bm{g}_{j,k-1},
\end{equation}
where
\begin{align*}
	\bm{a}_{j,k} &= \left[ a_{j,k}^{(0)}, a_{j,k}^{(1)}, \ldots, a_{j,k}^{(L_{j,k})}  \right] ^\top \\
	\bm{g}_{j,k-1} &= \left[ g_{j,k-1}^{(m)} \right]^\top_{m=0,\ldots, L_{j,k-1}}\\
	\bm{M}_{j,k} &= \left[ \mu_{j,k}^{(m,l)}  \right]_{m = 0, \ldots, L_{j,k};\ l = m,\ldots, L_{j,k} }	 = 
	\begin{pmatrix} 
		\mu_{j,k}^{(0,0)} & \mu_{j,k}^{(0,1)}  & \cdots & \mu_{j,k}^{(0,L_{j,k})}  \\[1.5mm] 
		        0         & \mu_{j,k}^{(1,1)}  & \cdots & \mu_{j,k}^{(1,L_{j,k})}  \\[1.5mm] 
		        \vdots    &  \ddots            & \ddots & \vdots                   \\[1.5mm] 
				0		  &   0                & \cdots & \mu_{j,k}^{(L_{j,k},L_{j,k})}
	\end{pmatrix}
\end{align*}
with  
\begin{align*}
	g_{j,k-1}^{(m)}  & =-\gamma \sum_{l=m}^{L_{j,k-1}} a_{j,k-1}^{(l)} \binom{l}{m} \omega^{l-m} \cos\left( \omega (k-1)(\alpha + 1)  + \frac{\pi}{2} (l-m)\right), \\ 
	\mu_{j,k}^{(m,l)} & =  \binom{l}{m} \omega^{l-m} \left[ \frac{(l-m)}{\omega} \cos\left( \omega k(\alpha + 1) + \frac{\pi}{2} (l-m)\right) \ - \right. \\ 
	& \hspace{2cm} \left. (\lambda_{j}  + k(\alpha + 1))\sin\left( \omega k(\alpha + 1) + \frac{\pi}{2} (l-m)\right)\right]. 
\end{align*}

In Appendix~\ref{AppB}, we detail how the above set of equations can be recursively reformulated as the system~\eqref{RrcSystDN} which depends on the solution in the previous step \(\bm{a}_{j,k-1}\), facilitating its computational implementation.

The main diagonal terms of the upper triangular matrix $\bm{M}_{j,k}$ are given by
\[
\mu_{j,k}^{(m,m)} = -(\lambda_{j} + k(\alpha + 1))\sin\left( \omega k(\alpha + 1) \right).
\]
Thus, the system becomes inconsistent if
\begin{equation}\label{CondDN_1}
	\lambda_{j} + k(\alpha + 1) = 0
\end{equation}
or
\begin{equation}\label{CondDN_2}
	\sin\left( \omega k(\alpha + 1) \right) = 0,
\end{equation}
as these conditions result in zero diagonal entries. To ensure system consistency, we increment the system size by one, i.e., $L_{j,k} = L_{j,k-1} + 1$. The resulting system, comprising $L_{j,k-1} + 2$ equations, retains a matrix $\bm{M}_{j,k}$ with a first column and last row of zeros. This system is consistent because the last element of $\bm{g}_{j,k-1}$, denoted $g_{j,k-1}^{(L_{j,k-1}+1)}$, is zero. The conditions~\eqref{CondDN_1} and~\eqref{CondDN_2} are fulfilled for  
\begin{equation}\label{CondDNResumen}
	(\alpha + 1)\frac{\omega}{\pi} = - \frac{2j-1}{2k} \quad \text{ or }\quad(\alpha + 1) \frac{\omega}{\pi} = \frac{p}{k},\quad\text{ with } j,k \in \N,\ p\in \Z. 
\end{equation}
When the system is augmented, the coefficient $a_{j,k}^{(0)}$ is undefined. However, $a_{j,k}^{(0)}$ is a coefficient of purely power type singularity, with no logarithmic term. Therefore, its associated term is included in another shifted main term. Thus, we can set $a_{j,k}^{(0)} = 0$ and the asymptotic expansion will keep complete.  Later in   Section~\ref{CharAsympSeries} we will analyse the    conditions \eqref{CondDNResumen} in detail.

\begin{rmk}\label{Rmk1}
	If $\alpha > -1$, then $\lambda_{j}  + k(\alpha + 1) > 0$ for all $\omega\in(0,2\pi], k\in\N$ and $j\in \N$ in this case, and the series always converges in a neighborhood of zero. For the case $\alpha < -1$ the series does not always converge. In addition, if $\alpha=-1$,  it can be deduced from~\eqref{CondDN_2} that in every step $k$ we increase the system size; however, the series does not converge, as we will see in the next sections. The conditions~\eqref{CondDN_1} and~\eqref{CondDN_2} are not fulfilled at the same time, see Proposition~\ref{PrpB1} in Appendix~\ref{AppA}, this means that  the system is augmented only due to one condition for a shadow term.  
\end{rmk}
\subsection{Shadow terms computed by the recursive procedure based on D-D BVPs}\label{SubsectCoeff_DD}

In this section we apply the recursive procedure defined by~\eqref{DD}. To distinguish between the two recursive procedures, we use the \textit{Sans-serif} typography for the present case (i.e. $\bm{\as}_{j,k}$,  $\as_{j,k}^{(l)}$,   $\bm{\Ms}_{j,k}$, $\ms_{j,k}^{(m,l)}$, $\bm{\gs}_{j,k}$, $\gs_{j,k}^{(m)}$). The procedure is similar to the previous one, using the negative sign in~\eqref{ExpRe} and interchanging the non-homogeneous Neumann boundary condition by a non-homogeneous Dirichlet boundary condition on $\Gamma_{2}$, which results in the equation 
\begin{equation}\label{SumSystDD}
\begin{split}
    & \sum_{m=0}^{L_{j,k}} \log^{m}(r) \sum_{l=m}^{L_{j,k}} \as_{j,k}^{(l)} \binom{l}{m} \omega^{l-m} \sin\left(\frac{\pi}{2} (l-m) - k\omega (\alpha + 1)  \right) = \\ 
    -\frac{1}{\gamma} & \sum_{m=0}^{L_{j,k-1}} \log^{m}(r)  \sum_{l=m}^{L_{j,k-1}} \as_{j,k-1}^{(l)} \binom{l}{m}  \omega^{l-m}  \left[ \frac{(l-m)}{\omega} \sin\left(  \frac{\pi}{2} (l-m) - (k-1)\omega (\alpha + 1) \right) \ + \right. \\ 
   & \hspace{5cm} \left. (\lambda_{j}  - (k-1)(\alpha + 1))\cos\left( \frac{\pi}{2} (l-m) - \omega (k-1)(\alpha + 1)  \right)\right],
  \end{split}
\end{equation}
which leads to a system of $L_{j,k}+1$ equations that can be expressed as
\begin{equation}\label{SystDD}
\bm{\Ms}_{j,k} \bm{\as}_{j,k} = \bm{\gs}_{j,k-1},
\end{equation}
where
\begin{align*}
	\bm{\as}_{j,k} &= \left[ \as_{j,k}^{(0)}, \as_{j,k}^{(1)}, \ldots, \as_{j,k}^{(L_{j,k})}  \right] ^\top, \\
	\bm{\gs}_{j,k-1} &= \left[ \gs_{j,k-1}^{(m)}\right]_{m=0,\ldots,L_{j,k-1}}^\top, \\
	\bm{\Ms}_{j,k} &= \left[ \ms_{j,k}^{(m,l)}  \right]_{m = 0, \ldots, L_{j,k};\ l = m,\ldots, L_{j,k} }	 = 
	\begin{pmatrix} 
		\ms_{j,k}^{(0,0)} & \ms_{j,k}^{(0,1)}  & \cdots & \ms_{j,k}^{(0,L_{j,k})}  \\[1.5mm] 
		        0         & \ms_{j,k}^{(1,1)}  & \cdots & \ms_{j,k}^{(1,L_{j,k})}  \\[1.5mm] 
		        \vdots    &  \ddots            & \ddots & \vdots                   \\[1.5mm] 
				0		  &   0                & \cdots & \ms_{j,k}^{(L_{j,k},L_{j,k})}
	\end{pmatrix},
\end{align*}
with  
\begin{align*}\label{DfSystDD}
	\ms_{j,k}^{(m,l)}  & = \binom{l}{m} \omega^{l-m} \sin\left(\frac{\pi}{2} (l-m) - k\omega (\alpha + 1)  \right), \\  
	\gs_{j,k-1}^{(m)}    & = -\frac{1}{\gamma} \sum_{l=m}^{L_{j,k-1}} \as_{j,k-1}^{(l)} \binom{l}{m}  \omega^{l-m}  \left[ \frac{(l-m)}{\omega} \sin\left(  \frac{\pi}{2} (l-m) - (k-1)\omega (\alpha + 1) \right) \ + \right. \\ 
   & \hspace{4cm} \left. (\lambda_{j}  - (k-1)(\alpha + 1))\cos\left( \frac{\pi}{2} (l-m) - \omega (k-1)(\alpha + 1)  \right)\right]. 
\end{align*}
Notice that terms on the main diagonal of the upper triangular matrix $\bm{\Ms}_{j,k}$ are given by 
\[\ms_{j,k}^{(m,m)} = \sin\left( -k \omega (\alpha + 1) \right).\] Thus, when  
\begin{equation}\label{CondDD}
	\sin\left( k \omega (\alpha + 1) \right) = 0,
\end{equation}
the system is inconsistent. Similar to the D-N approach, we increase $L_{j,k}$  by one, i.e., $L_{j,k} = L_{j,k-1} + 1$ to obtain a consistent system (see also the system~\eqref{RrcSystDD}, where we reformulate the previous system as a recursive one).
Conditions in~\eqref{CondDD} are fulfilled when  
\begin{equation}\label{CondDDResumen}
	(\alpha + 1) \frac{\omega}{\pi} = \frac{p}{k},\quad\text{ with } j,k \in \N,\ p\in \Z. 
\end{equation}
If we augment  the system, then the coefficient $\as_{j,k}^{(0)}$ is undefined. However, $\as_{j,k}^{(0)}$  is related to another shifted main term, as explained in the previous section, and it can be set $\as_{j,k}^{(0)} = 0$, without losing the completeness of the asymptotic series. The   condition~\eqref{CondDDResumen} is analysed in detail in Section~\ref{CharAsympSeries}.

\begin{rmk}\label{Rmk2}
	If $\alpha > -1$, then the series does not always converge in a neighborhood of   zero. In addition, if $\alpha=-1$, then from~\eqref{CondDD} we deduce that in every step $k$ we increase the system, however the series is not convergent.  
\end{rmk}

\subsection{Error in the Robin boundary condition}\label{Errores}

As mentioned in Remark~\ref{RmkErr}, the main terms do not satisfy the Robin boundary condition, that is, the error in the recursive procedures is due to the main term on the Robin boundary. Shadow terms are harmonic functions that satisfy the homogeneous Dirichlet boundary conditions on $\Gamma_1$ and either the non-homogeneous Neumann or Dirichlet boundary conditions on $\Gamma_2$.  The characteristic recursive expressions for shadow terms on $\Gamma_2$ can be found in~\cite[Propositions A1, A2 and A3]{Mghazli1992} for D-N approach and propositions~\ref{AppPrp1},~\ref{AppPrp2} and~\ref{AppPrpn} in Appendix~\ref{AppC}, for D-D approach. These expressions allow us to show that the error in the Robin boundary caused by the main term diminishes or cancels. Therefore, we only need to measure the error on the Robin boundary.  

Based on the asymptotic series proposed in~\eqref{eq:asympseries}  for $u_{j}(r,\theta)$,
we define the absolute error for the D-N approach defined by~\eqref{DN} as  
\begin{equation}\label{ErrAbsDN}
   E_{DN}(r):=\frac{1}{r}\frac{\partial u_{j} (r,\omega) }{\partial \theta } +  \gamma r^{\alpha} u_{j}(r,\omega),
\end{equation}
whereas the relative error as
\begin{equation}\label{ErrRltDN}
	e_{DN}(r):= \frac{- \frac{1}{r}\frac{\partial u_{j}(r,\omega)}{\partial \theta }}{ \gamma r^{\alpha} u_{j}(r,\omega)} - 1.
\end{equation}
From the definition of the absolute error, thanks to recursive BVPs~\eqref{DN} and the solution  expression in~\eqref{ExpRe}, it holds
\begin{equation}\label{ErrCharDN}
\begin{split}
	\frac{1}{r}\frac{\partial u_{j}(r,\omega)}{\partial \theta } +  \gamma r^{\alpha} u_{j}(r,\omega)  &=
	\frac{1}{r}\frac{\partial u_{j}^{(0)}(r,\omega)}{\partial \theta } +  \gamma r^{\alpha} u_{j}^{(0)}(r,\omega)  +   \frac{1}{r}\sum_{k=1}^{S_{j}} \frac{\partial u_{j}^{(k)}(r,\omega)}{\partial \theta } +  \gamma r^{\alpha} \sum_{k=1}^{S_{j}}u_{j}^{(k)}(r,\omega),\\
	&=  \gamma r^{\alpha} u_{j}^{(0)}(r,\omega)  -   \gamma r^{\alpha}\sum_{k=1}^{S_{j}} u_{j}^{(k-1)}(r,\omega)
	+  \gamma r^{\alpha} \sum_{k=1}^{S_{j}}u_{j}^{(k)}(r,\omega),\\
	&= \gamma r^{\alpha} u_{j}^{(S_{j})}(r,\omega)\\	
	& =\gamma r^{\lambda_{j} + S_{j}(\alpha + 1) + \alpha} v_{j}^{(S_{j})}(r,\omega),
\end{split}
\end{equation}
where \[ v_{j}^{(k)}(r,\theta) := \sum_{m=0}^{L_{j,k}} \log^{m}(r) \sum_{l=m}^{L_{j,k}}  a_{j,k}^{(l)}\binom{l}{m} \theta^{l-m} \sin\left( (\lambda_{j} + k(\alpha + 1))\theta  + \frac{\pi}{2} (l-m)\right),\]  
which allows us to distinguish between the powers of $r$ and logarithmic polynomial in $r$ for a better error analysis. Hence,
\begin{equation}\label{E_DN}
E_{DN}(r) =\gamma r^{\lambda_{j} + S_{j}(\alpha + 1) + \alpha } v_{j}^{(S_{j})}(r,\omega),   
\end{equation}
with $\lambda_{j} = (2j-1)\frac{\pi}{2\omega}$. From \eqref{E_DN}  it is easy to conclude for $r>0$ that 
\begin{itemize}
	\item If $\alpha > -1$ and $S_{j}$ tends to infinity, then $E_{DN}(r)$ converges to zero as $r$ tends to zero. 
	\item If $\alpha < -1$ and $S_{j}$ tends to infinity, then $E_{DN}(r)$ diverges as $r$ tends to zero.
	\item If $\alpha = -1$ and $j > \frac{\omega}{\pi} + \frac{1}{2}$, then $S_{j}$ tends to infinity and $E_{DN}(r)$ converges to zero as $r$ tends to zero.  
\end{itemize}

It is easy to calculate that
\begin{equation}
	e_{DN}(r) = \frac{ - E_{DN}(r)}{\gamma r^{\alpha} u_{j}(r,\omega)}.
\end{equation}
Thus, for the case of the D-N approach we have
\begin{align}\label{errCharDN}
	e_{DN}(r)   & =  \frac{-u_{j}^{(S_{j})}(r,\omega)}{\sum_{k=0}^{S_{j}} u_{j}^{(k)} (r,\omega) } 
	=   \frac{-r^{\lambda_{j}+ S_{j}(\alpha +1)}  v_{j}^{(S_{j})}(r,\omega)}{\sum_{k=0}^{S_{j}} r^{\lambda_{j}+ k(\alpha +1)} v_{j}^{(k)} (r,\omega)}
	= \frac{-r^{S_{j}(\alpha +1)}  v_{j}^{(S_{j})}(r,\omega)}{(-1)^{j-1} + \sum_{k=1}^{S_{j}} r^{k(\alpha +1)} v_{j}^{(k)} (r,\omega)} 
\end{align}
Under suitable bounding conditions for $|a_{j,k}^{(l)}\binom{l}{m} \omega^{l-m}|$,  we observe that, for positive exponents, the power series of the denominator converge absolutely, as shown in Proposition~\ref{PrpB4} in   Appendix~\ref{AppA}. Then, it holds,
\begin{itemize}
	\item If $\alpha > -1$ and $S_{j}$ tends to infinity, then $e_{DN}(r)$ converges to zero as $r$ tends to zero. 
	\item If $\alpha < -1$ and $S_{j}$ tends to infinity, then $e_{DN}(r)$ diverges as $r$ tends to  zero. 
	\item If $\alpha = -1$, then $S_{j}$ tends to infinity and for all $S_{j}\in \N$, $e_{DN}(r)$ diverges as $r$ tends to zero. 
\end{itemize} 
 
Similarly,  
we define the absolute error for the D-D approach defined by~\eqref{DD} as  
\begin{equation}\label{ErrAbsDD}
	E_{DD}(r):= \gamma r^{\alpha}u_{j}(r,\omega) +  \frac{1}{r}\frac{\partial u_{j} (r,\omega) }{\partial \theta }    
\end{equation}
from which we get 
\begin{equation}\label{ErrCharDD}
\begin{split}
\gamma r^{\alpha} u_{j}(r,\omega) + \frac{1}{r}\frac{\partial u_{j}(r,\omega)}{\partial \theta }  &=
\gamma r^{\alpha} u_{j}^{(0)}(r,\omega)  +\frac{1}{r}\frac{\partial u_{j}^{(0)}(r,\omega)}{\partial \theta } +    \gamma r^{\alpha} \sum_{k=1}^{S_{j}}u_{j}^{(k)}(r,\omega)  + \frac{1}{r}\sum_{k=1}^{S_{j}} \frac{\partial u_{j}^{(k)}(r,\omega)}{\partial \theta },\\
    & =   \frac{1}{r}\frac{\partial u_{j}^{(0)}(r,\omega)}{\partial \theta }  -    \frac{1}{r}\sum_{k=1}^{S_{j}} \frac{\partial u_{j}^{(k-1)}(r,\omega)}{\partial \theta } +   \frac{1}{r}\sum_{k=1}^{S_{j}} \frac{\partial u_{j}^{(k)}(r,\omega)}{\partial \theta } \\
	& =\frac{1}{r} \frac{\partial u_{j}^{(S_{j})}(r,\omega)}{\partial \theta}\\
	& = r^{\lambda_{j} - S_{j}(\alpha + 1)-1} \frac{\partial \vs_{j}^{(S_{j})}(r,\omega)}{\partial \theta},
\end{split}
\end{equation}
where 
\begin{equation}
\vs_{j}^{(k)}(r,\theta) :=\sum_{m=0}^{L_{j,k}} \log^{m}(r) \sum_{l=m}^{L_{j,k}}  \as_{j,k}^{(l)}  \binom{l}{m} \theta^{l-m}  \sin\left( (\lambda_{j} - k(\alpha + 1))\theta  + \frac{\pi}{2} (l-m)\right),
\end{equation}
with $\lambda_{j} = j\frac{\pi}{\omega}$. Thus, the absolute error for D-D approach is given by 
\begin{equation}
E_{DD}(r) = r^{\lambda_{j} - S_{j}(\alpha + 1)- 1} \frac{\partial \vs_{j}^{(S_{j})}(r,\omega)}{\partial \theta} .
\end{equation}
From the previous expression and for $r>0$, it is easy to conclude that 
\begin{itemize}
	\item If $\alpha < - 1$ and $S_{j}$ tends to infinity, then $E_{DD}(r)$ converges to zero as $r$ tends to zero. 
	\item If $\alpha > -1$ and $S_{j}$ tends to infinite, then $E_{DD}(r)$ diverges as $r$ tends to zero. 
	\item If $\alpha = -1$ and $j > \frac{\omega}{\pi}$, then $S_{j}$ tends to infinity and $E_{DD}(r)$ converges to zero as $r$ tends zero. 
\end{itemize} 
The relative error, in this case, is defined by 
\begin{equation}
e_{DD}(r) :=   \frac{-\gamma r^{\alpha} u_{j}(r,\omega)}{\frac{1}{r} \frac{\partial u_{j}(r,\omega)}{\partial \theta}}  -1.
\end{equation}
From which we get 
\begin{equation} \label{errCharDD}
	e_{DD}(r)  =  \frac{-E_{DD}(r)}{\frac{1}{r} \frac{\partial u_{j}(r,\omega)}{\partial \theta }} = \frac{-\frac{\partial u_{j}^{(S_{j})}(r,\omega)}{\partial \theta } }{\frac{\partial u_{j}(r,\omega)}{\partial \theta }}
     =  \frac{- r^{\lambda_{j} - S_{j}(\alpha +1)}\frac{\partial \vs_{j}^{(S_{j})}(r,\omega)}{\partial \theta}}{\sum_{k=1}^{S_{j}} r^{\lambda_{j} -k(\alpha + 1)}  \frac{\partial \vs_{j}^{(k)}(r,\omega)}{\partial \theta} }
     =  \frac{- r^{- S_{j}(\alpha +1)}\frac{\partial \vs_{j}^{(S_{j})}(r,\omega)}{\partial \theta}}{\sum_{k=1}^{S_{j}} r^{ -k(\alpha + 1)}  \frac{\partial \vs_{j}^{(k)}(r,\omega)}{\partial \theta} }.
\end{equation}
Then, we conclude that
\begin{itemize}
	\item If $\alpha < -1$ and $S_{j}$ tends to infinity, then $e_{DD}(r)$ converges to zero as $r$ tends to zero. 
	\item If $\alpha > -1$ and $S_{j}$ tends to infinity, then $e_{DD}(r)$ diverges as $r$ tends to zero. 
	\item If $\alpha = -1$, then $S_{j}$ tends to infinity and for all $S_{j}\in \N$, $e_{DD}(r)$ diverges as $r$ tends to zero. 
\end{itemize} 
\begin{rmk}
	We are interested in errors that are smooth and continuous functions that vanish  in a neighbourhood of $r=0$. In this way, the solution is essentially exact near the corner tip, for a sufficiently small $r\geq 0$. In practice, however, we can consider only  truncated series with a finite number of shadow terms even if $S_{j}=\infty$.  The cases of finite and infinite $S_{j}$  will be studied below. 
	The convergence of errors for the D-N approach are summarized in Table~\ref{TblDN2}, for $\alpha <-1$ and the errors for the D-D approach are  summarized in Table~\ref{TblDD2}, for $\alpha >-1$. Recall, that  the case $\alpha = -1$ was analysed in Section \ref{sec:case_alpha-1}, thus, it is not considered in the following Section \ref{CharAsympSeries}.
\end{rmk}

\subsection{Characteristic of the asymptotic series}\label{CharAsympSeries}
In this section, the behavior of the asymptotic series expansion of the singular eigensolutions is analyzed, determining when the recursive procedure stops or not, i.e., when the series is infinite or not, as a function of $\omega$ and $\alpha$, 
see \cite{Jimenez-Alfaro2020}, 
for a similar procedure, cf. \cite{Mghazli1992,CostabelDauge1996,Sinclair1999,Sinclair2004,Sinclair2015}. This study is summarized in Tables~\ref{TblDN1}--\ref{TblDD2}.

If  $\frac{\omega}{\pi}(\alpha  + 1)$ is an irrational number
\begin{equation}\label{Hip-1}
	\frac{\omega}{\pi}(\alpha  + 1) \notin \Q,
\end{equation}
then the conditions in~\eqref{CondDN_1},~\eqref{CondDN_2} and~\eqref{CondDD} are never fulfilled.	
Therefore, $L_{j,k} = 0$ for all $j,k\in \N$, and the shadow terms do not include logarithmic terms, for both recursive approaches D-N and D-D. The series of shadow terms is infinite, i.e. $S_{j}=\infty$, as will be shown in Section~\ref{InfSerie}.

Under the assumption that  $\frac{\omega}{\pi}(\alpha  + 1)$ is a rational number
\begin{equation} \label{Hip-2}
	\frac{\omega}{\pi}(\alpha  + 1) \in \Q,
\end{equation}
we call such a pair of values $(\omega,\alpha)$  a \textit{critical pair}.  In this case, there are two possibilities, the series has a finite number of shadow terms (without logarithmic terms), then the pair $(\omega,\alpha)$ is called \textit{apparent critical pair}, whereas if the series is defined with an infinite number of shadow terms (with logarithmic terms), then $(\omega,\alpha)$ is called  \textit{actual critical pair}. 
In the case $\alpha =-1$, we have $\sin(k\omega(\alpha + 1))=  0$ and $\cos(k\omega(\alpha + 1))\neq 0$ for all $k\in \N$ and $\omega \in (0,2\pi]$, which means that, for both recursive approaches D-N and D-D, the series of shadow terms is infinite adding logarithmic terms for each $k\in \N$, details about the infinite series will be seen in Subsection~\ref{InfSerie}, thus $(\omega,-1)$ is an actual critical pair. 

In the following subsections the cases of $\alpha\neq - 1$ are deeply analyzed, studying the two recursive approaches D-N and D-D. 

\subsubsection{D-N approach}
As shown in Section~\ref{SubsectCoeff_DN}, the behavior of the series of shadow terms is described by conditions~\eqref{CondDN_1} and~\eqref{CondDN_2}. This means that when at least one condition is fulfilled, the system~\eqref{SystDN} increases its size (adding logarithmic terms to the series). For the case $\alpha >-1$, which is represented in Table~\ref{TblDN1}, the condition~\eqref{CondDN_1} is never achieved. We now distinguish two possible irreducible fractions from~\eqref{Hip-2}, 
see Remark~\ref{PrpB0} in Appendix~\ref{AppA}.  

First, if 
\begin{equation}\label{ApprtCritical}
 	\frac{\omega}{\pi}(\alpha + 1) = \frac{2p - 1}{2q} \in \Q, \text{ with } p,q\in \N,
 \end{equation}	
 then $L_{j,q} = 0$. Indeed, for $k \leq q$ it holds $\sin(k\omega(\alpha + 1 )) = \sin\left(\frac{k}{q}(2p-1)\frac{\pi}{2} \right) \neq 0$, see Proposition~\ref{PrpB0} in Appendix~\ref{AppA}, i.e., the condition~\eqref{CondDN_2} is not fulfilled, hence, the size of the system~\eqref{SystDN} does not increase. 
 Moreover, $u_{j}^{(q)}(r,\omega) = 0$.  In fact, due to  $L_{j,q} = 0$ and by~\eqref{ExpRe} we infer  
 \[u_{j}^{(q)}(r, \omega) = (-1)^{j-1} a_{j,q}^{(0)}\ r^{\lambda_{j} + q(\alpha + 1)} \cos\left((2p-1)\frac{\pi}{2} \right) =  0. \]
 Therefore, $E_{DN}(r) = 0$ and $e_{DN}(r) = 0$ with $S_{j} = q$. This means that,
 before increasing the size of the $k$-th shadow term, (i.e., when $k$ is a multiple of $2q$) the absolute and relative errors are zero. In other words, the series is finite with $q$ terms and contains no logarithmic terms. Consequently, $(\omega, \alpha)$ is an apparently critical pair.
 
 Second, if 
\begin{equation}\label{ActCritical}
	\frac{\omega}{\pi}(\alpha + 1) = \frac{p}{2q - 1} \in \Q, \text{ with } p,q\in \N,
\end{equation}
then  $\sin(k\omega(\alpha + 1 )) = 0$ for $k$ a multiple of $2q-1$ and $L_{j,k} = L_{j,k-1} + 1$; that is, the $k$-th shadow term is increasing its size, resulting in an infinite series with powers of logarithmic function. Consequently, $(\omega, \alpha)$ is an actual critical pair. In Subsection~\ref{InfSerie}, we show that the series of shadow terms is infinite, i.e., $S_j = \infty$.




In the case of  $\alpha < -1$, summarized in the Table~\ref{TblDN2}, we can make a similar analysis to previous case, although in this case the condition~\eqref{CondDN_1} can be fulfilled. It is assumed that either
\begin{equation}\label{NOActCritical}
	\frac{\omega}{\pi}(\alpha + 1) = -\frac{p}{2q-1} \in \Q, \text{ with } p,q\in \N,
\end{equation}
or 
\begin{equation}\label{NOApprtCritical}
	\frac{\omega}{\pi}(\alpha + 1) = -\frac{2p - 1}{2q} \in \Q, \text{ with } p,q\in \N.
\end{equation}

Under the assumption~\eqref{NOActCritical}, the condition~\eqref{CondDN_1} is never fulfilled.  Thus, we only take into account the condition~\eqref{CondDN_2} (see Remark~\ref{Rmk1})
and the analysis is the same as the previous one, i.e., $(\omega,\alpha)$ is an actually critical pair. 

On the other hand, assuming~\eqref{NOApprtCritical}, the condition~\eqref{CondDN_1} holds when 
\[ k = q\frac{(2j-1)}{2p-1},\quad \text{with $j,p,q\in \N$ fixed and $k\in \N$},\] 
and two cases can be  distinguished, as it is shown in Table~\ref{TblDN2}. First, if $j = p$, for $k = q$, the number of shadow terms is increased by one (see Proposition~\ref{PrpB1} in Appendix~\ref{AppA}) and the series of shadow terms is infinite, as will be explained later.

On the contrary, if $j \neq p$ we have two possibilities, $2j-1$ is not a multiple of $2p-1$, hence $k\notin\N$ and~\eqref{CondDN_1} are not verified. The other option is that $2j-1$ is a multiple of $2p-1$, in case that it holds at least for $j > p$, from this it follows that $k > 2q$. Thus, before increasing the system by condition~\eqref{CondDN_1} we will have already finished our recursive procedure. Therefore, from these last two possibilities, we only take into account the condition~\eqref{CondDN_2} and the analysis is the same as the case $\alpha > -1$, i.e., $(\omega,\alpha)$ is an apparently critical pair. Consequently, $E_{DN}(r) = 0$ and $e_{DN}(r) = 0$, with $S_{j} = q$, i.e., with $q$ shadow terms  the absolute and relative errors in the Robin boundary condition vanish. 

\subsubsection{D-D approach}
A similar analysis can be made for the D-D recursive approach, in which the condition~\eqref{CondDD} describes the behavior of the asymptotic series, as was  explained in Section~\ref{SubsectCoeff_DD}.
For the case of $\alpha > -1$, summarized in Table~\ref{TblDD2}, we consider two options from the assumption~\eqref{Hip-2}. 

First, if 
\begin{equation}\label{ApprtCriticalDD}
	\frac{\omega}{\pi}(\alpha + 1) = \frac{2p - 1}{2q} \in \Q, \text{ with } p,q\in \N,
\end{equation}	
then we have $L_{j,q} = 0$. Indeed, for $k \leq q$ the condition in~\eqref{CondDD} is  
not accomplished, see Proposition~\ref{PrpB0} in Appendix~\ref{AppA}. Moreover, $\frac{\partial u_{j}^{(q)}(r,\omega)}{ \partial \theta} = 0$.  In fact, due to  $L_{j,q} = 0$, and by~\eqref{ExpRe} we have 
\begin{align}
	u_{j}^{(q)}(r, \theta) & = \as_{j,q}^{(0)}\ r^{\lambda_{j} - q(\alpha + 1)} \sin \left((\lambda_{j} -q(\alpha  + 1))\theta \right), 
\end{align}
which leads to
\begin{align}
	\frac{\partial u_{j}^{(q)}(r, \omega)}{\partial \theta} & = (-1)^{j-1}\as_{j,q}^{(0)}\ r^{\lambda_{j} - q(\alpha + 1)} (\lambda_{j} -q(\alpha  + 1))\cos\left(q\omega(\alpha  + 1) \right), \\
	& = (-1)^{j-1}\as_{j,q}^{(0)}\ r^{\lambda_{j} - q(\alpha + 1)} (\lambda_{j} -q(\alpha  + 1))\cos\left( (2p-1) \frac{\pi}{2}\right)\\
	& = 0.
\end{align}
Therefore, $E_{DD}(r) = 0$ and $e_{DD}(r) = 0$ with $S_{j} = q$; this means that, before increasing the $k$-th shadow term (i.e., when $k$ is a multiple of $2q$), the absolute and relative errors are zero when using $q$ shadow terms. Consequently, $(\omega, \alpha)$ is an apparently critical pair.

Second, if 
\begin{equation}\label{ActCriticalDD}
	\frac{\omega}{\pi}(\alpha + 1) = \frac{p}{2q - 1} \in \Q, \text{ with } p,q\in \N,
\end{equation}
then, the condition~\eqref{CondDD} is fulfilled for $k$ a multiple of $2q-1$, i.e., 
the size of the $k$-th shadow term is increased. This case results in an infinite series  of shadow terms, as we will see later. Consequently, $(\omega,\alpha)$ is an actually critical pair. 

In the case of  $\alpha < -1$, reported in Table~\ref{TblDD1}, the analysis is same to previous one, although assumptions~\eqref{ApprtCriticalDD} and~\eqref{ActCriticalDD} are negatives.

\begin{rmk}
Note that the errors (relative and absolute) in the D-D recursive approach depend on $\frac{\partial u_{j}^{(S_{j})}(r,\omega)}{\partial \theta}$. From~\eqref{ExpRe} we deduce
\begin{equation}
	\begin{split}
	\frac{\partial u_{j}^{(S_{j})}(r,\omega)}{\partial \theta}  =	& \  r^{\lambda_{j} - S_{j}(\alpha + 1)}\sum_{l=0}^{L_{j,S_{j} }}  \as_{j,S_{j}}^{(l)}  \sum_{m=0}^{l} \binom{l}{m} \log^{m}(r) \omega^{l-m} \left[ \frac{l-m}{\omega}\sin\left( (\lambda_{j} - S_{j}(\alpha + 1))\omega + \frac{\pi}{2} (l-m) \right)\right. \\
	& \left.  + \ (\lambda_{j} - S_{j}(\alpha + 1))\cos\left( (\lambda_{j} - S_{j}(\alpha + 1))\omega + \frac{\pi}{2} (l-m) \right) \right].
	\end{split}
\end{equation}
Under the assumption that $L_{j,S_{j}} = 0$ we have 
\begin{equation}
\frac{\partial u_{j}^{(S_{j})}(r, \omega)}{\partial \theta}  = (-1)^{j-1}\as_{j,S_{j}}^{(0)}\ r^{\lambda_{j} - S_{j}(\alpha + 1)} (\lambda_{j} - S_{j}(\alpha  + 1))\cos\left(S_{j}\omega(\alpha  + 1) \right). \end{equation}
If $\lambda_{j} - S_{j}(\alpha  + 1) = 0$, then $	\sin\left( S_{j} \omega (\alpha + 1) \right) = 0$. Thus, $L_{j,S_{j}} = 1$ and  
\begin{equation}
\frac{\partial u_{j}^{(S_{j})}(r, \omega)}{\partial \theta}  = \as_{j,S_{j}}^{(1)}. 
\end{equation}
\end{rmk}

\subsubsection{Criteria for an infinite series of shadow terms}\label{InfSerie}

In this section we establish a condition under which the series of shadow terms is infinite. 
Let $c \in \N$. We perform separate analyses depending on whether $\bm{a}_{j,c-1} \neq \bm{0}$ or $\bm{\as}_{j,c-1} \neq \bm{0}$.
Due to the fact that the systems~\eqref{SystDN} and~\eqref{SystDD} can be set consistent, we have a unique solution by fixing, when necessary, the values of coefficients $a_{j,c}^{(0)}$ and $\as_{j,c}^{(0)}$ by zero, respectively, depending on the approach. Then, the recursive procedure stops if and only if the right-hand side of systems, mentioned before, is  null, i.e.,  $\bm{g}_{j,c-1}= \bm{0}$, for D-N recursive approach and $\bm{\gs}_{j,c-1}= \bm{0}$, for the D-D approach. In that case, $\bm{a}_{j,c}= \bm{0}$ and $\bm{\as}_{j,c}=\bm{0}$ respectively and therefore, in the next step, $\bm{g}_{j,c} = \bm{0}$ and $\bm{\gs}_{j,c}= \bm{0}$, which leads to a null solution of the systems.
\begin{equation}
\bm{M}_{j,c+1} \bm{a}_{j,c+1} = \bm{g}_{j,c} \qquad \text{ and } \qquad  \bm{\Ms}_{j,c+1} \bm{\as}_{j,c+1} = \bm{\gs}_{j,c}
\end{equation}
Therefore, $\bm{a}_{j,k} = \bm{0}$ and $\bm{\as}_{j,k}= \bm{0}$ for all $k\geq c$. See Appendix~\ref{AppB} for a more detailed explanation about dependency of the recursive system. 

For the D-N approach, we analyze the following particular cases
\begin{itemize}
	\item[a)] If $L_{j,c}= L_{j,c-1} = 0$, then $\bm{g}_{j,c-1} =[g_{j,c-1}^{(0)}]$, where
\begin{equation}
    g_{j,c-1}^{(0)} =- \gamma a_{j,c-1}^{(0)} \cos((c - 1) \omega (\alpha + 1) ).
\end{equation}
	If $\cos((c - 1) \omega (\alpha + 1) )=0 $, then $g_{j,c-1}^{(0)} = 0$.   
	\item[b)] If $L_{j,c}= L_{j,c-1} + 1 = 1$, then $\bm{g}_{j,c-1} =[g_{j,c-1}^{(0)},  0 ]^\top$
	where 
\begin{equation}
    g_{j,c-1}^{(0)} =-\gamma a_{j,c-1}^{(0)} \cos((c - 1) \omega (\alpha + 1) ).
\end{equation}
	If $\cos((c - 1) \omega (\alpha + 1) ) = 0$, then $\bm{g}_{j,c-1} = \bm{0}$.
	\item[c)] If $L_{j,c} = L_{j,c-1} = 1$, then $\bm{g}_{j,c-1} =[g_{j,c-1}^{(0)},\ g_{j,c-1}^{(1)}]^\top$
	where 
	\begin{equation}\label{CriticalVector}
		\bm{g}_{j,c-1} = -\gamma
		\begin{pmatrix} 
			a_{j,c-1}^{(0)}\cos((c-1) \omega(\alpha + 1) ) -  a_{j,c-1}^{(1)}\omega\sin((c-1) \omega(\alpha+1)) \\[1.5mm] 
			a_{j,c-1}^{(1)} \cos((c-1) \omega(\alpha+1) ) \\[1.5mm] 				
	\end{pmatrix}
	\end{equation}
	From this, it is easy to see that
	\begin{itemize}
		\item[$\bullet$] If $a_{j,c-1}^{(0)} \neq 0$ and $a_{j,c-1}^{(1)}\neq 0$, then $\bm{g}_{j,c-1}\neq \bm{0}$.
		\item[$\bullet$] If $a_{j,c-1}^{(0)} = 0$ and $a_{j,c-1}^{(1)}\neq 0$, then  $\bm{g}_{j,c-1}\neq \bm{0}$. 
		\item[$\bullet$]If  $a_{j,c-1}^{(0)}\neq 0$ and  $a_{j,c-1}^{(1)} = 0$, then   $\bm{g}_{j,c-1} =\bm{0}$, when $\cos((c-1) \omega(\alpha + 1) )=0$.		
	\end{itemize}
\end{itemize}
In general, the vectors $\bm{g}_{j,c-1}$ at least contain a subvector of the form shown in~\eqref{CriticalVector}. Therefore, these last three cases are sufficient for our purposes, from which we deduce that if
\[\cos\left((c - 1)\omega(\alpha + 1)\right) \neq 0,\]
then the series contains infinitely many terms, which  is straightforward to verify under the assumptions~\eqref{Hip-1},~\eqref{ActCritical}, and~\eqref{NOActCritical}.
Moreover, this condition also holds in the case \(\alpha = -1\). 
A special case is when we are under the assumptions~\eqref{NOApprtCritical} and  $\lambda_{j} + c(\alpha + 1) = 0 $ or equivalently $c = q\frac{(2j-1)}{2p-1}$ with $j = p$, from this it follows that $\cos((c-1)\omega(\alpha + 1) )\neq 0$ (see Proposition~\ref{PrpB2} in Appendix~\ref{AppA}), i.e.,  if~\eqref{NOApprtCritical} is fulfilled and $j = p$, then the series of shadow terms has infinite many terms. 

Similar to the previous approach, we analyze the following particular cases for the recursive approach D-D.
\begin{itemize}
	\item[a)] If $L_{j,c}= L_{j,c-1} = 0$, then $\bm{\gs}_{j,c-1} =[\gs_{j,c-1}^{(0)}]$, where
	\begin{equation}
    \gs_{j,c-1}^{(0)} =\frac{-1}{\gamma} \as_{j,c-1}^{(0)}(\lambda_{j} - (c-1)(\alpha + 1) ) \cos((c - 1) \omega (\alpha + 1) ).
    \end{equation}
	
	If $\cos((c - 1) \omega (\alpha + 1) )=0 $, then $\gs_{j,c-1}^{(0)} = 0$.   
	\item[b)] If $L_{j,c}= L_{j,c-1} + 1 = 1$, then $\bm{\gs}_{j,c-1} =[\gs_{j,c-1}^{(0)},  0 ]^\top$
	where 
    \begin{equation} 
    \gs_{j,c-1}^{(0)} = \frac{-1}{\gamma} \as_{j,c-1}^{(0)} (\lambda_{j} - (c-1)(\alpha + 1) ) \cos((c - 1) \omega (\alpha + 1) ).
    \end{equation}
	If $\cos((c - 1) \omega (\alpha + 1) ) = 0$, then $\bm{\gs}_{j,c-1} = \bm{0}$.

	\item[c)] If $L_{j,c} = L_{j,c-1} = 1$, then $\bm{\gs}_{j,c-1} =[\gs_{j,c-1}^{(0)},\ \gs_{j,c-1}^{(1)}]^\top$
	where 
		\begin{equation}\label{CriticalVectorDD}
	\begin{aligned}
		\bm{\gs}_{j,c-1}   = \frac{-(\lambda_{j} - (c-1)(\alpha + 1))}{\gamma} &
		\begin{pmatrix} 
		\as_{j,c-1}^{(0)}\cos((c-1) \omega(\alpha + 1) ) - \as_{j,c-1}^{(1)}\omega \sin((c-1) \omega(\alpha+1))  \\[1.5mm] 
			\as_{j,c-1}^{(1)} \cos((c-1) \omega(\alpha+1) ) \\[1.5mm] 				
		\end{pmatrix} \\
		 -\frac{1}{\gamma} &
		\begin{pmatrix} 
			\as_{j,c-1}^{(1)} \cos((c-1) \omega(\alpha + 1) )\\[1.5mm] 
				0 \\[1.5mm] 				
			\end{pmatrix} 
	\end{aligned}
\end{equation}
	From this it is easy to see that  
	\begin{itemize}
		\item[$\bullet$] if $\as_{j,c-1}^{(0)} \neq 0$ and $\as_{j,c-1}^{(1)}\neq 0$, then $\bm{\gs}_{j,c-1}\neq \bm{0}$.
		\item[$\bullet$] If $\as_{j,c-1}^{(0)} = 0$ and $\as_{j,c-1}^{(1)}\neq 0$, then  $\bm{\gs}_{j,c-1}\neq \bm{0}$.  
		\item[$\bullet$] If  $\as_{j,c-1}^{(0)}\neq 0$ and $\as_{j,c-1}^{(1)} = 0$, then   $\bm{\gs}_{j,c-1} =\bm{0}$ if $\cos((c-1) \omega(\alpha + 1) )=0$.
	\end{itemize} 
\end{itemize}

\begin{rmk} 
	In the above items {\rm a)} and {\rm b)} we do not consider $(\lambda_{j} - (c-1)(\alpha + 1) ) = 0$, because in that case we have $\sin((c - 1) \omega (\alpha + 1)) = 0$ and therefore $L_{j,c-1} = 1$, $\as_{j,c-1}^{(0)} = 0$ and $\as_{j,c-1}^{(1)} \neq 0$. Thus, if  $(\lambda_{j} - (c-1)(\alpha + 1) ) = 0$, we should take the item {\rm c)}. Moreover, if $(\lambda_{j} - (c-1)(\alpha + 1) ) = 0$, then $\cos((c - 1) \omega (\alpha + 1)) \neq 0$ which implies that $\bm{\gs}_{j,c-1}\neq \bm{0}$. 
\end{rmk}
In general, the vectors $\bm{\gs}_{j,c-1}$ at least contain a subvector of the form shown in~\eqref{CriticalVectorDD}. Thus, the last three cases are sufficient to deduce that if \(\cos((c-1)\omega(\alpha + 1)) \neq 0\), the series of shadow terms has infinitely many terms. Under assumptions \eqref{Hip-1} and \eqref{ActCriticalDD}, it is clear that \(\cos((c-1)\omega(\alpha + 1)) \neq 0\). This condition also holds for \(\alpha = -1\).

\begin{rmk}
As we mentioned in  Subsection~\ref{Errores} and reported in Tables~\ref{TblDN2} and~\ref{TblDD2}, if $S_{j}$ tends to infinity, then the series does not converge in a neighborhood of zero when $\alpha<-1$ and $\alpha >-1$, for the recursive D-N and D-D approaches, respectively.	 
\end{rmk}

\subsection{Energy of singular eigensolutions in neighbourhood of the corner tip}
In this section we determine the conditions under which the energy of a singular eigensolution in neighbourhood of the corner tip is  finite or infinite.

The energy for an eigensolution of the above   D-R problem,  deduced from the variational formulation of the problem, evaluated in a neighbourhood of the corner tip of radius $R$,  denoted as $\Omega_R =  \{(r\cos{\theta}, r\sin{\theta}), 0<r<R, 0 <\theta< \omega\}$ with $\Gamma_{2R} =  \{(r\cos{\omega}, r\sin{\omega}), 0<r<R\}$, is given by
\begin{equation}
	E_R(u) := \frac{1}{2}\int_{\Omega_R} \|\nabla u \|^2 \df \bm{x} + \gamma \int_{\Gamma_{2R}} \|\bm {x}\|^{\alpha} |u(\bm{x}) |^2 \df S(\bm{x}).
\end{equation}
 Thus, using polar coordinates we can compute $E_R(u)$ as a limit
\begin{equation}\label{energy}
	E_R(u) = \lim_{\epsilon\rightarrow 0_{+}} \frac{1}{2}\int_{\epsilon}^{R}\int_{0}^{\omega} r\left( \frac{\partial u(r,\theta) }{\partial r}\right)^2 + \frac{1}{r}\left( \frac{\partial u(r,\theta) }{\partial \theta}\right)^2   \df \theta \df r + \gamma \int_{\epsilon}^{R} r^{\alpha + 1} |u(r,\omega)|^2 \df r .
\end{equation}
When applying this integral expression to the general representation of a shadow term in~\eqref{ExpRe}, we note that this representation is given by a sum of functions with separate variables $r$ and $\theta$. Thus, the energy of the shadow term can be studied considering the integrals for $r$ and $\theta$ independently. Therefore, we consider a general part of this representation, which can be written as
\begin{equation}
	f(r) := r^{\lambda_{j}  \pm k(\alpha + 1)} \log^n (r)
\end{equation}
and 
\begin{equation}
	g(\theta) :=  \theta^n \sin((\lambda_{j}  \pm k(\alpha  + 1))\theta + n\frac{\pi}{2} ) 
\end{equation}
with $n\in\N_{0}$. Note that $g$ and its derivatives are bounded, therefore their integrals in $(0,\omega)$ also are bounded. Thus, for our aim, it is enough to study the following integrals which are deduced from~\eqref{energy} 
\begin{align}
	\int_{\epsilon}^{R}r(f'(r))^2 \df r  &=  \int_{\epsilon}^{R} r^{2(\lambda_{j}\pm k(\alpha +1))-1} \log^{2(n-1)}(r)\left( (\lambda_{j}\pm k(\alpha +1))\log(r) +  n\right)^2 \df r,\\
	\int_{\epsilon}^{R}\frac{1}{r}(f(r))^2 \df r & = \int_{\epsilon}^{R} r^{2(\lambda_{j}\pm k(\alpha +1))-1} \log^{2n}(r) \df r,\\
	\int_{\epsilon}^{R}r^{\alpha + 1} (f(r))^2 \df r & = \int_{\epsilon}^{R} r^{2(\lambda_{j}\pm k(\alpha +1))+ \alpha  + 1 }  \log^{2n}(r) \df r.
\end{align}
Given that our analysis focuses on the energy behavior near zero, we can, without loss of generality, assume $R=1$. Integration by parts  $2n$ times (see Proposition~\ref{PrpB3} in Appendix) it holds 
\begin{equation}\label{EnerL1}
\begin{aligned}
      \lim_{\epsilon\to 0} \int_{\epsilon}^{1} r (f(r)')^{2}\df r  &=   
	\begin{cases} 
		C_{j,k}(\alpha), & \lambda_{j}  \pm k (\alpha + 1) > 0 \\ 
		\pm\infty, & \lambda_{j}  \pm k (\alpha + 1) \leq 0  
	\end{cases} \\
     \lim_{\epsilon\to 0} \int_{\epsilon}^{1} \frac{1}{r}(f(r))^{2}\df r  &=   
	\begin{cases} 
		C_{j,k}(\alpha), & \lambda_{j}  \pm k (\alpha + 1) > 0 \\ 
		\pm\infty, & \lambda_{j}  \pm k (\alpha + 1) \leq 0  
	\end{cases}\\
    \lim_{
    \epsilon\to 0} \int_{\epsilon}^{1} r^{\alpha+1} (f(r))^{2}\df r & =   
	\begin{cases} 
	C_{j,k}(\alpha), & 2(\lambda_{j}  \pm k (\alpha + 1)) + (\alpha + 2) > 0 \\ 
		\pm\infty, & 2(\lambda_{j}  \pm k (\alpha + 1)) + (\alpha + 2) \leq 0,  
	\end{cases}
\end{aligned}
\end{equation}
where $C_{j,k}(\alpha)$ is a suitable constant in each case.
In the recursive D-N approach with $\alpha \geq  -1$ and for the D-D approach with $\alpha \leq -1$ it is easy to see that the energy is finite. Thus, we only have to analyze the complementary cases.  

\paragraph{Energy in the case of the D-N approach with $\alpha < -1$.}

We distinguish two cases:
\begin{itemize}
    \item $k$ tends to infinite, that means that, the series has infinite terms and we deduce from limits~\eqref{EnerL1} that the energy tends to infinite near the zero. 
    \item $k$ finite, the series has a quantity finite of terms, that means, $\frac{\omega}{\pi}(\alpha + 1) = -\frac{2p-1}{2q}$ with $j\neq p$.  If $j>p$, then 
\begin{equation}
\lambda_{j}  + k(\alpha + 1) \geq \lambda_{j}  + q (\alpha + 1)  = (2j-1)\frac{\pi}{2\omega} + q (\alpha + 1) > 0. 
\end{equation}
Indeed, $(2j-1)\frac{\pi}{2\omega} + q (\alpha + 1) > 0 $ is equivalent to $\frac{2j-1}{2q} +  \frac{\omega(\alpha + 1) }{\pi}  > 0 $, therefore \[\frac{2j-1}{2q}  -  \frac{2p-1}{2q}  > 0, \quad \text{when } j > p.\] From this last statement and~\eqref{EnerL1} we conclude that the energy is finite and infinite for $j \leq p$. Note that on $\Gamma_{2}$ the $q-$th shadow term vanishes, thus from~\eqref{EnerL1} we deduce
\[2\lambda_{j}  + 2k(\alpha + 1) + \alpha + 2  >   (2j-1)\frac{\pi}{\omega} +  (2q-1)(\alpha + 1). \]
Hence, $(2j-1)\frac{\pi}{\omega} +  (2q-1)(\alpha + 1) > 0$ when $j >p$. In fact, the above is equivalent to
\begin{equation}
(2j-1) -  \frac{(2q-1)}{2q}(2p-1) > 0,
\end{equation}
which is true for $j > p$.  
\end{itemize}

\paragraph{Energy in the case of the D-D approach with $\alpha > -1$.}

Similar to the previous case, we have:
\begin{itemize}
    \item $k$ tends to infinite, implies that the energy is infinite.
    \item  $k$ finite, we can apply the same analysis as for the  D-N approach, but in this case we should take in account that $\lambda_{j} = j\frac{\pi}{\omega}$.
\end{itemize}

We now summarize our results in Tables~\ref{TblDN1}--\ref{TblDD2}.
\begin{table}[H]
	\begin{center}
\begin{tabular}{r c c l }
	\hline\\
	 $\frac{\omega}{\pi}(\alpha + 1 )$ &  $S_{j} $  &  $L_{j,k}$  &  Description   \\[1.5mm]
	 \hline
$\notin \Q$ &  $\infty$  & 0 & Infinite series without log terms\\[1.5mm] 
$\frac{2p-1}{2q}\in \Q$ & $q$ & 0 & Exact solution without log terms\\[1.5mm]

$\frac{p}{2q-1}\in \Q $ & $\infty$ &  $\frac{k}{2q-1}\in \N$ & Infinite series with log terms\\[1.5mm]
\end{tabular}
\caption{Classification of the asymptotic series of shadow terms for the recursive D-N approach, with $j,k\in \N$ and $\alpha > -1 $. }
\label{TblDN1}
\end{center}
\end{table}

\begin{table}[H]
\begin{center}
\begin{tabular}{r c c c c c l}
	\hline\\
	 $\frac{\omega}{\pi}(\alpha + 1 ) $ & $j$ & $S_{j} $  &  $L_{j,k}$  & Conv. & Energy & Description\\[1.5mm]
	 \hline
$\notin \Q$ & $\in \N$ & $\infty$  & 0 & $\bm{\xs}$ & $\infty$  & Infinite series without log terms \\[1.5mm] 
$-\frac{2p-1}{2q}\in \Q$ & $j \neq p$ & $q$ & 0 & \checkmark & Finite,  $j > p$ & Exact solution without log terms \\[2mm]
$-\frac{2p-1}{2q}\in \Q$ & $j = p$ & $\infty$ & $\left\{ \begin{minipage}[r]{1.2cm} $k = q$ $\frac{k}{2q}\in\N $\end{minipage} \right.$ & $\bm{\xs}$ & $\infty$ & Infinite series with log terms \\[5mm]
$-\frac{p}{2q-1}\in \Q $ & $\in \N$	& $\infty$ &  $ \frac{k}{2q-1}\in \N$ & $\bm{\xs}$ &$\infty$ & Infinite series with log terms \\[1.5mm]
\end{tabular}
\caption{Classification of the asymptotic series of shadow terms for the recursive D-N approach, with $k\in \N$ and  $\alpha < -1$.}
\label{TblDN2}
\end{center}
\end{table}

 \begin{table}[H]
	\begin{center}
\begin{tabular}{r c c l}
	\hline\\
	 $\frac{\omega}{\pi}(\alpha + 1 ) $ & $S_{j} $  &  $L_{j,k}$  &  Description \\[1.5mm]
	 \hline
$\notin \Q$ &  $\infty$  & 0 & Infinite series without log terms\\[1.5mm] 
$-\frac{2p-1}{2q}\in \Q$ & $q$ & 0 & Exact solution without log terms\\[1.5mm]
$-\frac{p}{2q-1}\in \Q $ & $\infty$ &  $ \frac{k}{2q-1}\in \N$ & Infinite series with log terms\\[1.5mm]
\end{tabular}
\caption{Classification of the asymptotic series of shadow terms for the recursive D-D approach with $j,k\in \N$ and 
  $\alpha < -1$. }
\label{TblDD1}
\end{center}
\end{table}

\begin{table}[H]
	\begin{center}
\begin{tabular}{r c c c c l}
	\hline\\
	$\frac{\omega}{\pi}(\alpha + 1 ) $ & $S_{j} $  &  $L_{j,k}$  & Conv. & Energy & Description\\[1.5mm]
	 \hline
$\notin \Q$ &  $\infty$  & 0 & $\bm{\xs}$ & $\infty$ & Infinite series without log terms \\[1.5mm] 
$\frac{2p-1}{2q}\in \Q$ & $q$ & 0 & \checkmark & Finite, $j \geq p$ & Exact solution without log terms \\[1.5mm]
$\frac{p}{2q-1}\in \Q $ & $\infty$ &  $ \frac{k}{2q-1}\in \N$ & $\bm{\xs}$ & $\infty$ & Infinite series with log terms \\[1.5mm]
\end{tabular}
\caption{Classification of the asymptotic series of shadow terms for the recursive D-D approach with $j,k \in \mathbb{N}$ and $\alpha > -1$.}
\label{TblDD2}
\end{center}
\end{table}
Note that the column of $L_{j,k}$ tells us when we increase the size of the linear system for the coefficients in the shadow term representation; for instance, if $k$ is a multiple of $2q-1$ we increase the value of $L_{j,k}$ by one. 

\section{Examples}

The recursive procedures presented in previous sections are implemented in Mathematica~\cite{Wolfram1991}. Thus,  some examples of $\omega$ and $\alpha$ 
were chosen to illustrate the behavior of some  singular eigensolutions $u_j$ and their derivatives $\frac{\partial u_j}{\partial r}$ (denoted in plots as $u_{,r}$, for the sake of simplicity) and $\frac{\partial u_j}{r\partial \theta}$ (denoted in plots as $r^{-1} u_{,\theta}$) , i.e.  the components of the eigensolution gradient in polar coordinates, considering both the D-N and the D-D approaches, in  a neighbourhood of the singularity point -- the corner tip. Moreover, the eigensolution for the special case of $\alpha=-1$ is also represented. It is worth mentioning that the results for $\alpha = 0$ under the approach D-N coincide with those presented in~\cite{Jimenez-Alfaro2020}.  It is important to note that for simplicity, the parameter $\gamma$ is set to $1$ in all the solution plots.

\subsection{Graphics for the recursive D-N approach}

 We   consider four representative examples for   values of $\omega$ and $\alpha$.  A preliminary analysis of the asymptotic series obtained in each case is described in the following lines.
\begin{itemize}
  \item  For $\omega = \frac{\pi}{2}$ and  $\alpha= \frac{3}{2}$ we get   according to  Table \ref{TblDN1} 
  \[ \frac{\omega(\alpha+1) }{\pi} =  \frac{5}{4} = \frac{2p-1}{2q}, \quad \text{with } p = 3 \text{ and } q = 2,\]
this means that the exact solution is obtained including $q$ shadow terms and without $\log$ terms, i.e., $(\omega,\alpha)$ is an  apparently critical pair. The solution for $j=1$  
given by 
\[ u_{1}(r,\theta) =  \frac{1}{21} \left(r^6 \sin (6 \theta )-6 \sqrt{2} r^{7/2} \sin \
\left(\frac{7 \theta }{2}\right)\right) + r \sin (\theta ),\]
is composed of the following main and shadow terms with their respective coefficients, 
\begin{align*}
u_{1}^{(0)}(r,\theta) & = a_{1,0}^{(0)}r \sin (\theta ), &  a_{1,0}^{(0)}  &=  1,\\
u_{1}^{(1)}(r,\theta) &=  a_{1,1}^{(0)}r^{7/2} \sin \left(\frac{7 \theta }{2}\right),&  a_{1,1}^{(0)} &= -\frac{2\sqrt{2}}{7},  \\
u_{1}^{(2)}(r,\theta) &=  a_{1,2}^{(0)}r^6 \sin (6 \theta ),& a_{1,2}^{(0)} &=  \frac{1}{21}.
\end{align*}
Fig. \ref{fig:GrafSolStress_DN_OmegaPi_2Alpha3_2} shows $u_1$ and its derivatives in the domain $\Omega_1$. Looking at Fig. \ref{fig:GrafSolStress_DN_OmegaPi_2Alpha3_2_a}, the Dirichlet boundary condition is verified, as expected, at $\theta=0$. No stress singularities are observed in the stress components in Figs. \ref{fig:GrafSolStress_DN_OmegaPi_2Alpha3_2_b} and \ref{fig:GrafSolStress_DN_OmegaPi_2Alpha3_2_c}.
\begin{figure}[H]
    \centering
    \subfloat[]{\includegraphics[width=0.33\textwidth]{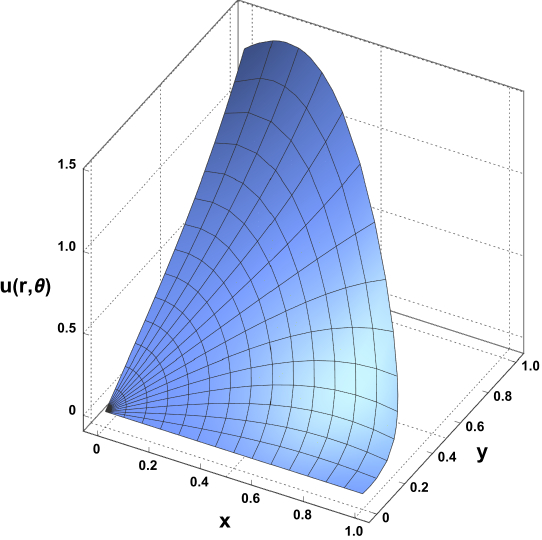}
    \label{fig:GrafSolStress_DN_OmegaPi_2Alpha3_2_a}}
    \subfloat[]{\includegraphics[width=0.33\textwidth]{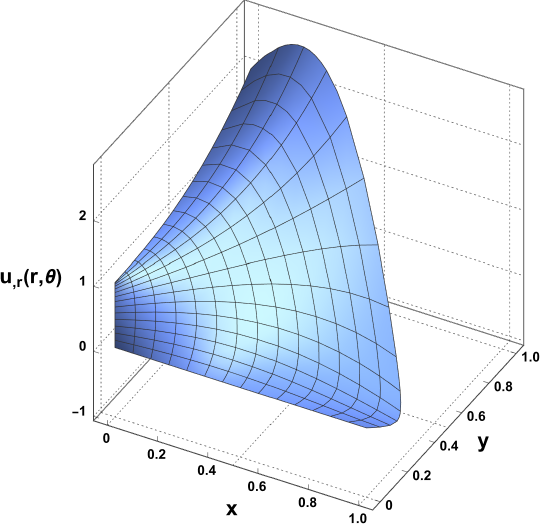}
    \label{fig:GrafSolStress_DN_OmegaPi_2Alpha3_2_b}}
    \subfloat[]{\includegraphics[width=0.33\textwidth]{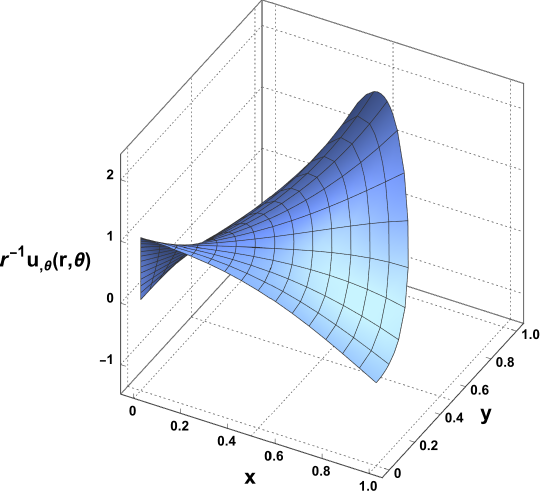}
    \label{fig:GrafSolStress_DN_OmegaPi_2Alpha3_2_c}}
    \caption{3D plots of the eigensolution (a)  $u_1$ and its derivatives (b) $u_{1,r}$ and (c) $r^{-1} u_{1,\theta}$, for $\omega=\frac{\pi}{2}$ and $\alpha=\frac{3}{2}$.}
    \label{fig:GrafSolStress_DN_OmegaPi_2Alpha3_2}
 \end{figure}

 The absolute and relative error in the Robin boundary condition is analyzed in Fig. \ref{fig:GrafErrores_DN_OmegaPi_2Alpha3_2}. It can be seen that they decrease when increasing the number of shadow terms.  Notice that, the case $S_1=0$ corresponds to the analysis of the main term. Moreover, when $S_1=2$, the exact solution is obtained, resulting in zero error for the Robin boundary condition.
\begin{figure}[H]
    \centering
    \subfloat[]{\includegraphics[scale = 0.6]{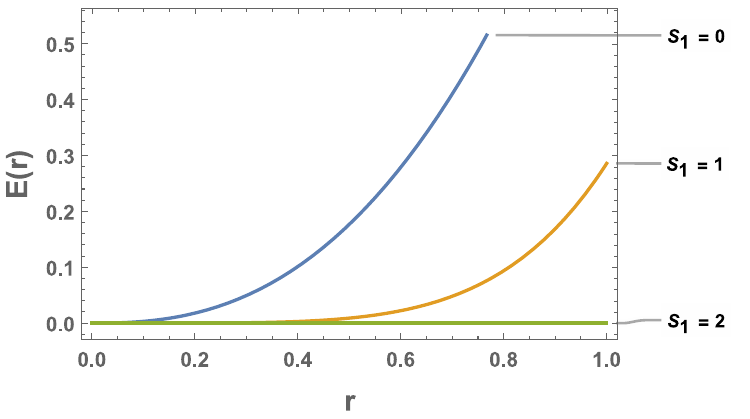}}\hspace{0.5cm}
    \subfloat[]{\includegraphics[scale = 0.6]{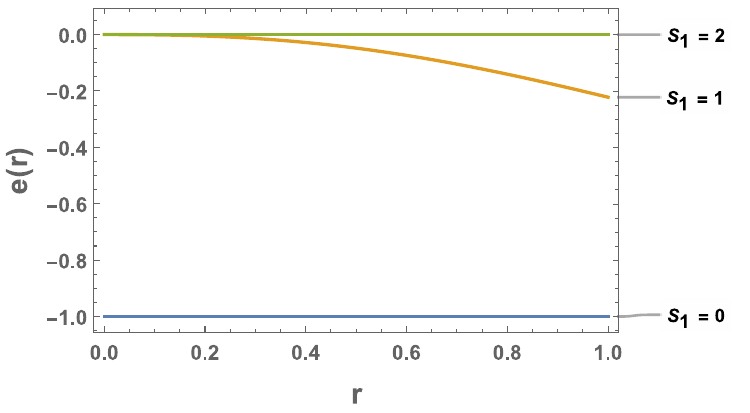}}
     \caption{Absolute and relative errors, $E_{DN}(r)$ and $e_{DN}(r)$,  of approximations of $u_1$, for $\omega=\frac{\pi}{2}$,  $\alpha=\frac{3}{2}$ and $j=1$.}
    \label{fig:GrafErrores_DN_OmegaPi_2Alpha3_2} 
 \end{figure}

 Fig.~\ref{fig:GrafSolOmStressRZOm_DN_OmegaPi_2Alpha3_2} shows the approximations of the eigensolution $u_{1}(r, \pi/2)$ and its derivative $u_{1,r}(r,\pi/2)$ for several values of $S_1$. Note that small differences can be observed for $S_{1} = 1$ and $S_{1} = 2$, although only the latter case would lead to the exact solution. Obviously,  $u_1(r,0) = 0$ and $u_{1,r}(r,0) = 0$.
\begin{figure}[H]
    \centering
    \subfloat[]{\includegraphics[scale = 0.6]{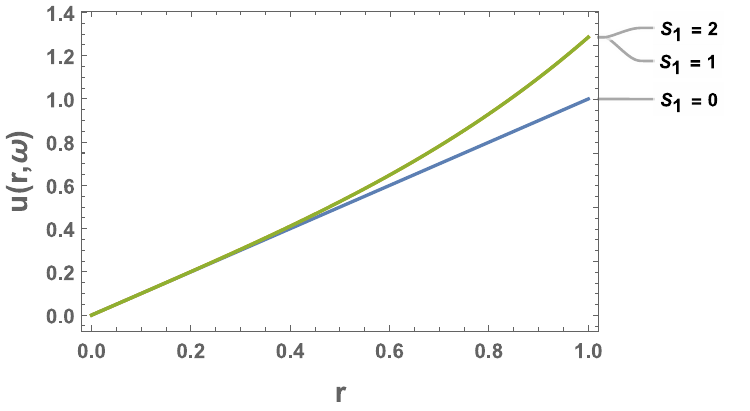}}\hspace{0.5cm}
    \subfloat[]{\includegraphics[scale = 0.6]{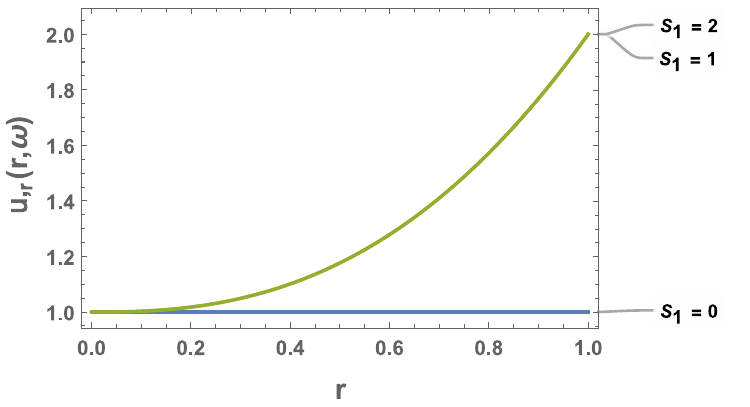}}\hspace{0.5cm}
     \caption{Plot of the approximations of (a) $u_{1}(r,\pi/2)$ and (b) $u_{1,r}(r,\pi/2)$ with increasing $S_1$, for $\omega=\frac{\pi}{2}$ and $\alpha=\frac{3}{2}$.}
    \label{fig:GrafSolOmStressRZOm_DN_OmegaPi_2Alpha3_2} 
 \end{figure}

 Fig. \ref{fig:GrafStressTZOm0_DN_OmegaPi_2Alpha3_2} shows how the approximations of the derivative $r^{-1}u_{1,\theta}(r,0)$ and  $r^{-1}u_{1,\theta}(r,\pi/2)$ change with $S_1$.  

 \begin{figure}[H]
  \centering
  \subfloat[]{\includegraphics[scale = 0.6]{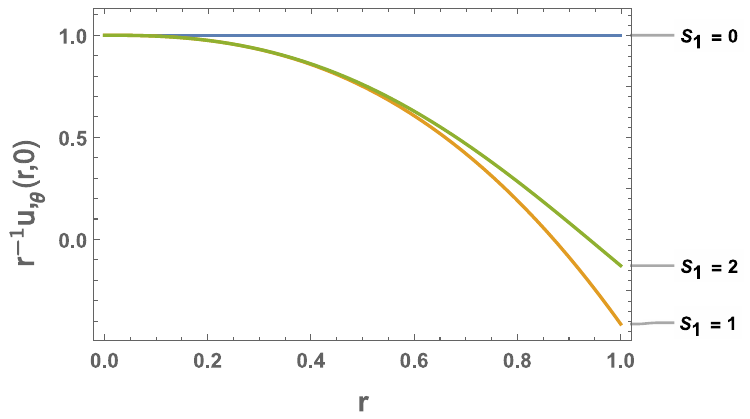}}\hspace{0.5cm}
  \subfloat[]{\includegraphics[scale = 0.6]{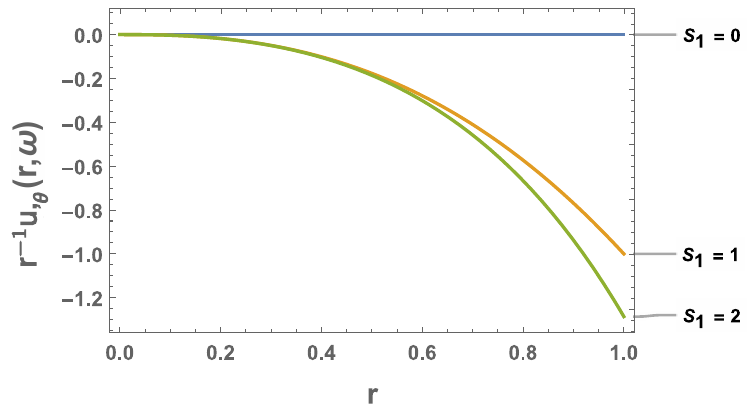}}
   \caption{Plots of the approximations of (a) $r^{-1}u_{1,\theta}(r,0)$ and  (b) $r^{-1}u_{1,\theta}(r,\pi/2)$, with increasing $S_1$, for $\omega=\frac{\pi}{2}$ and $\alpha = 3/2$.}
  \label{fig:GrafStressTZOm0_DN_OmegaPi_2Alpha3_2} 
\end{figure}

In general, for $j\in \N$ we have 

\[ u_{j}(r,\theta) =  u_{j}^{(0)}(r,\theta) + u_{j}^{(1)}(r,\theta) + u_{j}^{(2)}(r,\theta)\]
 
where the main and shadow terms with their respective coefficients are  
\begin{align*}
u_{j}^{(0)}(r,\theta) & = a_{j,0}^{(0)}r^{2j-1} \sin ((2j-1)\theta ), &  a_{j,0}^{(0)}  &=  1,\\
u_{j}^{(1)}(r,\theta) &=  a_{j,1}^{(0)}r^{2j + 3/2} \sin \left((2j + 3/2) \theta \right),&  a_{j,1}^{(0)} &=  \frac{ -\gamma \sin ((2j-1)\pi/2 )}{(2j + 3/2) \cos((2j+3/2)\pi/2)} = -\frac{2\sqrt{2}}{3+4j}, \\
u_{j}^{(2)}(r,\theta) &=  a_{j,2}^{(0)}r^{2j + 4} \sin ((2j + 4) \theta ),& a_{j,2}^{(0)} &= 
 \frac{-\gamma a_{j,1}^{(0)}\sin((2j+3/2)\pi/2) }{ (2j + 4) \cos((2j+4)\pi/2 )} = \frac{1}{(2+j)(3+4j)}. 
\end{align*}

\item  For $\omega = \frac{3\pi}{2}$ and  $\alpha= -\frac{3}{2}$ we get according to     Table \ref{TblDN2}   
\[ \frac{\omega(\alpha+1) }{\pi} =  -\frac{3}{4} = -\frac{2p-1}{2q}, \quad \text{with } p = 2 \text{ and } q = 2.\]
In this case $j > p$ is chosen, and therefore the exact solution is obtained including $q$ shadow terms and without $\log$ terms.  Hence, $\omega,\alpha)$ is an apparently critical pair. The solution for $j = 3$ is given by 
\[ u_{3}(r,\theta) =  \frac{1}{7} r^{2/3} \left(9 \sin \left(\frac{2 \theta }{3}\right)-6 \
\sqrt{2} \sqrt{r} \sin \left(\frac{7 \theta }{6}\right)+7 r \sin \
\left(\frac{5\theta }{3}\right)\right),\]
where the main and shadow terms with their respective coefficients are, 
\begin{align*}
u_{3}^{(0)}(r,\theta) & = a_{3,0}^{(0)} r^{5/3} \sin \left(\frac{5 \theta }{3}\right), &  a_{3,0}^{(0)}  &=  1,\\
u_{3}^{(1)}(r,\theta) &=  a_{3,1}^{(0)}  r^{7/6} \sin \left(\frac{7 \theta }{6}\right),&  a_{3,1}^{(0)} &=  -\frac{6}{7} \sqrt{2},   \\
u_{3}^{(2)}(r,\theta) &=  a_{3,2}^{(0)} r^{2/3} \sin \left(\frac{2 \theta }{3}\right)  ,& a_{3,2}^{(0)} &= \frac{9}{7}.
\end{align*}

 Fig. \ref{fig:GrafSolStress_DN_Omega3Pi_2Alpha-3_2} shows $u_3$ and its derivatives in the domain $\Omega_1$, where a clear   singularity in the derivatives is observed at the corner tip. Interestingly, it is not a logarithmic, but a polynomial singularity, since these derivatives  tend to infinity 
due to the negative power in $r$, in this case $-\frac{1}{3}$.  
\begin{figure}[H]
  \centering
  \subfloat[]{\includegraphics[width=0.33\textwidth]{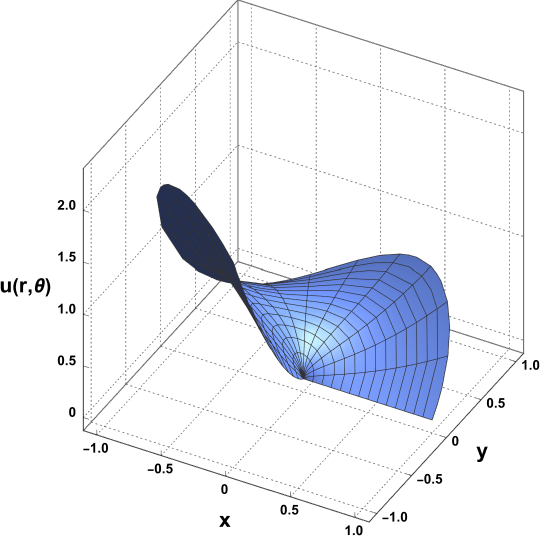}}
  \subfloat[]{\includegraphics[width=0.33\textwidth]{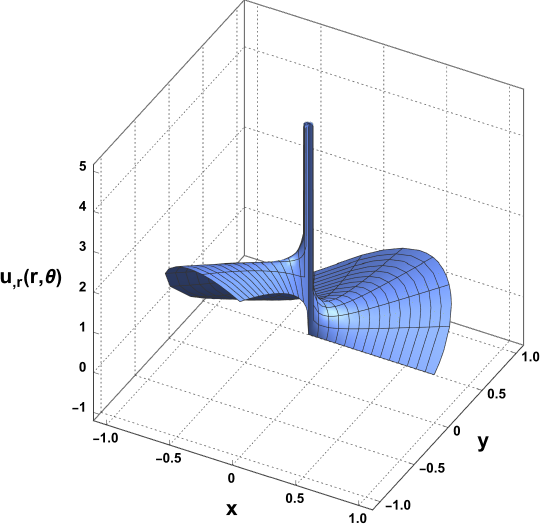}}
  \subfloat[]{\includegraphics[width=0.33\textwidth]{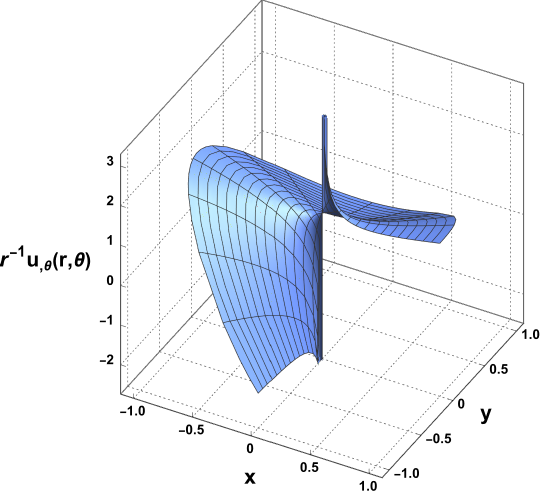}}
  \caption{3D plots of   the eigensolution   (a)   $u_3$ and its derivatives  (b) $u_{3,r}$ and (c) $r^{-1}u_{3,\theta}$, for $\omega=\frac{3\pi}{2}$ and  $\alpha=-\frac{3}{2}$.}
  \label{fig:GrafSolStress_DN_Omega3Pi_2Alpha-3_2}
\end{figure}
 Fig. \ref{fig:GrafErrores_DN_Omega3Pi_2Alpha-3_2} shows how the relative and absolute error in the Robin condition decreases with increasing $S_3$.  
\begin{figure}[H]
  \centering
  \subfloat[]{\includegraphics[scale = 0.6]{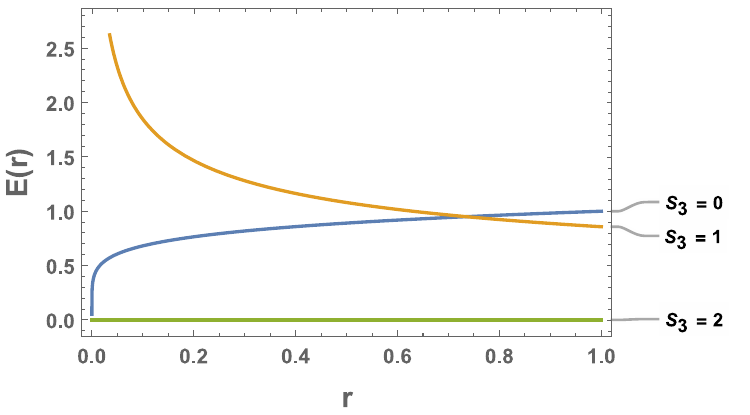}}\hspace{0.5cm}
  \subfloat[]{\includegraphics[scale = 0.6]{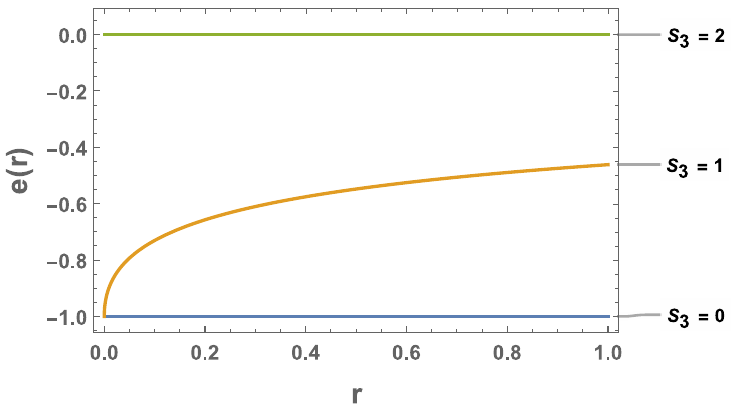}}
   \caption{Absolute and relative errors, $E_{DN}(r)$ and $e_{DN}(r)$, of approximations of $u_3$, for $\omega=\frac{3\pi}{2}$,   $\alpha=-\frac{3}{2}$ and $j=3$.}
  \label{fig:GrafErrores_DN_Omega3Pi_2Alpha-3_2} 
\end{figure}

The   singularity in derivatives can be also observed when the solution is represented at the boundaries of the corner, as  shown in Fig. \ref{fig:GrafStressTZOm0_DN_Omega3Pi_2Alpha-3_2}, where $r^{-1}u_{3,\theta}(r,0)$ and  $r^{-1}u_{3,\theta}(r,3\pi/2)$ are represented. Notice that in this case the solution strongly changes when $S_3$ is increased. This is not observed in Fig. \ref{fig:GrafSolOmStressRZOm_DN_Omega3Pi_2Alpha-3_2}, where  $u_{3}(r,3\pi/2)$ and   $u_{3,r}(r,3\pi/2)$ are represented. Therefore, the complete solution with $S_{3} = 2$ is needed to correctly represent the derivative singularity. 
\begin{figure}[H]
  \centering
  \subfloat[]{\includegraphics[scale = 0.6]{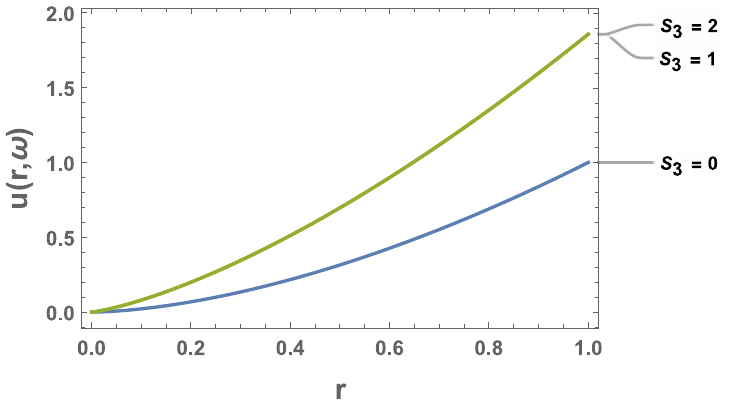}}\hspace{0.5cm}
  \subfloat[]{\includegraphics[scale = 0.6]{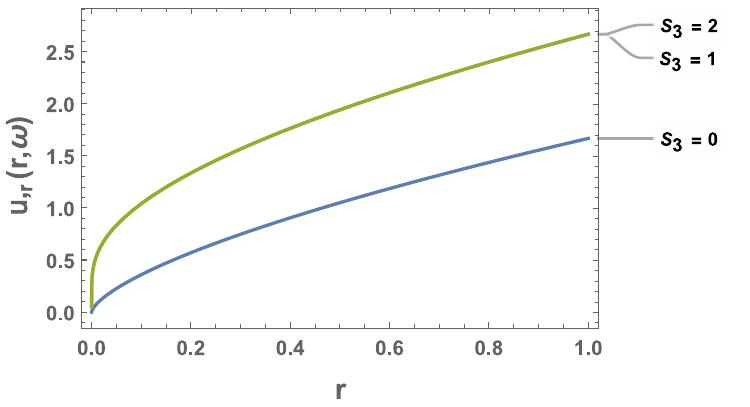}}\hspace{0.5cm}
   \caption{Plot of approximations of (a) $u_{3}(r,3\pi/2)$ and (b) $u_{3,r}(r,3\pi/2)$ with increasing $S_3$, for   $\omega=\frac{3\pi}{2}$ and $\alpha=-\frac{3}{2}$.}
  \label{fig:GrafSolOmStressRZOm_DN_Omega3Pi_2Alpha-3_2} 
\end{figure}

\begin{figure}[H]
\centering
\subfloat[]{\includegraphics[scale = 0.6]{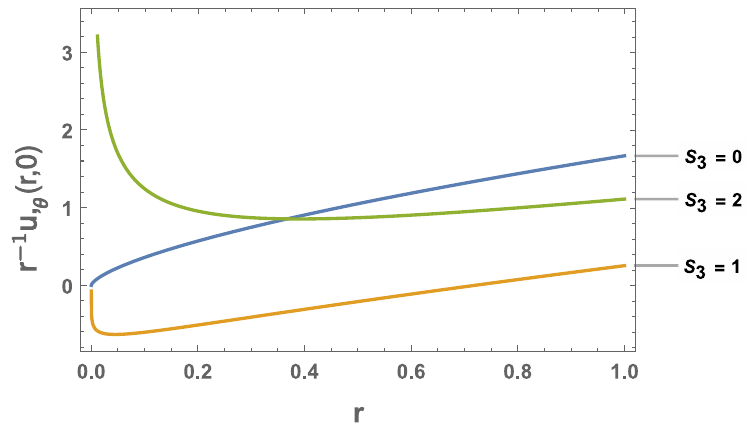}}\hspace{0.5cm}
\subfloat[]{\includegraphics[scale = 0.6]{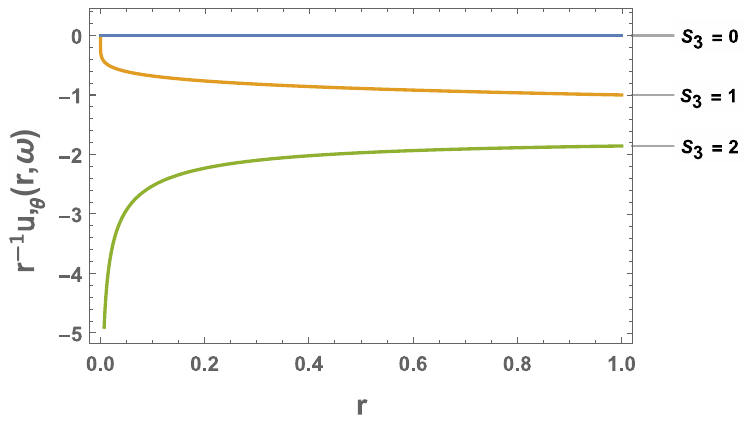}}
 \caption{Plots of approximations of (a) $r^{-1}u_{3,\theta}(r,0)$   and (b) $r^{-1}u_{3,\theta}(r,3\pi/2)$ for  increasing $S_3$, for $\omega=\frac{3\pi}{2}$ and $\alpha=-\frac{3}{2}$.}
\label{fig:GrafStressTZOm0_DN_Omega3Pi_2Alpha-3_2} 
\end{figure}

In general, for $j \geq 3$ we have 

\[ u_{j}(r,\theta) =  u_{j}^{(0)}(r,\theta) + u_{j}^{(1)}(r,\theta) + u_{j}^{(2)}(r,\theta),\]
where the main and shadow terms with their respective coefficients are  
\begin{align*}
u_{j}^{(0)}(r,\theta) & = a_{j,0}^{(0)}r^{(2j-1)/3} \sin ((2j-1)\theta/3 ), &  a_{j,0}^{(0)}  &=  1,\\
u_{j}^{(1)}(r,\theta) &=  a_{j,1}^{(0)}r^{(4j - 5)/6} \sin \left((4j -5)\theta/6 \right),&  a_{j,1}^{(0)} &=  \frac{ -\gamma \sin ((2j-1)\pi/2 )}{((4j - 5)/6) \cos((4j - 5 )\pi/4)} = \frac{6\sqrt{2}}{5-4j}, \\
u_{j}^{(2)}(r,\theta) &=  a_{j,2}^{(0)}r^{(2j - 4)/3} \sin ((2j - 4) \theta/3 ),& a_{j,2}^{(0)} &= 
\frac{-\gamma a_{j,1}^{(0)}\sin((4j - 5) \pi/4) }{ ((2j - 4)/3) \cos((2j - 4)\pi/2 )} = \frac{9}{(j-2) (4j- 5)}.
\end{align*}
Applying~\eqref{energy} and considering $R=1$, then  
\begin{align*}
  \int_{\epsilon}^{1} r^{\alpha + 1} |u(r,\omega)|^2 \df r & = \frac{6 r^{\frac{4 j}{3}-\frac{7}{6}}}{(5 - 4j)^2} \left(\frac{36}{8j -7} + \sqrt{r}\frac{3(4j-5)}{2j-1} + r\frac{(4j-5)^2}{8j-1}\right)\bigg. \bigg|_{\epsilon}^{1} 
\end{align*}
and
\begin{align*}
  \int_{\epsilon}^{1}\int_{0}^{\omega} & r\left( \frac{\partial u(r,\theta) }{\partial r}\right)^2 + \frac{1}{r}\left( \frac{\partial u(r,\theta) }{\partial \theta}\right)^2 \df \theta \df r    = \frac{r^{\frac{4 (j-2)}{3}}}{6 (5-4 j)^2} \left(\frac{48 (2 j-1) (5-4 j)^2 r^{3/2}}{7-8 j} + 72 (1-2 j) r  \right. \\ 
  & +  \left.\frac{3}{2} \pi  (4 j-5) r ((2 j (4j-7)+5)r + 36) +   \frac{864(4 j-5) \sqrt{r}}{13-8 j}+\frac{243 \pi }{j-2}\right)\bigg. \bigg|_{\epsilon}^{1}
\end{align*}
From these last integrals, it is easy to see that the energy is finite for $j \geq 3$.  In the case $j=2$, $a_{j,2}^{(0)}$ is infinite, therefore, the energy is infinite, and for $j = 1$ the power of $r$ is negative and when $\epsilon$ vanishes  the energy is infinite, which is in agreement with Table \ref{TblDN2}.

\item  For $\omega = \pi$ and  $\alpha= -\frac{3}{2}$ we get   according to    Table \ref{TblDN2}    
\begin{equation}
\frac{\omega(\alpha + 1) }{\pi} =  -\frac{1}{2} = -\frac{2p-1}{2q}, \quad \text{with } p = 1 \text{ and } q = 1.
\label{ex:DN_omega_pi_alpha-3/2}
\end{equation} 
In this case  we also choose $j > p$, Thus, we obtain the exact solution without log terms and $q$ shadow terms. Thus, $(\omega,\alpha)$ is an apparent critical pair. The solution for $j = 2$ is given by
\[ u_{2}(r,\theta) =  r^{3/2} \sin \left(\frac{3 \theta }{2}\right)- r \sin (\theta ),\]
where the main and shadow terms with their respective coefficients are  
\begin{align*}
u_{2}^{(0)}(r,\theta) & = a_{2,0}^{(0)}  r^{3/2} \sin \left(\frac{3 \theta }{2}\right), &  a_{2,0}^{(0)}  &=  1,\\
u_{2}^{(1)}(r,\theta) &= a_{2,1}^{(0)}r \sin \left(\theta\right),&  a_{2,1}^{(0)} &= -1.  \\
\end{align*}

Fig.~\ref{fig:GrafSolStress_DN_OmegaPi_1Alpha-3_2}  shows $u_2$ and its derivatives in the domain $\Omega_1$, where a clear non differentiability of  derivatives is observed at the corner tip.

\begin{figure}[H]
    \centering
    \subfloat[]{\includegraphics[width=0.33\textwidth]{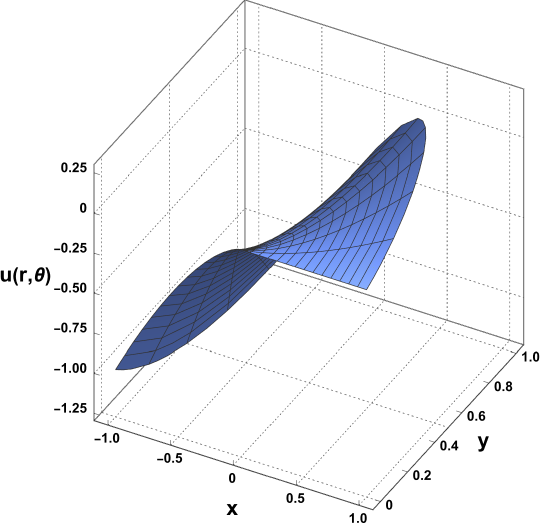}}
    \subfloat[]{\includegraphics[width=0.33\textwidth]{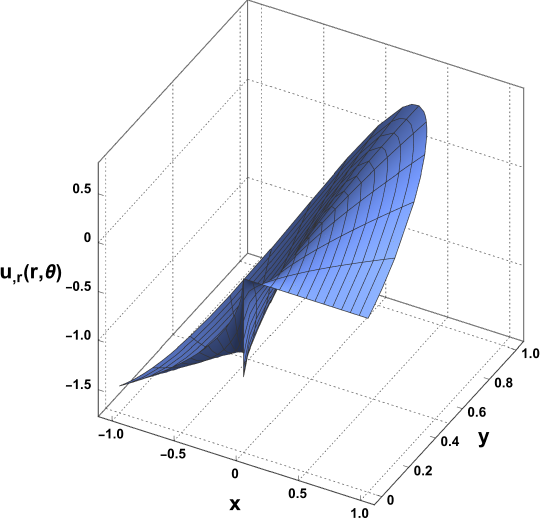}}
    \subfloat[]{\includegraphics[width=0.33\textwidth]{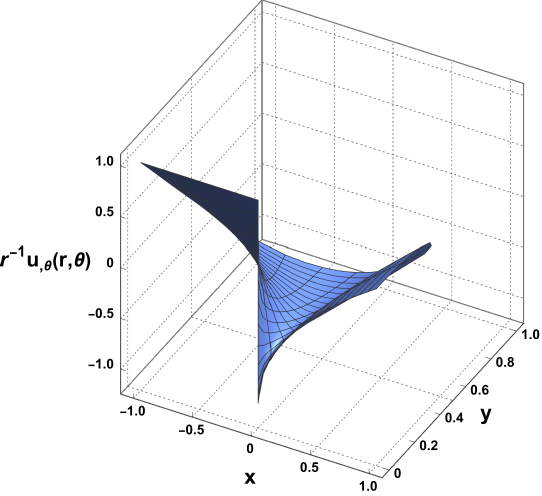}}
    \caption{3D plots of   the eigensolution   (a)   $u_2$ and its derivatives  (b) $u_{2,r}$ and (c) $r^{-1}u_{2,\theta}$, for $\omega=\pi$ and $\alpha=-\frac{3}{2}$.}
    \label{fig:GrafSolStress_DN_OmegaPi_1Alpha-3_2}
 \end{figure}
Since the exact solution is obtained when including just one shadow term, the error in the Robin boundary condition vanishes. Note that the approximation of $u_2(r,\pi)$  on the Robin boundary  for $S_2=0$ is equal to the solution $u_2(r,\pi)$ given by taking  $S_2=1$, see Fig.~\ref{fig:GrafSolOmStressRZOm_DN_OmegaPi_1Alpha-3_2}(a). The same observation is obviously valid for $u_{2,r}(r,\pi)$, see Fig.~\ref{fig:GrafSolOmStressRZOm_DN_OmegaPi_1Alpha-3_2}(b). 

\begin{figure}[H]
    \centering
    \subfloat[]{\includegraphics[scale = 0.6]{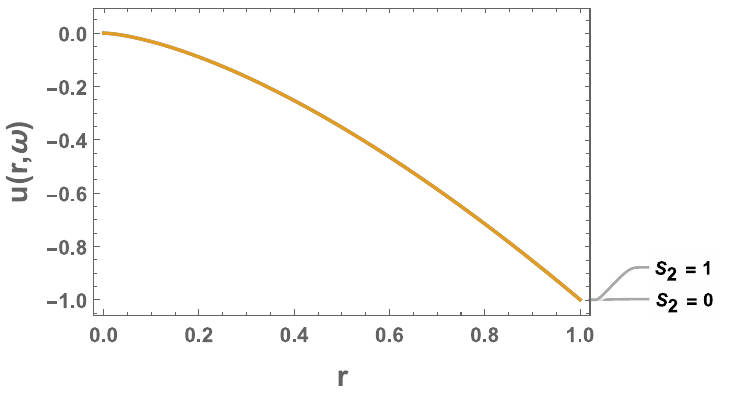}}\hspace{0.5cm}
    \subfloat[]{\includegraphics[scale = 0.6]{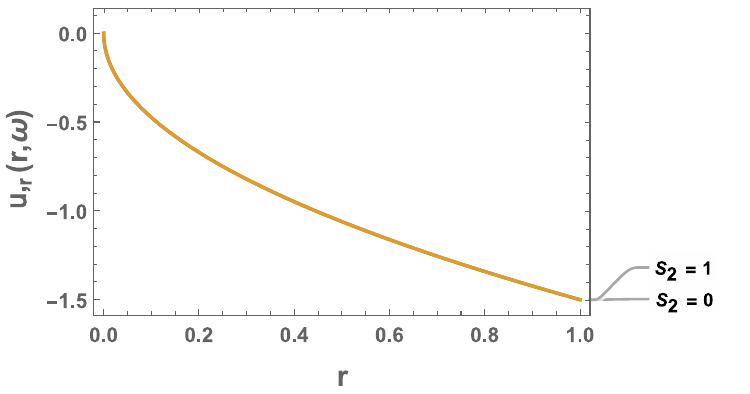}}\hspace{0.5cm}
     \caption{Plot of approximations of (a) $u_{2}(r,\pi)$ and (b) $u_{2,r}(r,\pi)$ with increasing $S_2$, for $\omega=\pi$ and $\alpha=-\frac{3}{2}$.}
    \label{fig:GrafSolOmStressRZOm_DN_OmegaPi_1Alpha-3_2} 
 \end{figure}
 
 The derivative  $r^{-1}u_{2,\theta}$ on the Dirichlet and Robin boundaries is  shown in Fig. \ref{fig:GrafStressTZOm0_DN_OmegaPi_1Alpha-3_2}, where the effect of including the shadow term can be noticed.  
 
 The   non differentiability of both derivatives  $u_{2,r}(r,\pi)$ and $r^{-1}u_{2,\theta}(r,0)$  at $r=0$, appreciated in the 3D representation in Fig.~\ref{fig:GrafSolStress_DN_OmegaPi_1Alpha-3_2},  is also observed in Figs.~\ref{fig:GrafSolOmStressRZOm_DN_OmegaPi_1Alpha-3_2}(b)  and~\ref{fig:GrafStressTZOm0_DN_OmegaPi_1Alpha-3_2}(a).
 \begin{figure}[H]
  \centering
  \subfloat[]{\includegraphics[scale = 0.6]{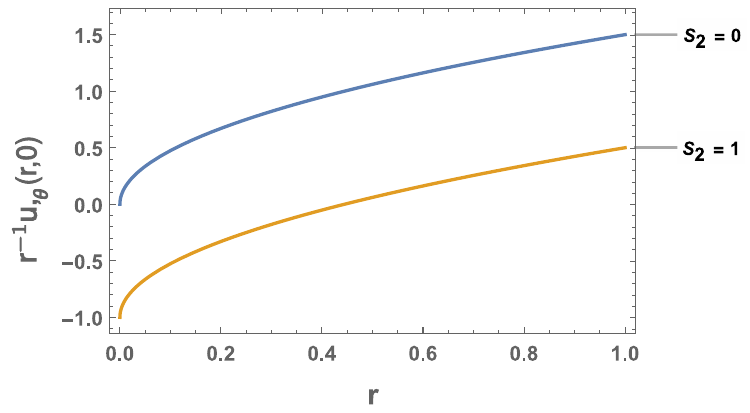}}\hspace{0.5cm}
  \subfloat[]{\includegraphics[scale = 0.6]{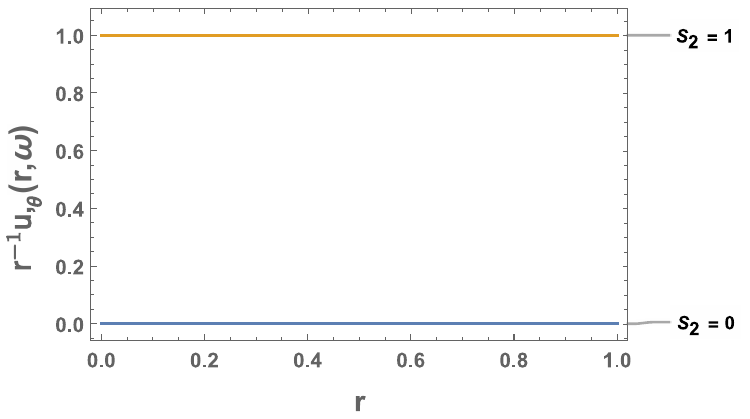}}
   \caption{Plots of approximations of (a) $r^{-1}u_{2,\theta}(r,0)$   and (b) $r^{-1}u_{2,\theta}(r,\pi)$ for  increasing $S_2$, for $\omega=\pi$ and $\alpha=-\frac{3}{2}$.}
  \label{fig:GrafStressTZOm0_DN_OmegaPi_1Alpha-3_2} 
 \end{figure}
In general, for $j\geq 2$ we have 
\[ u_{j}(r,\theta) =  u_{j}^{(0)}(r,\theta) + u_{j}^{(1)}(r,\theta),\]
where the main and shadow terms with their respective coefficients are  
\begin{align*}
u_{j}^{(0)}(r,\theta) & = a_{j,0}^{(0)}r^{(2j-1)/2} \sin ((2j-1)\theta/2 ), &  a_{j,0}^{(0)}  &=  1,\\
u_{j}^{(1)}(r,\theta) &=  a_{j,1}^{(0)}r^{j-1} \sin \left((j-1) \theta \right),&  a_{j,1}^{(0)} &=  \frac{ -\gamma \sin ((2j-1)\pi/2 )}{(j -1) \cos((j-1)\pi)} =\frac{1}{1-j} ,
\end{align*}
and similarity to the previous example we can show that the energy is finite for $j\geq 2$. From  \eqref{energy}  it holds 
\begin{align*}
  \int_{\epsilon}^{1} r^{\alpha + 1} |u(r,\omega)|^2 \df r & = \frac{2 r^{2 j- \frac{1}{2}}}{4j -1} \bigg. \bigg|_{\epsilon}^{1} 
\end{align*}
and
\begin{align*}
  \int_{\epsilon}^{1}\int_{0}^{\omega} & r\left( \frac{\partial u(r,\theta) }{\partial r}\right)^2 + \frac{1}{r}\left( \frac{\partial u(r,\theta) }{\partial \theta}\right)^2 \df \theta \df r    = \frac{1}{4} \left(r^{2 j-2} \left(\pi  (2 j-1) r+\frac{16 (1-2 j) \sqrt{r}}{4 j-3}+\frac{2 \pi }{j-1}\right)\right)\bigg. \bigg|_{\epsilon}^{1}.
\end{align*}
From this last integral, it is easy to conclude that the energy is finite for $j \geq 2$.

\item  For $\omega = 2\pi/3$ and  $\alpha= 2$ we get   according to Table \ref{TblDN1}  
\[ \frac{\omega(\alpha + 1) }{\pi} =  \frac{2}{1} = \frac{p}{2q-1}, \quad \text{with } p = 2 \text{ and } q = 1.\]
Thus,  the series of shadow terms is  infinite and includes $\log$ terms, i.e. $(\omega,\alpha)$ is an actual critical pair. For the sake of simplicity, only one shadow term is considered.   The eigensolution  $u_{1}$ is approximated  by the sum of the main and   the first shadow term 
\[ u_{1}(r,\theta) \approx  \frac{r^{3/4} \left(5 \pi  \sin \left(\frac{3 \theta }{4}\right)+2 \
\theta  r^3 \cos \left(\frac{15 \theta }{4}\right)+2 r^3 \sin \left(\
\frac{15\theta }{4}\right) \log (r)\right)}{5 \pi },\]
where the main and the shadow term with their respective coefficients are 
\begin{align*}
u_{1}^{(0)}(r,\theta) & = a_{1,0}^{(0)}  r^{3/4} \sin\left( \frac{3 \theta }{4} \right), &  a_{1,0}^{(0)}  & =  1,\\
u_{1}^{(1)}(r,\theta) & = a_{1,1}^{(0)}  r^{15/4} \sin \left(\frac{15 \theta }{4}\right) + a_{1,1}^{(1)}  r^{15/4} \left(\theta\cos \left(\frac{15 \theta }{4}\right) + \sin \left(\frac{15 \theta }{4}\right) \log(r)\right),&  a_{1,1}^{(1)} &= \frac{2}{5\pi}.  
\end{align*}
By construction, the term $a_{1,1}^{(0)} =  0$. 

Fig.~\ref{fig:GrafSolStress_DN_Omega2Pi_3Alpha2_1}  shows $u_1$ and its derivatives in the domain $\Omega_1$, where a clear singularity in the derivatives is observed at the corner tip, this behaviour disappears for sufficiently large $j$.

\begin{figure}[H]
    \centering
    \subfloat[]{\includegraphics[width=0.33\textwidth]{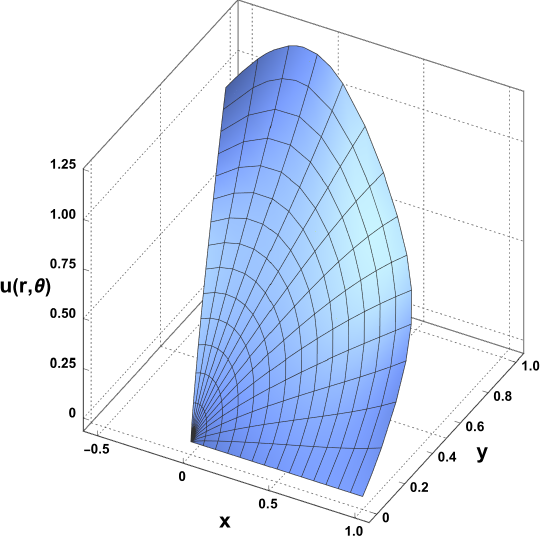}}
    \subfloat[]{\includegraphics[width=0.33\textwidth]{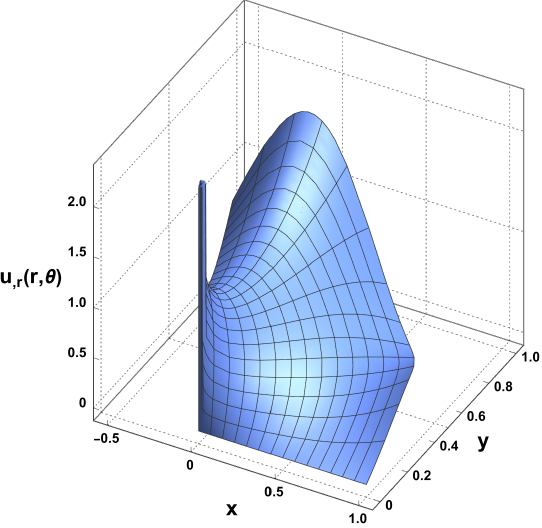}}
    \subfloat[]{\includegraphics[width=0.33\textwidth]{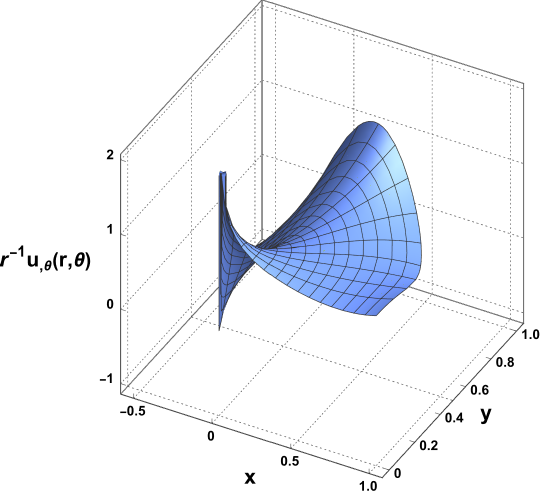}}
    \caption{3D plots of   an approximation of the eigensolution   (a)   $u_1$ and its derivatives  (b) $u_{1,r}$ and (c) $r^{-1}u_{1,\theta}$, for $S_1=1$, $\omega=2\pi/3$ and $\alpha=2$.}   \label{fig:GrafSolStress_DN_Omega2Pi_3Alpha2_1}
 \end{figure}

Fig.\ref{fig:GrafErrores_DN_Omega2Pi_3Alpha2_1} shows how the relative and absolute error in the Robin boundary condition decreases with increasing $S_1$.  A different range $r\in(0,3/2)$ has been used for a better representation. 

\begin{figure}[H]
    \centering
    \subfloat[]{\includegraphics[scale = 0.6]{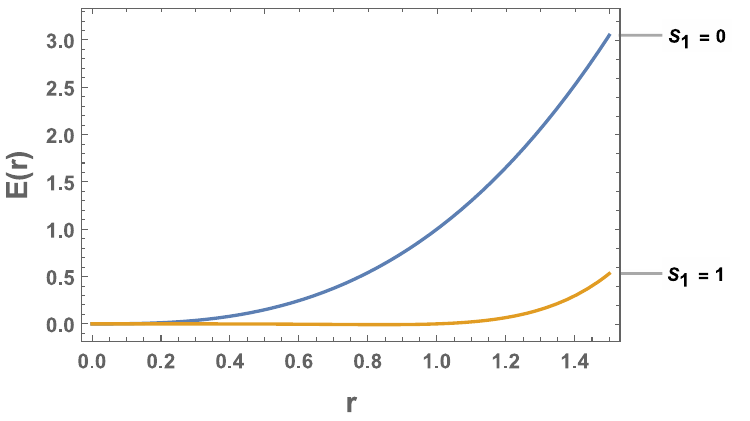}}\hspace{0.5cm}
    \subfloat[]{\includegraphics[scale = 0.6]{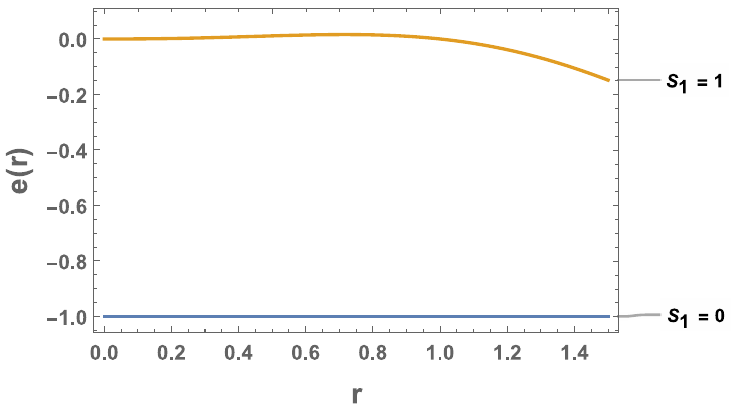}}
     \caption{Absolute and relative errors, $E_{DN}(r)$ and $e_{DN}(r)$, for $\omega=2\pi/3$, $\alpha=2$ and $j=1$.}
    \label{fig:GrafErrores_DN_Omega2Pi_3Alpha2_1} 
 \end{figure}

Fig. \ref{fig:GrafSolOmStressRZOm_DN_Omega2Pi_3Alpha2_1} shows how  the approximation of $u_{1}(r, 2\pi/3)$
and its derivative  $u_{1,r}(r,2\pi/3)$ change when increasing $S_1$. A greater influence of $S_{1}$ is appreciated in the derivative approximation.
\begin{figure}[H]
    \centering
    \subfloat[]{\includegraphics[scale = 0.6]{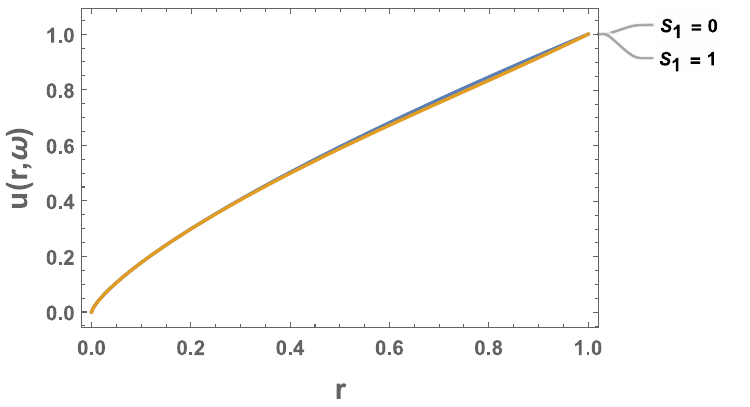}}\hspace{0.5cm}
    \subfloat[]{\includegraphics[scale = 0.6]{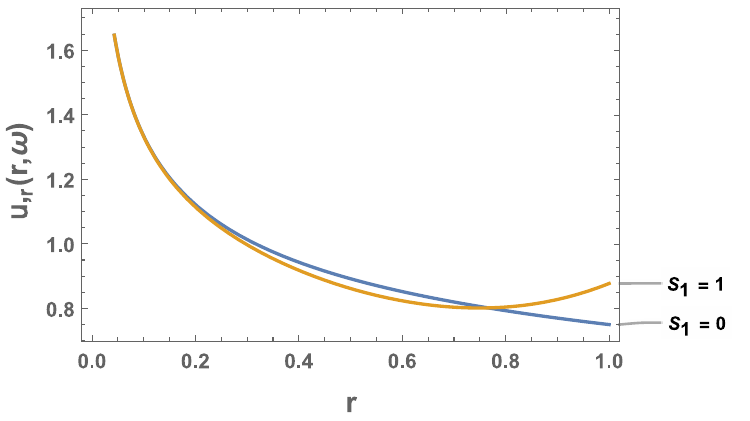}}\hspace{0.5cm}
     \caption{Plots of approximations of (a) $u_{1}(r,2\pi/3)$ and (b) $u_{1,r}(r,2\pi/3)$ with increasing $S_1$, for $\omega=2\pi/3$ and $\alpha=2$.}
    \label{fig:GrafSolOmStressRZOm_DN_Omega2Pi_3Alpha2_1} 
 \end{figure}

Fig. \ref{fig:GrafStressTZOm0_DN_Omega2Pi_3Alpha2_1} shows how the derivative $r^{-1}u_{1,\theta}$ changes with   increasing $S_1$ on the Dirichlet and  Robin boundaries. The singularity in this derivative observed in Fig. \ref{fig:GrafSolStress_DN_Omega2Pi_3Alpha2_1} is clearly shown here on the Dirichlet boundary, but not on the Robin boundary. Therefore, there is a jump in this derivative at the corner tip.
 \begin{figure}[H]
  \centering
  \subfloat[]{\includegraphics[scale = 0.6]{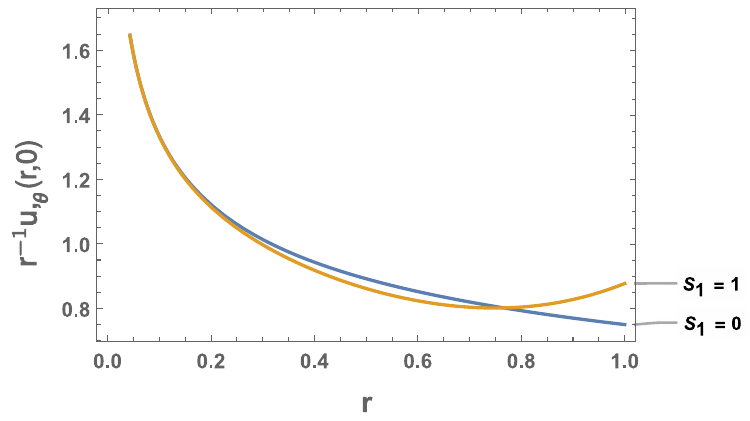}}\hspace{0.5cm}
  \subfloat[]{\includegraphics[scale = 0.6]{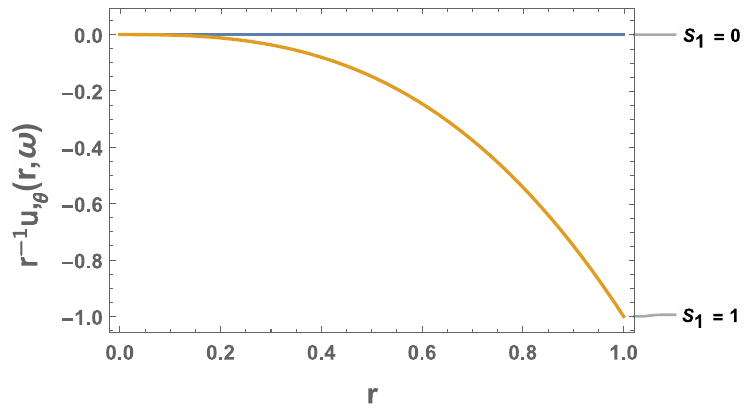}}
   \caption{Plots of approximations of (a) $r^{-1}u_{1,\theta}(r,0)$   and (b) $r^{-1}u_{1,\theta}(r,2\pi/3)$ with  increasing $S_1$, for $\omega=2\pi/3$ and $\alpha=2$.}
  \label{fig:GrafStressTZOm0_DN_Omega2Pi_3Alpha2_1} 
  \end{figure}

In general, for $j\in \N$ we have 

\[ u_{j}(r,\theta)    \approx  u_{j}^{(0)}(r,\theta) + u_{j}^{(1)}(r,\theta),\]
 
where the main and shadow terms with their respective coefficients are  
\begin{align*}
u_{j}^{(0)}(r,\theta) & = a_{j,0}^{(0)}  r^{(2j-1) 3/4} \sin\left( (2j-1)\frac{3 \theta }{4} \right),\\  
u_{j}^{(1)}(r,\theta) & = r^{(6j + 9)/4}\left[ a_{j,1}^{(0)}\sin\left( (6j + 9)\frac{\theta}{4} \right) +   a_{j,1}^{(1)} \left(\theta \cos \left((6j + 9)\frac{\theta}{4} \right) + \log(r)\sin\left( (6j + 9)\frac{\theta}{4}\right) \right)\right]  ,
\end{align*}
\end{itemize}
$a_{j,0}^{(0)}   =  1$,  $a_{j,1}^{(0)} =0 $ and $a_{j,1}^{(1)} = \frac{2}{\pi(2j + 3)}$.  In the case of using two shadow terms, the coefficients are obtained by solving the following system
\[\left(
\begin{array}{ccc}
 0 & -\pi(\frac{7}{2} + j) & -\frac{4 \pi }{3} \\[0.25cm]
 0 & 0 & -\pi(7 + 2j)) \\[0.25cm]
 0 & 0 & 0 \\
\end{array}
\right) \left(\begin{array}{c}
  a_{j,2}^{(0)} \\[0.25cm]
  a_{j,2}^{(1)} \\[0.25cm]
  a_{j,2}^{(2)} \\
 \end{array}  \right)  =  
  \left(
\begin{array}{c}
 0\\[0.25cm]
 -\frac{2}{\pi  (2 j+3)}\\[0.25cm]
 0\\
\end{array}\right),\qquad \left(\begin{array}{c}
  a_{j,2}^{(0)} \\[0.25cm]
  a_{j,2}^{(1)} \\[0.25cm]
  a_{j,2}^{(2)} \\
 \end{array}  \right)  = \left(\begin{array}{c}
 0\\[0.25cm]
 -\frac{16}{3 (3+2 j) (7+2 j)^2 \pi^2}\\[0.25cm]
 \frac{2}{(21+20 j+4 j^2) \pi^2}\\
 \end{array}  \right) \]
Therefore, 
\begin{align*}
  u_{j}^{(2)}(r,\theta)  = &  -\frac{2 r^{\frac{3}{4} (2 j+7)}}{3 \pi ^2 (2 j+3) (2 j+7)^2} \left[\sin \left(\frac{3}{4} \theta  (2 j+7)\right) \left(3 \theta ^2 (2 j+7)-3 (2 j + 7) \log ^2(r)+8 \log (r)\right)\right. \\
   & \left. -2 \theta  \cos \left(\frac{3}{4} \theta  (2 j+7)\right) (3 (2 j+7) \log (r)-4)\right] .
\end{align*}

\subsection{Graphics for the recursive  D-D approach}\label{subsec:GraphDD}
For values of $\omega$ and $\alpha$, we are going to consider three representative examples.
\begin{itemize}
  \item  For $\omega = \pi$ and  $\alpha= -3/2$ we get   according to  Table \ref{TblDD1}   
  \[ \frac{\omega(\alpha + 1) }{\pi} =  -\frac{1}{2} = -\frac{2p-1}{2q}, \quad \text{with } p = 1 \text{ and } q = 1,\]
this means, that the exact solution is obtained including $q$ shadow terms and without $\log$ terms, i.e., $(\omega,\alpha)$ is an  apparently critical pair.
 The solution for $j=1$   given by 
  \begin{equation}\label{ex:DD_omega_pi_alpha-3/2}
    u_{1}(r,\theta) = r \sin (\theta ) - r^{3/2} \sin \left(\frac{3 \theta }{2}\right),
  \end{equation}
is composed by the following main and shadow terms with their respective coefficients
\begin{align*}
u_{1}^{(0)}(r,\theta) & = \as_{1,0}^{(0)} r \sin (\theta ), &  \as_{1,0}^{(0)}  &=  1,\\
u_{1}^{(1)}(r,\theta) & = \as_{1,1}^{(0)}r^{3/2} \sin \left(\frac{3 \theta }{2}\right),&  \as_{1,1}^{(0)} &= -1.
\end{align*}
This eigensolution $u_{1}$ has the same expression except for the change of sign as $u_{2}$, for the same corner problem in the D-N approach  (see above), but the main and shadow terms are interchanged. Thus, similar conclusions could be deduced. 
3D plots of $u_{1}$ and its derivatives in the domain $\Omega_1$ are shown in Fig. \ref{fig:GrafSolStress_DD_OmegaPi_1Alpha-3_2}, cf. Fig.~\ref{fig:GrafSolStress_DN_OmegaPi_1Alpha-3_2}. Since the exact solution is obtained by including just one shadow term, the error in the Robin boundary condition vanishes.
\begin{figure}[H]
    \centering
    \subfloat[]{\includegraphics[width=0.33\textwidth]{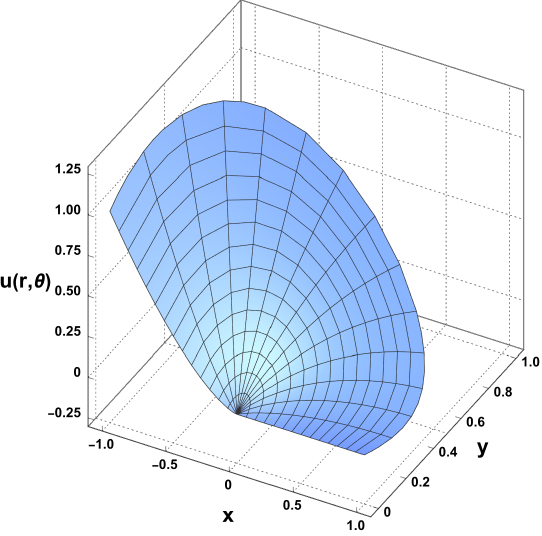}}
    \subfloat[]{\includegraphics[width=0.33\textwidth]{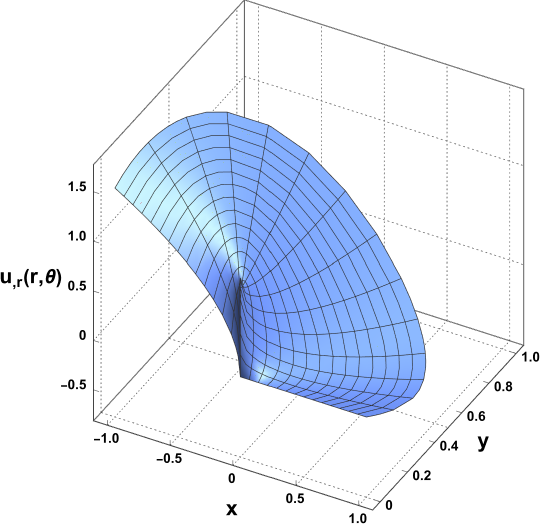}}
    \subfloat[]{\includegraphics[width=0.34\textwidth]{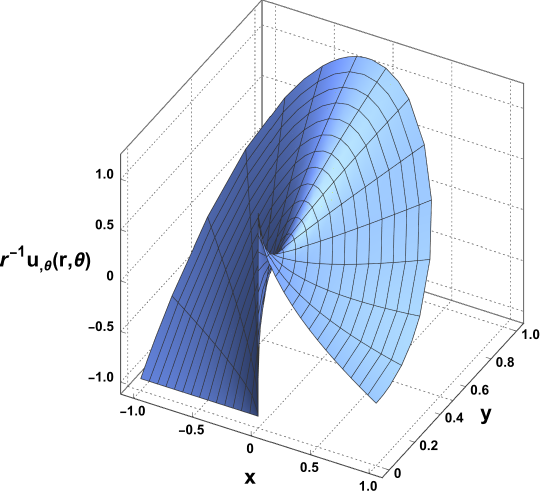}}
    \caption{3D plots of the eigensolution (a)  $u_1$ and its derivatives (b) $u_{1,r}$ and (c) $r^{-1} u_{1,\theta}$, for $\omega=\pi$, $\alpha=-\frac{3}2$.}
    \label{fig:GrafSolStress_DD_OmegaPi_1Alpha-3_2}
 \end{figure}

The Fig. \ref{fig:GrafSolOmStressRZOm_DD_OmegaPi_1Alpha-3_2}  shows how  $u_{1}(r, \pi)$ and its derivative $u_{1,r}(r,\pi)$ change  with increasing $S_1$, cf.  
Fig. \ref{fig:GrafSolOmStressRZOm_DN_OmegaPi_1Alpha-3_2}.
\begin{figure}[H]
    \centering
    \subfloat[]{\includegraphics[scale = 0.6]{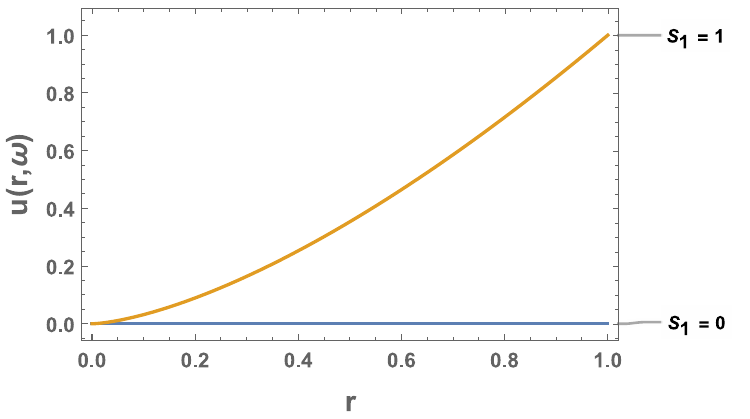}}\hspace{0.5cm}
    \subfloat[]{\includegraphics[scale = 0.6]{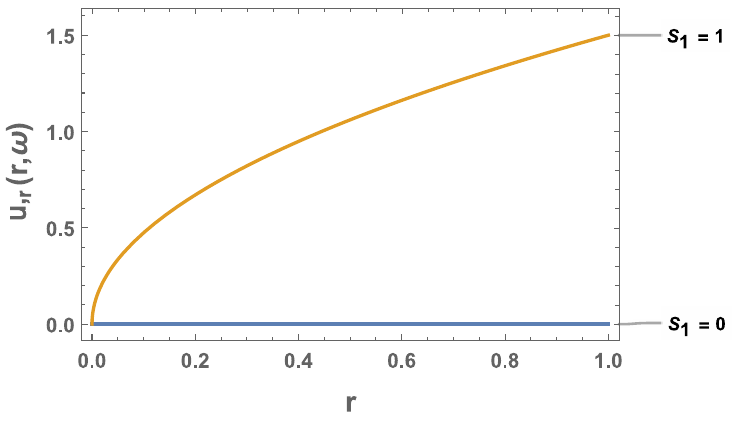}}\hspace{0.5cm}
     \caption{Plots of approximations of (a) $u_{1}(r,\pi)$ and (b) $u_{1,r}(r,\pi)$ with increasing $S_1$, for $\omega=\pi$ and $\alpha=-\frac{3}{2}$.}
\label{fig:GrafSolOmStressRZOm_DD_OmegaPi_1Alpha-3_2} 
 \end{figure}

 The derivative  $r^{-1}u_{1,\theta}$ on the Dirichlet and Robin boundaries is  plotted in Fig.~\ref{fig:GrafStressTZOm0_DD_OmegaPi_1Alpha-3_2}, where the effect of including the shadow term is noticed only on the Dirichlet boundary, cf. Fig.~\ref{fig:GrafStressTZOm0_DN_OmegaPi_1Alpha-3_2} 

 \begin{figure}[H]
  \centering
  \subfloat[]{\includegraphics[scale = 0.6]{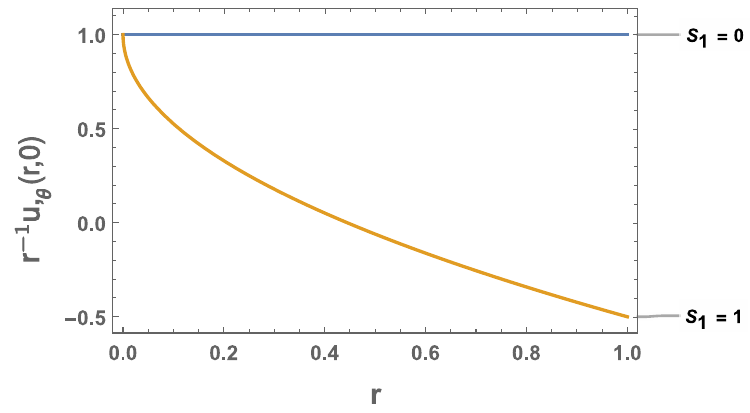}}\hspace{0.5cm}
  \subfloat[]{\includegraphics[scale = 0.6]{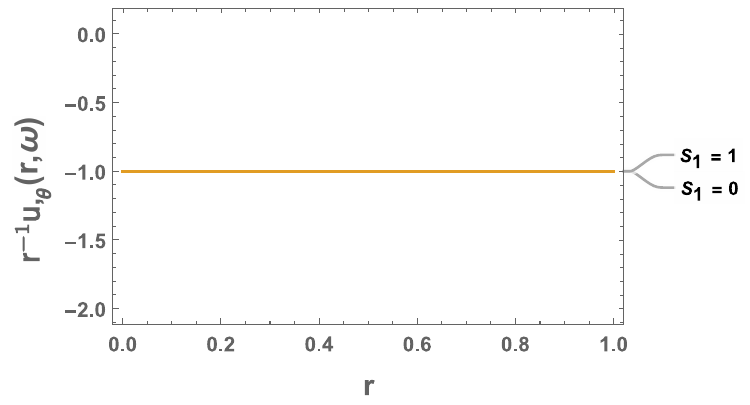}}
   \caption{Plots of approximations of (a) $r^{-1}u_{1,\theta}(r,0)$   and (b) $r^{-1}u_{1,\theta}(r,\pi)$ for  increasing $S_1$, for $\omega=\pi$ and $\alpha=-\frac{3}{2}$.}
  \label{fig:GrafStressTZOm0_DD_OmegaPi_1Alpha-3_2} 
\end{figure}

 In general, for $j\in \N$ we have 
\[ u_{j}(r,\theta) =  u_{j}^{(0)}(r,\theta) + u_{j}^{(1)}(r,\theta),\]
where the main and shadow terms with their respective coefficients are  
\begin{align*}
u_{j}^{(0)}(r,\theta) & = \as_{j,0}^{(0)}r^{j} \sin (\theta ), &  \as_{j,0}^{(0)}  &=  1,\\
u_{j}^{(1)}(r,\theta) &=  \as_{j,1}^{(0)}r^{j+ 1/2} \sin((j+ 1/2) \theta ),&  \as_{j,1}^{(0)} &=  -j .
\end{align*}
 
  \item  For $\omega = \pi$ and  $\alpha = -5/3$ we get   according to  Table \ref{TblDD1}  
\[\frac{\omega(\alpha+1) }{\pi} =  -\frac{2}{3} = -\frac{2p}{2q-1}, \quad \text{with } p = 2 \text{ and } q = 2.\]
Thus, the series of shadow terms is infinite and includes $\log$ terms, the size of the linear system increases when $k$ is a multiple of $3$, i.e. $(\omega,\alpha)$ is an actual critical pair. 

An approximation of the eigensolution  $u_{1}$     by the sum of the main and   the three  shadow terms  is given by
\begin{equation}
 u_{1}(r,\theta) \approx  \frac{10}{9} r^{7/3} \sin \left(\frac{7 \theta }{3}\right)-\frac{2 r^{5/3} \sin \left(\frac{5 \theta }{3}\right)}{\sqrt{3}}+\frac{35 r^3 (\theta\cos (3 \theta )+\sin (3 \theta ) \log (r))}{27 \pi }+r \sin (\theta ),   
 \label{DDu1-5/3}
\end{equation}
where the main and shadow terms with their respective coefficients are  
\begin{align*}
u_{1}^{(0)}(r,\theta) & = \as_{1,0}^{(0)}r \sin (\theta ), &  \as_{1,0}^{(0)}  &=  1\\
u_{1}^{(1)}(r,\theta) & = \as_{1,1}^{(0)}r^{5/3} \sin \left(\frac{5 \theta }{3}\right),&  \as_{1,1}^{(0)} &= -\frac{2}{\sqrt{3}}  \\
u_{1}^{(2)}(r,\theta) & = \as_{1,2}^{(0)} r^{7/3} \sin \left(\frac{7 \theta }{3} \right), & \as_{1,2}^{(0)} &=  \frac{10}{9}\\
u_{1}^{(3)}(r,\theta) & = \as_{1,3}^{(0)}r^3 \sin(3\theta) + \as_{1,3}^{(1)} r^{3} \left[\theta \cos(3\theta) + \log(r) \sin (3 \theta)\right], &  \as_{1,3}^{(0)} &= 0\\
 &  & \as_{1,3}^{(1)} &=  \frac{35}{27\pi} 
\end{align*}
 
3D plots of an approximation of $u_1$ and its derivatives in the domain $\Omega_1$ are shown in Fig.~\ref{fig:GrafSolStress_DD_OmegaPi_1Alpha-5_3}.
\begin{figure}[H]
    \centering
    \subfloat[]{\includegraphics[width=0.33\textwidth]{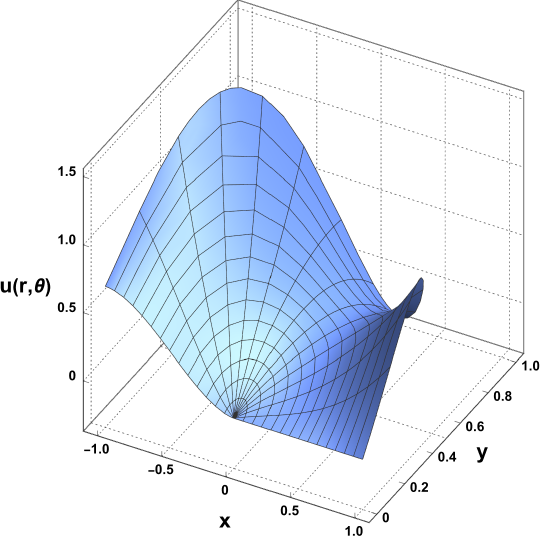}}
    \subfloat[]{\includegraphics[width=0.33\textwidth]{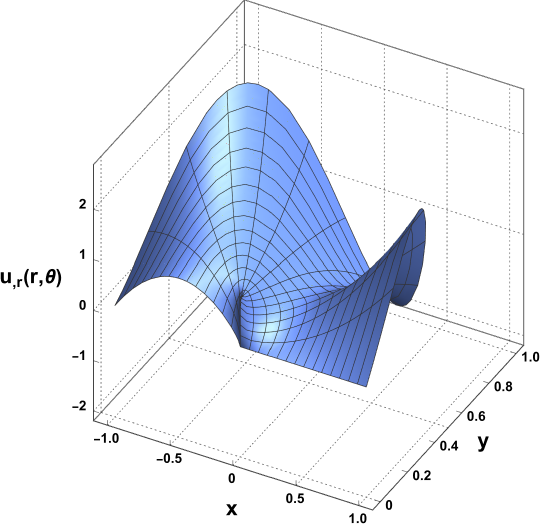}}
    \subfloat[]{\includegraphics[width=0.34\textwidth]{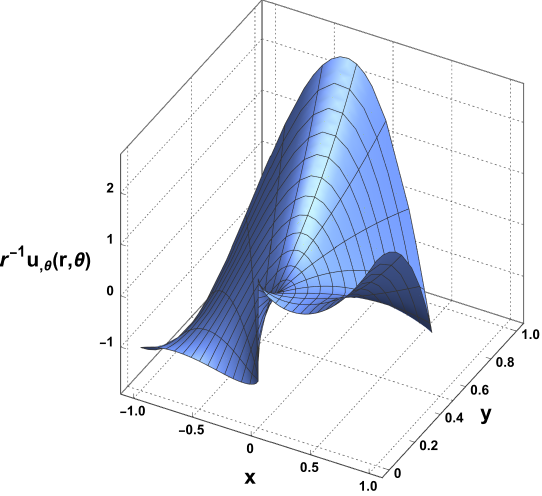}}
    \caption{3D plots of an approximation  of the eigensolution (a)  $u_1$ and its derivatives (b) $u_{1,r}$ and (c) $r^{-1} u_{1,\theta}$, for $\omega=\pi$ and $\alpha=-\frac{5}{3}$.}
    \label{fig:GrafSolStress_DD_OmegaPi_1Alpha-5_3}
 \end{figure}
 Fig.~\ref{fig:GrafErrores_DD_OmegaPi_1Alpha-5_3} shows how the relative and absolute errors in the Robin boundary condition decrease with increasing $S_1$.  
\begin{figure}[H]
    \centering
    \subfloat[]{\includegraphics[scale = 0.6]{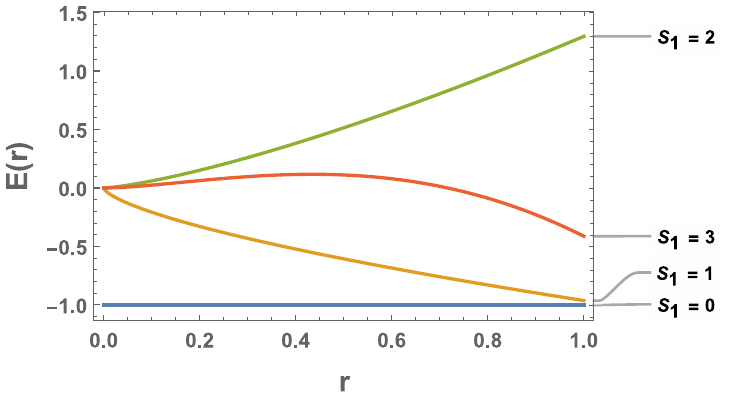}}\hspace{0.5cm}
    \subfloat[]{\includegraphics[scale = 0.6]{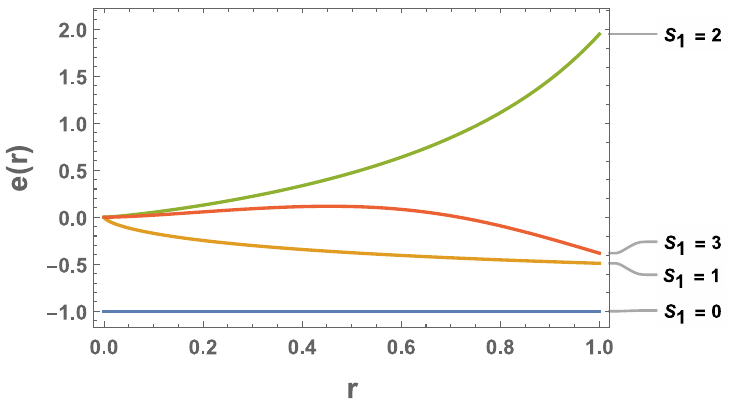}}
     \caption{Absolute and relative errors, $E_{DN}(r)$ and $e_{DN}(r)$, for $\omega=\pi$, $\alpha=-\frac{5}{3}$ and $j=1$.}
    \label{fig:GrafErrores_DD_OmegaPi_1Alpha-5_3} 
 \end{figure}
Fig.~\ref{fig:GrafSolOmStressRZOm_DD_OmegaPi_1Alpha-5_3}    shows how   $u_{1}(r, \pi)$ and $u_{1}(r, \pi)$ change  with increasing $S_1$. 
\begin{figure}[H]
    \centering
    \subfloat[]{\includegraphics[scale = 0.6]{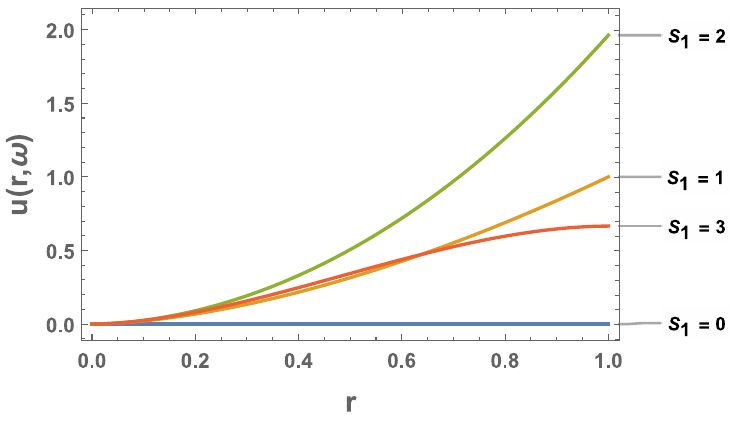}}\hspace{0.5cm}
    \subfloat[]{\includegraphics[scale = 0.6]{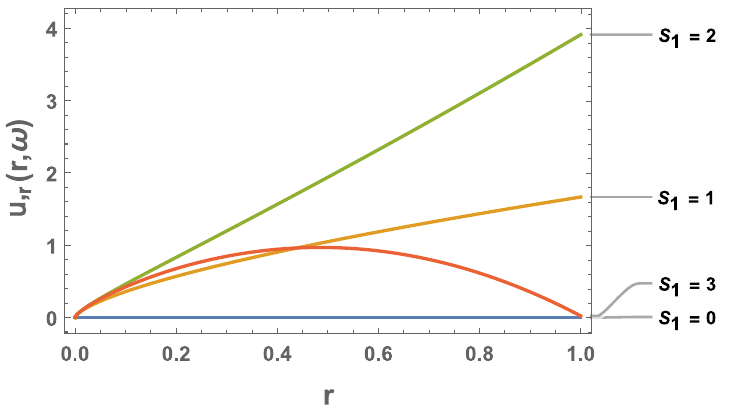}}\hspace{0.5cm}
     \caption{Plots of approximations of (a) $u_{1}(r,\pi)$ and (b) $u_{1,r}(r,\pi)$ with increasing $S_1$,  for $\omega=\pi$ and $\alpha=-\frac{5}{3}$.}  \label{fig:GrafSolOmStressRZOm_DD_OmegaPi_1Alpha-5_3} 
 \end{figure}
Fig.~\ref{fig:GrafStressTZOm0_DD_OmegaPi_1Alpha-5_3} shows how the derivative $r^{-1}u_{1,\theta}$ changes with   increasing $S_1$ on the Dirichlet and  Robin boundaries. 
 \begin{figure}[H]
  \centering
  \subfloat[]{\includegraphics[scale = 0.6]{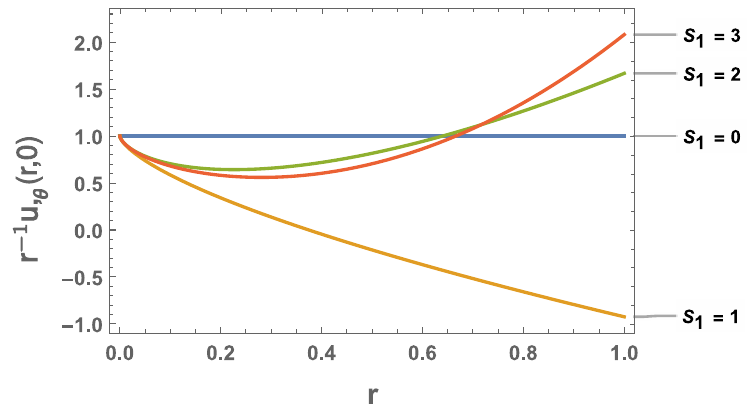}}\hspace{0.5cm}
  \subfloat[]{\includegraphics[scale = 0.6]{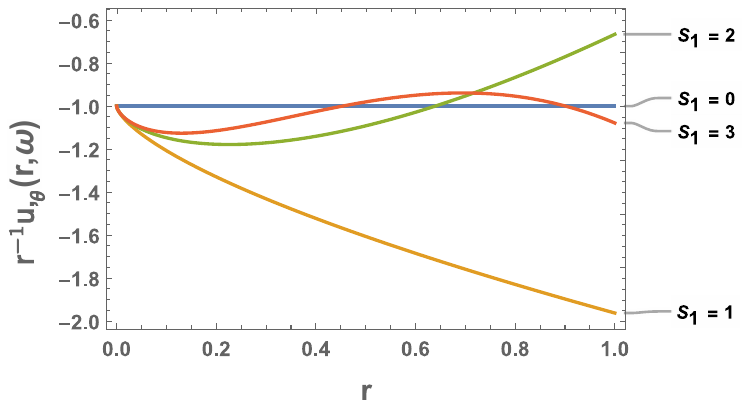}}
   \caption{Plots of approximations of (a) $r^{-1}u_{1,\theta}(r,0)$   and (b) $r^{-1}u_{1,\theta}(r,\pi)$ for  increasing $S_1$,  for $\omega=\pi$ and $\alpha=-\frac{5}{3}$.}
  \label{fig:GrafStressTZOm0_DD_OmegaPi_1Alpha-5_3} 
\end{figure}

In general, for $j\in \N$ we have 
\[ u_{j}(r,\theta) \approx  u_{j}^{(0)}(r,\theta) + u_{j}^{(1)}(r,\theta) +   u_{j}^{(2)}(r,\theta) + u_{j}^{(3)}(r,\theta)\]
where the main and shadow terms with their respective coefficients are
\begin{align*}
  u_{j}^{(0)}(r,\theta) & = \as_{j,0}^{(0)} r^j \sin (\theta  j), \hspace{2.8cm} \as_{j,0}^{(0)}   =  1,    \\  
  u_{j}^{(1)}(r,\theta) & = \as_{j,1}^{(0)} r^{j+\frac{2}{3}} \sin \left(\theta  \left(j+\frac{2}{3}\right)\right), \qquad \as_{j,1}^{(0)}   =  -\frac{2 j}{\sqrt{3}},  \\  
  u_{j}^{(2)}(r,\theta) & = \as_{j,2}^{(0)} r^{j+\frac{4}{3}} \sin \left(\theta  \left(j+\frac{4}{3}\right)\right), \qquad \as_{j,2}^{(0)}  =  \frac{2}{9}(3j^2 + 2 j), \\   
  u_{j}^{(3)}(r,\theta) & = r^{j+2} \left(\as_{j,3}^{(0)}\sin (\theta  (j+2)) + \as_{j,3}^{(1)}(\theta  \cos (\theta  (j+2)) + \log(r) \sin (\theta  (j+2))) \right), \\
   & \quad\ \as_{j,3}^{(0)}  = 0,   \qquad \as_{j,3}^{(1)}  =  \frac{j(3 j+2) (3 j+4)}{27\pi}  .
\end{align*}

  \item  For $\omega = \pi/2$ and  $\alpha = 1/2$ we get   according to  Table \ref{TblDD2}  
  \[ \frac{\omega(\alpha + 1) }{\pi} =  \frac{3}{4} = \frac{2p-1}{2q}, \quad \text{with } p = 2 \text{ and } q = 2.\]

In this case  we   choose $j \geq p$, thus, we obtain the exact solution without log terms and $q$ shadow terms. Thus, $(\omega,\alpha)$ is an apparent critical pair. The solution for $j = 2$   given by
  \[u_{2}(r,\theta) = 4 \sqrt{2} r^{5/2} \sin \left(\frac{5 \theta }{2}\right) + r^4 \sin\  (4 \theta )+10 r \sin (\theta ),\]
is composed by the following main and shadow terms with their respective coefficients 
\begin{align*}
u_{2}^{(0)}(r,\theta) & = \as_{2,0}^{(0)} r^4 \sin (4\theta ), &  \as_{1,0}^{(0)}  &=  1,\\
u_{2}^{(1)}(r,\theta) & = \as_{2,1}^{(0)} r^ {5/2} \sin \left(\frac{5 \theta }{2}\right),&  \as_{2,1}^{(0)} &= 4\sqrt{2},\\
u_{2}^{(2)}(r,\theta) & = \as_{2,2}^{(0)} r \sin (\theta ),&  \as_{2,2}^{(0)} &= 10.\\
\end{align*}

3D plots of  $u_2$ and its derivatives are shown in Fig.~\ref{fig:GrafSolStress_DD_OmegaPi_2Alpha1_2}.%
Fig.~\ref{fig:GrafErrores_DD_OmegaPi_2Alpha1_2} shows how the relative and absolute errors in the Robin boundary condition decrease with increasing $S_2$.  
\begin{figure}[H]
    \centering
    \subfloat[]{\includegraphics[width=0.33\textwidth]{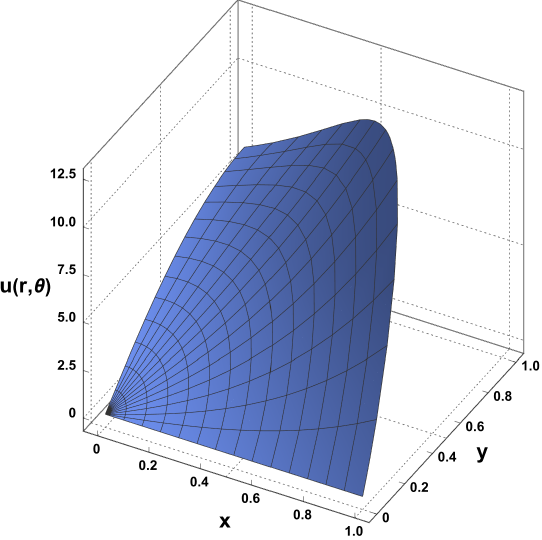}}
    \subfloat[]{\includegraphics[width=0.33\textwidth]{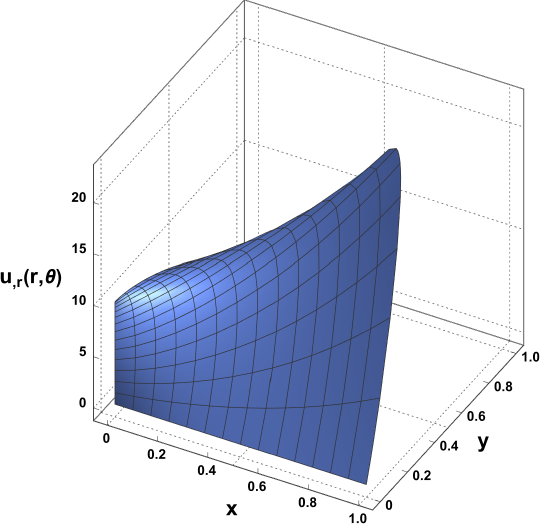}}
    \subfloat[]{\includegraphics[width=0.34\textwidth]{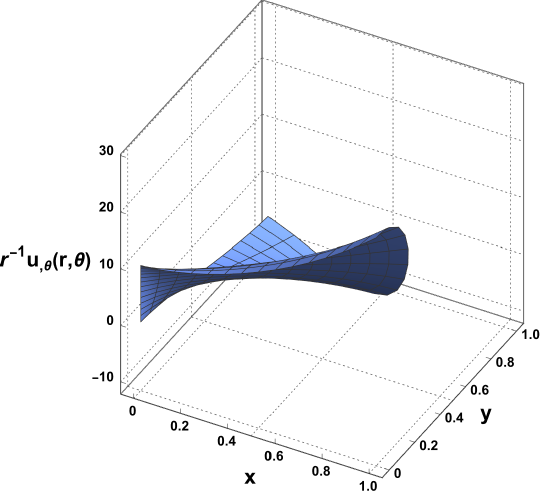}}
    \caption{3D plots of the eigensolution (a)  $u_2$ and its derivatives (b) $u_{2,r}$ and (c) $r^{-1} u_{2,\theta}$, for $\omega=\pi/2$ and $\alpha=\frac{1}2$.}
\label{fig:GrafSolStress_DD_OmegaPi_2Alpha1_2}
 \end{figure}
 \begin{figure}[H]
  \centering
  \subfloat[]{\includegraphics[scale = 0.6]{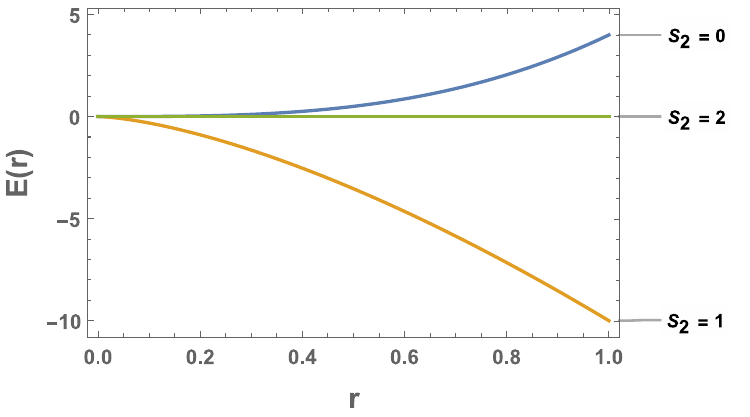}}\hspace{0.5cm}
  \subfloat[]{\includegraphics[scale = 0.6]{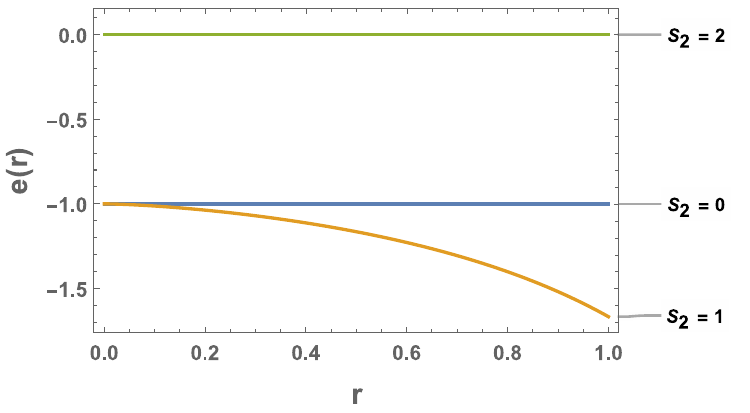}}
   \caption{Absolute and relative errors, $E_{DN}(r)$ and $e_{DN}(r)$, for  $\omega=\pi/2$, $\alpha=\frac{1}{2}$ and $j = 2$.}
  \label{fig:GrafErrores_DD_OmegaPi_2Alpha1_2} 
\end{figure}
Fig.~\ref{fig:GrafSolOmStressRZOm_DD_OmegaPi_2Alpha1_2}  shows how the approximations of the eigensolution  $u_{2}(r, \pi/2)$  and its derivative  $u_{2,r}(r, \pi/2)$ change with increasing $S_2$.  
\begin{figure}[H]
  \centering
  \subfloat[]{\includegraphics[scale = 0.6]{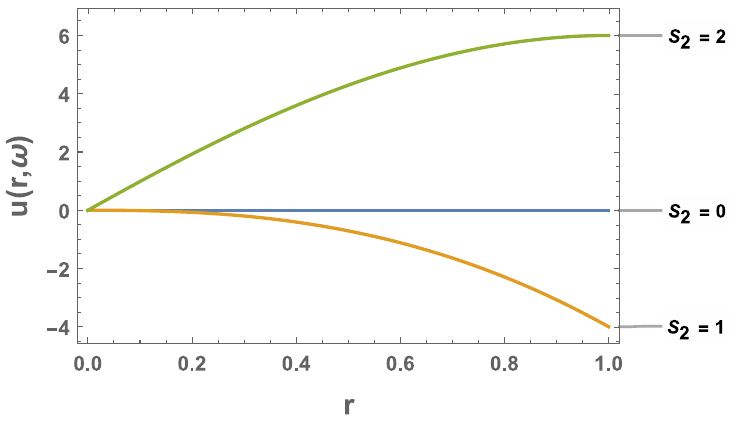}}\hspace{0.5cm}
  \subfloat[]{\includegraphics[scale = 0.6]{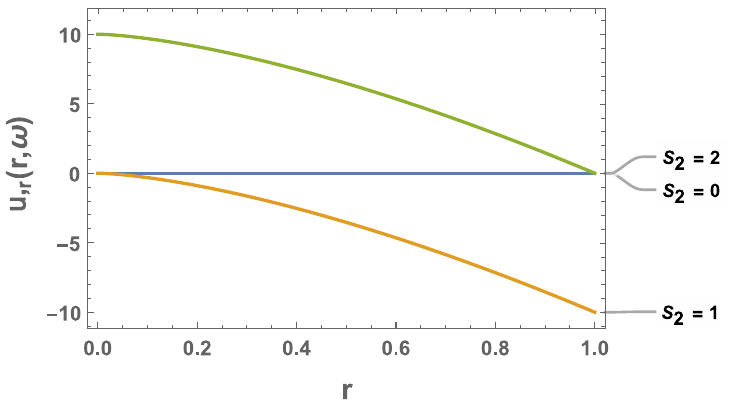}}\hspace{0.5cm}
   \caption{Plots of approximations of (a) $u_{2}(r,\pi/2)$ and (b) $u_{2,r}(r,\pi/2)$ with increasing $S_2$, for  $\omega=\pi/2$ and $\alpha=\frac{1}{2}$.}
\label{fig:GrafSolOmStressRZOm_DD_OmegaPi_2Alpha1_2} 
\end{figure}
Fig.~\ref{fig:GrafStressTZOm0_DD_OmegaPi_2Alpha1_2} shows how the approximations of the derivative $r^{-1}u_{2,\theta}(r,0)$ and  $r^{-1}u_{2,\theta}(r,\pi/2)$ change with $S_2$.  
 Note that for $r^{-1}u_{2,\theta}(r,\pi/2)$   the solutions   overlap for $S_{1}$ and $S_{2}$.

\begin{figure}[H]
\centering
\subfloat[]{\includegraphics[scale = 0.6]{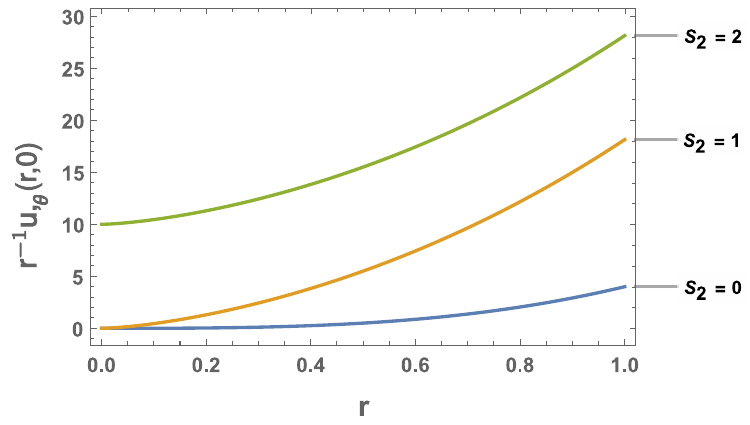}}\hspace{0.5cm}
\subfloat[]{\includegraphics[scale = 0.6]{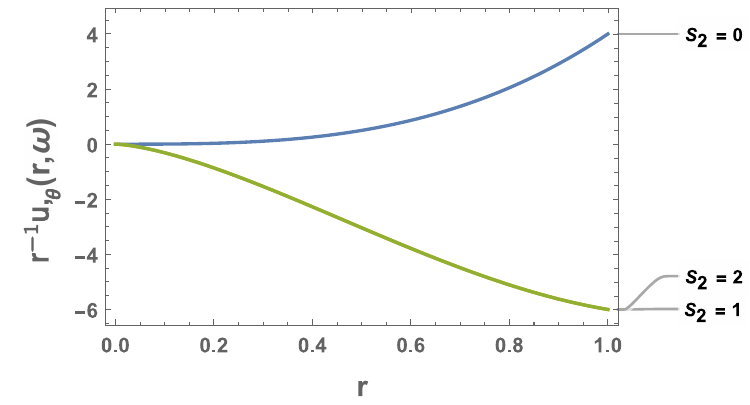}}
 \caption{Plots of approximations of (a) $r^{-1}u_{2,\theta}(r,0)$   and (b) $r^{-1}u_{2,\theta}(r,\pi/2)$ with  increasing $S_2$,  for  $\omega=\pi/2$ and $\alpha=\frac{1}{2}$.}
\label{fig:GrafStressTZOm0_DD_OmegaPi_2Alpha1_2}
\end{figure}
 \end{itemize}
 
 In general, for $j\in \N$ we have 
 \[ u_{j}(r,\theta) =  u_{j}^{(0)}(r,\theta) + u_{j}^{(1)}(r,\theta) +   u_{j}^{(2)}(r,\theta) + u_{j}^{(3)}(r,\theta),\]
 where the main and shadow terms with their respective coefficients are
  \begin{align*}
   u_{j}^{(0)}(r,\theta) & = \as_{j,0}^{(0)}  r^{2 j} \sin (2 \theta  j),   & \as_{j,0}^{(0)}  & =  1, \\
   u_{j}^{(1)}(r,\theta) & = \as_{j,1}^{(0)}  r^{2 j-\frac{3}{2}} \sin \left(\theta  \left(2 j-\frac{3}{2}\right)\right), &  \as_{j,1}^{(0)}  & = 2 \sqrt{2} j, \\
   u_{j}^{(2)}(r,\theta) & = \as_{j,2}^{(0)} r^{2 j-3} \sin (\theta  (2 j-3)), &   \as_{j,2}^{(0)}  & = 4 j^2-3 j  .   
 \end{align*}

  To compute the eigensolution energy in a neighbourhood of the corner tip   we apply  \eqref{energy} giving 
 \begin{align*}
 \int_{\epsilon}^{R} r^{\alpha + 1} |u(r,\omega)|^2 \df r & =2 j^2 r^{4 j- 7/2} \left(\frac{(3-4 j)^2}{8 j-7} + r^{3/2}\frac{3-4j}{2j- 1} + r^3\frac{4}{4j - 1}\right)\bigg. \bigg|_{\epsilon}^{R},
 \end{align*}
 and
 \begin{align*}
   \int_{\epsilon}^{R}\int_{0}^{\omega} & r\left( \frac{\partial u(r,\theta) }{\partial r}\right)^2 + \frac{1}{r}\left( \frac{\partial u(r,\theta) }{\partial \theta}\right)^2 \df \theta \df r   = \frac{1}{12} j r^{4 j} \left(\frac{16 j \left((3-2 j) r^{3/2} + \frac{4 (4 j-3) r^3}{8 j-3}+\frac{2 (2 j-3) (3-4 j)^2}{8 j-9}\right)}{r^{9/2}}\right. \\
     & + \  \left. \frac{3 \pi  \left(4 j (4 j-3) r^3+j (2 j-3) (3-4 j)^2+2 r^6\right)}{r^6}\right)\bigg. \bigg|_{\epsilon}^{R}.
 \end{align*}
 From the integrals above, it is easy to see that the energy is finite for $j \geq 2$.  In the case $j = 1$, the power of $r$ is negative and when $\epsilon\rightarrow 0$ the obtained energy is infinite, which is in agreement with Table~\ref{TblDD2}.

\subsection{Graphics for \texorpdfstring{$\alpha = -1$}{a = -1}}
Consider, as an example,  the inner corner angle  $\omega = \pi/2$. To find the roots of \eqref{trascendental}   the command \verb|FindRoot| of Mathematica \cite{Wolfram1991} is used.  3D plots of the eigensolution $u_1$ and its derivatives for $\gamma = 1/2$ are shown in   
 Fig.~\ref{fig:GrafSolStress_DR_OmegaPi_2Alpha-1}.
\begin{figure}[H]
  \centering
  \subfloat[]{\includegraphics[width=0.33\textwidth]{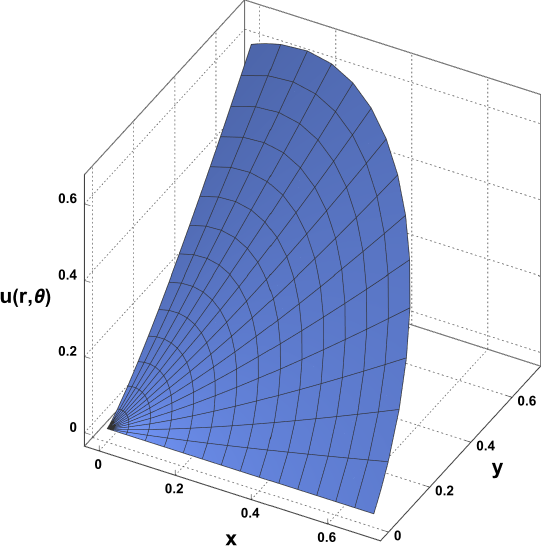}}
  \subfloat[]{\includegraphics[width=0.33\textwidth]{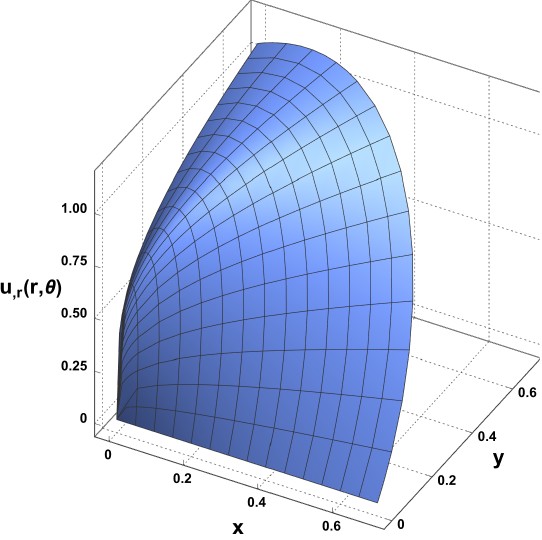}}
  \subfloat[]{\includegraphics[width=0.33\textwidth]{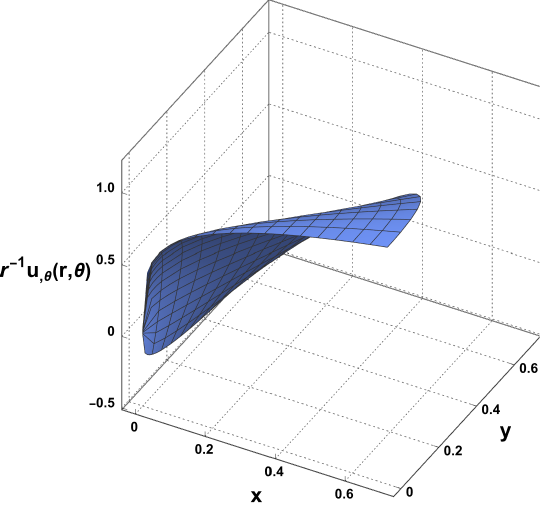}}
  \caption{3D plots of the eigensolution (a)  $u_1$ and its derivatives (b) $u_{1,r}$ and (c) $r^{-1} u_{1,\theta}$ associated to $\lambda_{1}$, for $\omega=\pi/2$, $\alpha = -1$, and $\gamma = 1/2$.
 }
  \label{fig:GrafSolStress_DR_OmegaPi_2Alpha-1}
\end{figure}

The first five roots of the transcendental eigenequation~\eqref{trascendental} calculated for several values of $\gamma$ are summarised in Table \ref{Tabalpha-1}, and the first three of them are also  shown in Fig. \ref{Figalpha-1}.
\begin{table}[H]
	\begin{center}
\begin{tabular}{| c | c || c || c |}
	\hline 
            & $\gamma =  1$   &  $\gamma = 2 $  &  $\gamma = 1/2$ \\
  $j\in \N$ & $\lambda_{j}$   &  $\lambda_{j} $  &  $\lambda_{j} $ \\
  \hline  & & & \\[-1.5 mm]
   $1$ & $1.395773844$ &  $1.575274586$ &  $1.243401927$ \\
   $2$ & $3.193207935$ &  $3.343211251$ &  $3.101747463$ \\
   $3$ & $5.122730124$ &  $5.232427165$ &  $5.062670665$ \\
   $4$ & $7.089212594$ &  $7.173105492$ &  $7.045106072$ \\
   $5$ & $9.069907943$ &  $9.137183491$ &  $9.035194103$\\
   \hline
\end{tabular}
\caption{Five first roots of $\tan(\lambda_{j}\omega) + \frac{\lambda_{j}}{\gamma} = 0$ . }
\label{Tabalpha-1}
\end{center}
\end{table}
\begin{figure}[H]
  \centering
  \includegraphics[scale = 0.6]{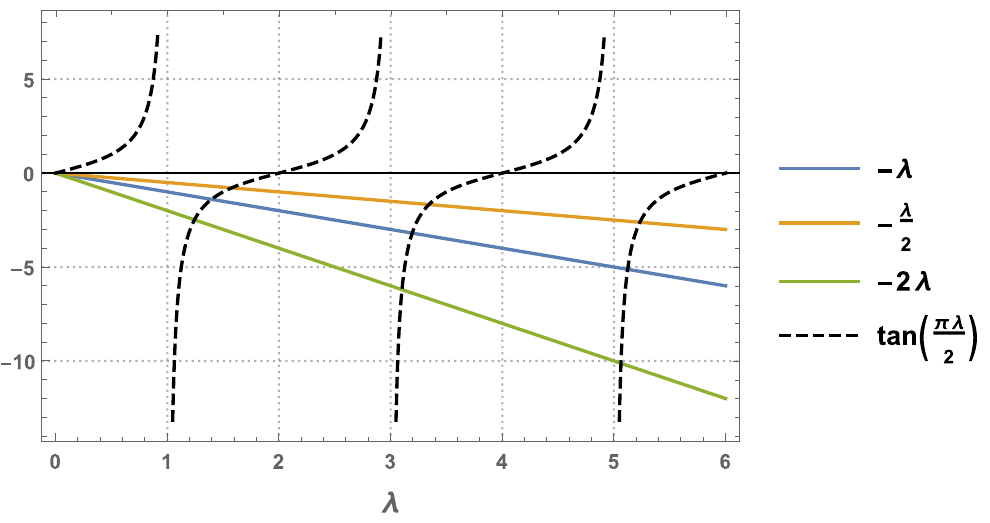}
  \label{fig:GrafFuncionTranscendente} 
  \caption{Graph of intersection between $\tan(\lambda\omega )$ and $-\lambda/\gamma$, for several values of $\gamma$. }
  \label{Figalpha-1}
\end{figure}
We can observe that the relationship $(2j-1)\frac{\pi}{2\omega} < \lambda_{j} < j\frac{\pi}{\omega}$ in~\eqref{trascendental} is fulfilled. 

\section{An application to fracture mechanics. A bridged crack problem}

The following application of the previous results to fracture mechanics is related to the original motivation of the present work as described in Section~\ref{sec:DR_definition}. Singularities in the derivatives of  displacement $u$, analysed in the previous sections, determine also singularities in stresses, by means of the linear elastic constitutive law. Notably, stress singularities are a critical feature in fracture mechanics studies.
In the following we will focus on a model problem of a whole plane with a semi-infinite crack bridged
by a continuous distribution of linear (Winkler) springs with power-law variation of spring
stiffness in antiplane mode, also called Mode III. As schematically indicated in Fig.~\ref{fig:BridgedCrack}, such a crack problem, can be reduced to a D-R problem in a half-plane. Thus, in the present notation, we  consider a D-R corner problem with $\omega=\pi$ and $K(r) = K_0 \left(\frac{r}{a}\right)^{\alpha}= \kappa r^\alpha$ with $\alpha\in\mathbb{R}$. 
\begin{figure}[H]
\centering
 \includegraphics[width=1\textwidth]{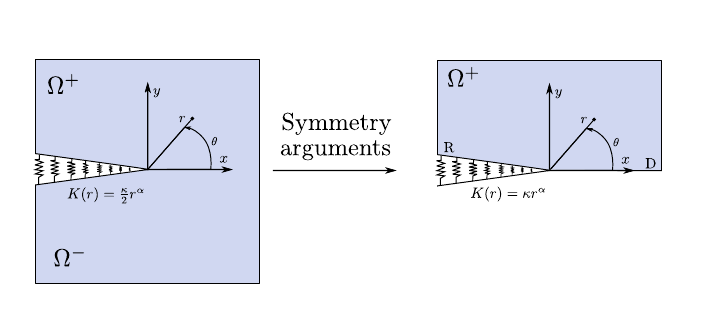} 
\caption{Mode III bridged crack problem reduces to a D-R problem for half-plane: $\omega = \pi$ and $K(r) = \kappa r^\alpha$.}
\label{fig:BridgedCrack}
\end{figure}

The stress singularity analysis for Mode III bridged cracks with power-law variation of spring stiffness is performed, identifying three characteristic regimes:
\begin{itemize}
\item For $\alpha>-1$. The recursive D-N procedure  converges, yielding singularity exponents (eigenvalues) 
\begin{equation*}
\lambda_j=(2j-1)\frac{\pi}{2\omega}=j-\frac{1}{2}.
\end{equation*}
Thus, for $j=1$, $\lambda_1 = \frac{1}{2}$. The corresponding singular eigensolution is
\begin{equation*}
u_1(r,\theta)=\sqrt{r} \sin(\theta/2 )+\text{more regular shadow terms}=O(\sqrt{r}).
\end{equation*} 
This result represents the \textit{classical square root stress singularity at the crack tip} as considered in Linear Elastic Fracture Mechanics (LEFM).

\item For $\alpha<-1$, the recursive D-D procedure  converges, leading to
\begin{equation*}
\lambda_{j} = j\frac{\pi}{\omega}=j. 
\end{equation*} 
Thus, for $j=1$, $\lambda_1=1$. The corresponding eigensolution takes de form
\begin{equation*}
  u_1(r,\theta)=r\sin(\theta )+\text{more regular shadow terms}=O(r).
\end{equation*}
In this case, the stresses at the bridged-crack tip are continuous with no stress singularity.

\item For $\alpha=-1$, a closed form  solution is available, but the singularity exponent $\lambda_{j}$ is determined by numerically solving a transcendental equation
\begin{equation}
  \tan(\lambda_{j}\pi) + \frac{\lambda_{j}}{\gamma} = 0 \quad  \text{with }   (2j-1)\frac{\pi}{2\omega}=j-\frac{1}{2}< \lambda_{j} < j\frac{\pi}{\omega}=j.
\end{equation}
Thus, for $j=1$, $\frac{1}{2}<\lambda_1(\gamma)<1$.  Therefore, the corresponding eigensolution takes de form
\begin{equation*}
u_{1}(r,\theta) = r^{\lambda_{1}} \sin(\lambda_{1} \theta)=O(r^{\lambda_{1}} ).
\end{equation*}
This represents a \textit{weak stress-singularity at the bridged-crack tip.} This particular case was previously studied by Ueda et al. \cite{Ueda2006}.
\end{itemize}

In summary, the analysis determines three characteristic regimes based on the value $\alpha$, see Fig.~\ref{StressSingBridged}.
 \begin{itemize}
  \item  $\alpha>-1$ with \textit{classical crack-tip  stress singularity}
  \item $\alpha=-1$ with \textit{weak  stress singularity at the crack tip}  
  \item  $\alpha<-1$ with \textit{no stress singularity at the crack tip, with  continuous stresses}. 
  \end{itemize}

  \begin{figure}[H]
    \includegraphics[width=1\textwidth]{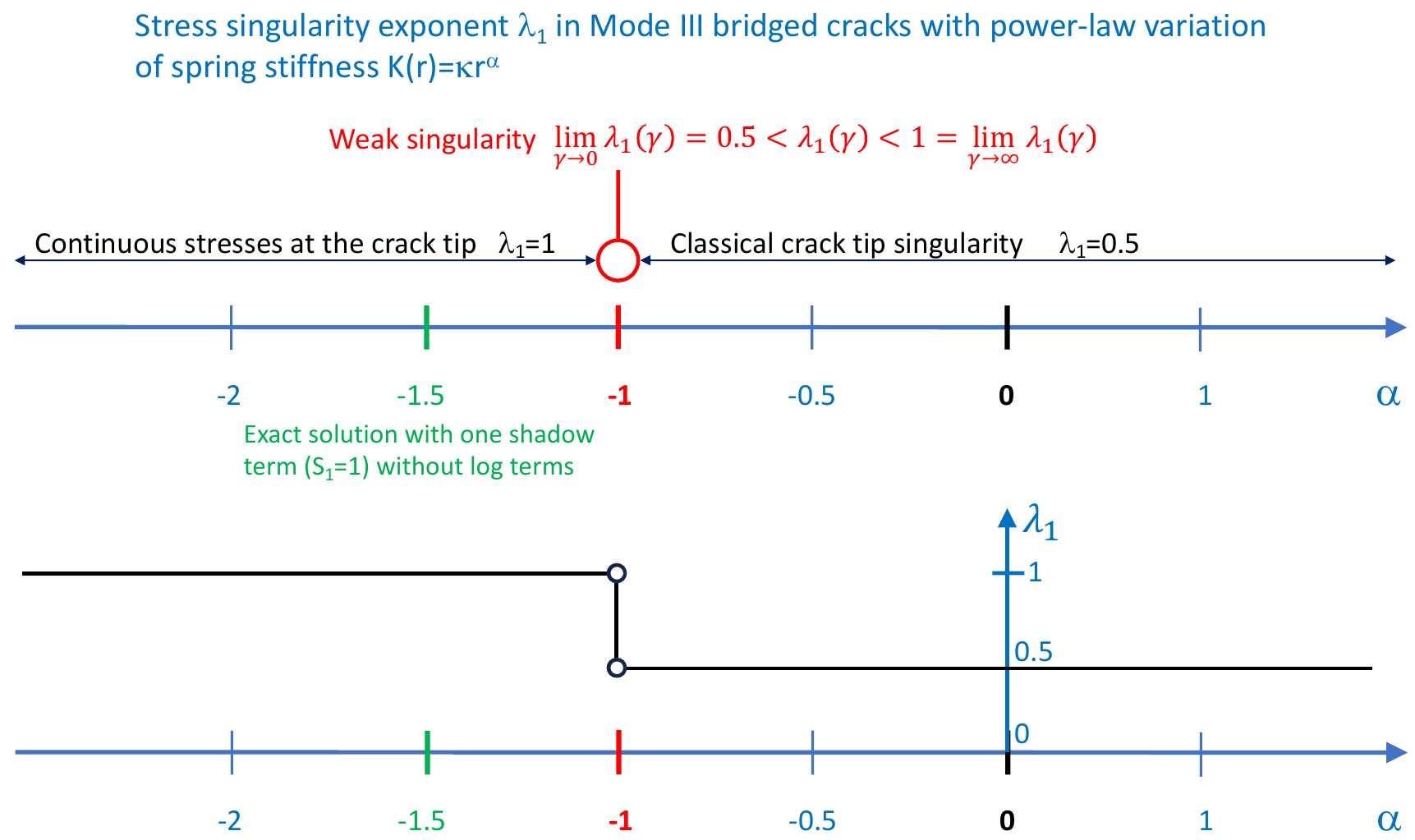}
    \caption{Overview of stress singularities in bridged cracks.}
    \label{StressSingBridged}
   \end{figure}
  
An example for Mode III bridged crack with $\alpha=-\frac{3}{2}$ is studied  by the recursive D-D approach in Section \ref{subsec:GraphDD}, where the expression of the eigensolution $u_1$ in  \eqref{ex:DD_omega_pi_alpha-3/2} yields a continuous traction along the $x$-axis, cf. \eqref{ex:DN_omega_pi_alpha-3/2}, 
\begin{equation}
\sigma_{yz}(x,y=0)  = \begin{cases}
 1-\frac{3}{2} \sqrt{x} & \text{  for  }   x\geq 0,\\
 1 & \text{  for  }  x\leq 0.
\end{cases}
\end{equation} 
A similar study for $\alpha=-\frac{5}{3}$ 
provides an approximation of the eigensolution $u_1$ in  \eqref{DDu1-5/3}, leading to analogous conclusions.

\section{Concluding remarks}
An original recursive methodology is developed to compute singular eigensolutions of corner problems for the Laplace equation with the homogeneous Dirichlet boundary condition and the homogeneous Robin boundary condition with power-law variation of its coefficient. A key advantage of this methodology is its simplicity, since it does not need advanced mathematical concepts, and its suitability for implementation in computer algebra software. The latter feature is especially relevant for applications of the present results in physics and engineering, in view of the structure of these singular eigensolutions: the sum of a main term and a finite or infinite series of the associated higher-order terms whose expressions can be long and complicated. 

The present recursive methodology is general in the sense that it covers the full ranges of both the inner angle of corner domain   $0<\omega\leq 2\pi$ and the exponent in power-law  variation of the coefficient in the Robin boundary condition  $\alpha\in\R$. Some previous works covered particular cases, such as Ueda et al. \cite{Ueda2006} derived singular eigensolutions for  $\omega=\pi$ and $\alpha=-1$, and Jimenez-Alfaro et al. \cite{Jimenez-Alfaro2020} derived singular eigensolutions for  $0<\omega\leq 2\pi$ but only for $\alpha=0$.

The present approach determines for any combination of $\omega$ and $\alpha$ whether the series of shadow terms is finite or infinite, and in the case that it is finite which is the number of shadow terms, and also if the shadow terms include just power-law terms or also power-logarithmic terms and how many of them, see Tables~\ref{TblDN1} and~\ref{TblDD1}.


Once the double asymptotic series of the main and the associated shadow terms for the D-R corner problems is constructed and the critical pairs $(\omega,\alpha)$  are identified (see Tables~\ref{TblDN1} and~\ref{TblDD1}), the question of how to construct stable asymptotic expansions for the present problem when the corner angle $\omega$ varies can be posed, similar  to the stable asymptotic expansions introduced in \cite{MazyaRossmann1992} for the D-D corner problems.

The present approach and results might  be applied in the domain decomposition methods to solve BVPs by using Robin interface condition, and in the numerical solution of contact problems by using penalty method.
    
\section{Acknowledgments}%
This research  was partially supported by the Spanish Ministry of Science and Innovation and
European Regional Development Fund (PID2021-123325OB-I00).
The research of the first author (NP-L) was funded under Grant QUALIFICA (PROGRAMA:\@ AYUDAS A ACCIONES COMPLEMENTARIAS DE I+D+i) by Junta de Andalucía grant number QUAL21 005 USE and the funding received from the European Union’s Horizon 2020 research and innovation programme under the Marie Sklodowska-Curie grant agreement No 101034297 – Project Energy for Future.
The third author (SJ-A) acknowledges the Iberdrola Foundation under the Marie Sklodowska-Curie Grant Agreement No 101034297.


\appendix

\section{Basic results}\label{AppA}
\begin{rmkAppndx}
When we consider an irreducible fraction, $\frac{p}{q}$, then we have three possibilities to write this irreducible fraction, i.e, there are $p',q' \in \N$ such that $\frac{p}{q}$ has the following form:
\[\frac{2p'-1}{2q'},\qquad\frac{2p'}{2q'-1},\quad \text{or}\quad \frac{2p'-1}{2q'-1}\]
which we can reduce to two in following form
\[\frac{2p'-1}{2q'},\quad \text{or} \quad \frac{p'}{2q'-1}. \] 
\end{rmkAppndx}

\begin{prpAppndx}\label{PrpB0}
Let $k,q, p\in\N$, such that $\frac{2p-1}{2q}$ is an irreducible fraction. Then, $\sin(\frac{k}{q}(2p-1)\frac{\pi}{2})\neq 0$,  for all $k\leq q$. 
\end{prpAppndx}
\begin{proof}
For $k = q$ we have $\sin((2p-1)\frac{\pi}{2})= (-1)^{p-1}\neq 0$. For the case $k < q$ with $q \geq 2$, we note that $\frac{2p-1}{q}$ is irreducible and $k$ is not a multiple of $q$, then $\frac{k}{2q}(2p-1)\in\Q\setminus \N$. Therefore, $\sin\left(\frac{k}{q}(2p-1)\frac{\pi}{2}\right)\neq 0$.
\end{proof}

\begin{prpAppndx}\label{PrpB1}
	Let $j,k\in \N$ and $\alpha \in \R$. If $(2j-1)\frac{\pi}{2\omega} + k(\alpha +1) = 0,$ then $\sin(\omega k(\alpha + 1)) \neq 0$. 
\end{prpAppndx}

\begin{proof}
It holds $\omega k(\alpha + 1 ) = -(2j-1)\frac{\pi}{2} $. Then, $\sin(\omega k(\alpha + 1))= \sin(-(2j-1)\frac{\pi}{2})= (-1)^{j+1}\neq 0$ for all $j\in\N$. 
\end{proof}

\begin{prpAppndx}\label{PrpB2}
	Let $j,k\in \N$ and $\alpha \in \R$. If $(2j-1)\frac{\pi}{2\omega} + k(\alpha +1) = 0,$ then $\cos(\omega(k-1)(\alpha + 1)) \neq 0$  
\end{prpAppndx}
\begin{proof}
It holds $ \omega k(\alpha +1) =-(2j-1)\frac{\pi}{2}$.  Then,
\[\cos(\omega(k-1)(\alpha + 1)) = \cos((2j-1)\frac{\pi}{2} + \omega(\alpha + 1)) = -(-1)^{j+1}\sin(\omega(\alpha+1)) =  (-1)^{j} \sin\left(-\frac{(2j-1)}{2k} \pi\right) \neq 0,\] for all $j,k\in \N$, because  $\frac{2j-1}{2k}$ is an irreducible fraction.   
\end{proof}

\begin{prpAppndx}\label{PrpB3}
	Let $a\in \R$ and $n\in\N_{0}$. Then, 
  \begin{equation*}
      \int r^a \log^n(r)\, dr = 
\begin{cases}
\displaystyle
r^{a+1} \sum_{j = 0}^{n} (-1)^j \frac{n!}{(n-j)!} \frac{\log^{n-j}(r)}{(a+1)^{j+1}} + C, & \text{if } a \neq -1, \\[1.2em]
\displaystyle
\frac{\log^{n+1}(r)}{n+1} + C, & \text{if } a = -1.
\end{cases}   
  \end{equation*}  
  where $C$ is a constant. 
\end{prpAppndx}

\begin{proof}
	The proof follows from integration by parts $n$ times. 
\end{proof}

\begin{prpAppndx}\label{PrpB4}
	The series $\sum_{n = 0}^{\infty} (r\log(r))^{n}$ converges absolutely, at least for $0< r < 1$.
\end{prpAppndx}

\begin{proof}
	The partial sums are given by $\sum_{n = 0}^{m} (r|\log(r)|)^{n} =  \frac{1-(r|\log(r)|)^m}{1-r|\log(r)|}$. Taking the limit over $m$ we deduce that 
	\[ \lim_{m\to \infty}r^m |\log(r)|^m  = 0, \quad\text{with}\quad 0< r <1,\] and, then, our result. 
\end{proof}

\section{Recursive structure of the linear systems}\label{AppB}
We rewrite the linear systems in~\eqref{SystDN} and~\eqref{SystDD}.  For this, we begin from the equalities~\eqref{SumSystDN} and~\eqref{SumSystDD}. Since the calculations are similar, we will focus on the D-N case. Using the angle sum on the right-hand side of~\eqref{SumSystDN} it holds:
\begin{align*}
	&\sum_{m=0}^{L_{j,k}} \log^{m}(r) \sum_{l=m}^{L_{j,k}} a_{j,k}^{(l)} \binom{l}{m} \omega^{l-m} \left[ \frac{(l-m)}{\omega} \cos\left(  \omega k(\alpha + 1) + \frac{\pi}{2} (l-m)\right) \right. \\
	& \hspace{5cm} \left. - (\lambda_{j} + k(\alpha + 1))\sin\left( \omega k(\alpha + 1) + \frac{\pi}{2} (l-m)\right)\right] \\
	& = -\gamma \sum_{m=0}^{L_{j,k-1}} \log^{m}(r)  \sum_{l=m}^{L_{j,k-1}} a_{j,k-1}^{(l)} \binom{l}{m}  \omega^{l-m} \left[ \cos(\omega(\alpha +1)) \cos\left(  \omega k(\alpha + 1) + \frac{\pi}{2} (l-m)\right) \right. \\ 
   & \hspace{6cm} + \ \left. \sin(\omega(\alpha +1))\sin\left( \omega k(\alpha + 1) + \frac{\pi}{2} (l-m)\right)\right].
\end{align*}
For each $j$ and $k$ fixed we can interpret, the previous polynomial sum, as a system
\begin{align}\label{RrcSystDN}
\left( \frac{1}{\omega}\bm{\tilde {R}}_{j,k} - (\lambda_{j}  + k(\alpha + 1)) \bm{N}_{j,k} \right) \bm{a}_{j,k} = -\gamma\left(\cos(\omega(\alpha +1)) \bm{R}_{j,k} \frac{}{} + \frac{}{} \sin(\omega(\alpha +1)) \bm{N}_{j,k} \right) \bm{a}_{j,k-1},
\end{align}
where $\bm{a}_{j,k}$ is defined as in the system~\eqref{SystDN},
\begin{align*}
	 \bm{R}_{j,k} &= \left[ \rho_{j,k}^{(m,l)}  \right]_{m = 0, \ldots, L_{j,k};\ l = m,\ldots, L_{j,k} }	 = 
	\begin{pmatrix} 
		\rho_{j,k}^{(0,0)} & \rho_{j,k}^{(0,1)}  & \cdots & \rho_{j,k}^{(0,L_{j,k})}  \\[1.5mm] 
		        0         & \rho_{j,k}^{(1,1)}  & \cdots & \rho_{j,k}^{(1,L_{j,k})}  \\[1.5mm] 
		        \vdots    &  \ddots            & \ddots & \vdots                   \\[1.5mm] 
				0		  &   0                & \cdots & \rho_{j,k}^{(L_{j,k},L_{j,k})}
	\end{pmatrix},\\	
	\bm{\tilde{R}}_{j,k} &= \left[ \tilde{\rho}_{j,k}^{(m,l)}  \right]_{m = 0, \ldots, L_{j,k};\ l = m,\ldots, L_{j,k} }	 = 
	\begin{pmatrix} 
		\tilde{\rho}_{j,k}^{(0,0)} & \tilde{\rho}_{j,k}^{(0,1)}  & \cdots & \tilde{\rho}_{j,k}^{(0,L_{j,k})}  \\[1.5mm] 
		        0         & \tilde{\rho}_{j,k}^{(1,1)}  & \cdots & \tilde{\rho}_{j,k}^{(1,L_{j,k})}  \\[1.5mm] 
		        \vdots    &  \ddots            & \ddots & \vdots                   \\[1.5mm] 
				0		  &   0                & \cdots & \tilde{\rho}_{j,k}^{(L_{j,k},L_{j,k})}
	\end{pmatrix},\\	
	\bm{N}_{j,k} &= \left[ \nu_{j,k}^{(m,l)}  \right]_{m = 0, \ldots, L_{j,k};\ l = m,\ldots, L_{j,k} }	 = 
	\begin{pmatrix} 
		\nu_{j,k}^{(0,0)} & \nu_{j,k}^{(0,1)}  & \cdots & \nu_{j,k}^{(0,L_{j,k})}  \\[1.5mm] 
		        0         & \nu_{j,k}^{(1,1)}  & \cdots & \nu_{j,k}^{(1,L_{j,k})}  \\[1.5mm] 
		        \vdots    &  \ddots            & \ddots & \vdots                   \\[1.5mm] 
				0		  &   0                & \cdots & \nu_{j,k}^{(L_{j,k},L_{j,k})}
	\end{pmatrix}
\end{align*}
with  
\begin{align*}
\rho_{j,k}^{(m,l)} & =\binom{l}{m} \omega^{l-m} \cos\left(\omega k(\alpha + 1) + \frac{\pi}{2} (l-m)\right)\\ 
\nu_{j,k}^{(m,l)} & =\binom{l}{m} \omega^{l-m}  \sin\left( \omega k(\alpha + 1) + \frac{\pi}{2} (l-m)\right) \\
\tilde{\rho}_{j,k}^{(m,l)} &= (l-m)\rho_{j,k}^{(m,l)}
\end{align*}
and \[\bm{a}_{j,k-1} = \left[ a_{j,k-1}^{(0)}, a_{j,k-1}^{(1)}, \ldots, a_{j,k-1}^{(L_{j,k-1})}  \right]^\top,\quad \text{when}\quad L_{j,k} = L_{j,k-1}\] or 
\[\bm{a}_{j,k-1} = \left[ a_{j,k-1}^{(0)}, a_{j,k-1}^{(1)}, \ldots, a_{j,k-1}^{(L_{j,k-1})}, 0  \right]^\top,\quad \text{when} \quad L_{j,k} = L_{j,k-1} + 1;\] which agrees with the conditions given by~\eqref{CondDN_1} and~\eqref{CondDN_2}. 
Analogous calculations can be made for the case D-D, obtaining
\begin{align}\label{RrcSystDD}
	\bm{\Ms}_{j,k} \bm{\as}_{j,k} =  - &\frac{\cos(\omega(\alpha  + 1))}{\gamma} \left(\frac{1}{\omega} \bm{\tilde{\Ms}}_{j,k}  + (\lambda_{j}- (k-1)(\alpha +1)) \bm{\Ns}_{j,k} \right)\bm{\as}_{j,k-1}  \\ 
	- & \frac{\sin(\omega(\alpha  + 1))}{\gamma} \left( \frac{1}{\omega} \bm{\tilde{\Ns}}_{j,k}  - (\lambda_{j}- (k-1)(\alpha +1)) \bm{\Ms}_{j,k}  \right) \bm{\as}_{j,k-1} \nonumber, 
\end{align}
where, $\bm{\Ms}_{j,k}$ and $\bm{\as}_{j,k}$ are defined as in the system~\eqref{SystDD},
\begin{align*}
	\bm{\tilde{\Ms}}_{j,k} &= \left[ \tilde{\ms}_{j,k}^{(m,l)}  \right]_{m = 0, \ldots, L_{j,k}; \ l = m,\ldots, L_{j,k} }	 = 
	\begin{pmatrix} 
		\tilde{\ms}_{j,k}^{(0,0)} & \tilde{\ms}_{j,k}^{(0,1)}  & \cdots & \tilde{\ms}_{j,k}^{(0,L_{j,k})}  \\[1.5mm] 
		        0         & \tilde{\ms}_{j,k}^{(1,1)}  & \cdots & \tilde{\ms}_{j,k}^{(1,L_{j,k})}  \\[1.5mm] 
		        \vdots    &  \ddots            & \ddots & \vdots                   \\[1.5mm] 
				0		  &   0                & \cdots & \tilde{\ms}_{j,k}^{(L_{j,k},L_{j,k})}
	\end{pmatrix}\\
	\bm{\tilde{\Ns}}_{j,k} &= \left[ \tilde{\ns}_{j,k}^{(m,l)}  \right]_{m = 0, \ldots, L_{j,k};\ l = m,\ldots, L_{j,k} }	 = 
	\begin{pmatrix} 
		\tilde{\ns}_{j,k}^{(0,0)} & \tilde{\ns}_{j,k}^{(0,1)}  & \cdots & \tilde{\ns}_{j,k}^{(0,L_{j,k})}  \\[1.5mm] 
		        0         & \tilde{\ns}_{j,k}^{(1,1)}  & \cdots & \tilde{\ns}_{j,k}^{(1,L_{j,k})}  \\[1.5mm] 
		        \vdots    &  \ddots            & \ddots & \vdots                   \\[1.5mm] 
				0		  &   0                & \cdots & \tilde{\ns}_{j,k}^{(L_{j,k},L_{j,k})}
	\end{pmatrix}\\
	\bm{\Ns}_{j,k} &= \left[ \ns_{j,k}^{(m,l)}  \right]_{m = 0, \ldots, L_{j,k};\ l = m,\ldots, L_{j,k} }	 = 
	\begin{pmatrix} 
		\ns_{j,k}^{(0,0)} & \ns_{j,k}^{(0,1)}  & \cdots & \ns_{j,k}^{(0,L_{j,k})}  \\[1.5mm] 
		        0         & \ns_{j,k}^{(1,1)}  & \cdots & \ns_{j,k}^{(1,L_{j,k})}  \\[1.5mm] 
		        \vdots    &  \ddots            & \ddots & \vdots                   \\[1.5mm] 
				0		  &   0                & \cdots & \ns_{j,k}^{(L_{j,k},L_{j,k})}
	\end{pmatrix},
\end{align*}
with
\begin{align*}
	\tilde{\ms}_{j,k}^{(m,l)} &= (l-m)\ms_{j,k}^{(m,l)},\\
	\ns_{j,k}^{(m,l)} & =  \binom{l}{m} \omega^{l-m} \cos\left(\frac{\pi}{2} (l-m) - k\omega (\alpha + 1)  \right), \\  \\
	\tilde{\ns}_{j,k}^{(m,l)} & = (l-m)\ns_{j,k}^{(m,l)},
\end{align*}
and 
\[\bm{\as}_{j,k-1} = \left[ \as_{j,k-1}^{(0)}, \as_{j,k-1}^{(1)}, \ldots, \as_{j,k-1}^{(L_{j,k-1})}  \right]^\top\quad \text{when} \quad  L_{j,k} = L_{j,k-1}\] or 
\[\bm{\as}_{j,k-1} = \left[ \as_{j,k-1}^{(0)}, \as_{j,k-1}^{(1)}, \ldots, \as_{j,k-1}^{(L_{j,k-1})}, 0  \right]^\top \quad \text{when}\quad L_{j,k} =L_{j,k-1} + 1;\] which agrees with the condition given by~\eqref{CondDD}. With these recursive systems and taking in account  Tables~\ref{TblDN1}--\ref{TblDD2}, 
 it is easy to build an algorithm to find the coefficients of solutions.

\section{Solutions for D-D approach}\label{AppC}
Recall that $\Omega$ is an angular sector defined in \eqref{Omegadef}.
\begin{prpAppndx}\label{AppPrp1}
	Consider the following system
	 \begin{equation}\label{SystPrp1}
		\begin{split}
		\Delta u & = 0 ,\qquad \quad\text{in } \Omega,\\
		u(r,0)  & = 0,\qquad\quad \forall \ r>0, \\
		u(r,\omega) & =   c_{0} r^{\beta }, \quad \forall \ r>0,
		\end{split}
	\end{equation}
	where $c_{0}$ and $\beta$ are given real numbers. Then, the solution is given by 
	\[u(r,\theta) =  r^{\beta} \left[(a_{0} + a_{1} \log(r))\sin (\beta \theta )  + a_{1}\theta \cos(\beta \theta ) \right] ,\]
where $a_{0} = \displaystyle{\frac{c_{0}}{ \sin(\beta \omega )}}$ and $a_{1}=0$ when $\sin(\beta\omega)\neq 0$, and  if  $\sin(\beta\omega)= 0$, then $a_ {0}$ is arbitrary and  $a_{1} = \displaystyle{\frac{c_{0}}{ \omega\cos(\beta\omega )}}$. 
\end{prpAppndx}
We note that the above coefficients solve the following system
\[\begin{pmatrix}\sin{(\beta\omega)} & \omega\cos{(\beta\omega) }  \\
	0 & \sin{(\beta\omega)} 
\end{pmatrix} \begin{pmatrix} a_{0}\\ a_{1} \end{pmatrix} = \begin{pmatrix}c_{0}\\ 0 \end{pmatrix}.\]

\begin{prpAppndx}\label{AppPrp2}
Consider the following system
	 \begin{equation}\label{SystPrp2}
		\begin{split}
		\Delta u & = 0 ,\qquad \quad\text{in } \Omega,\\
		u(r,0)  & = 0,\qquad\quad \forall \ r>0, \\
		u(r,\omega) & =   c_{1} r^{\beta }\log(r), \quad \forall \ r>0,
		\end{split}
	\end{equation}
	where $c_{1}$ and $\beta$ are  given real numbers. Then, the solution is given by 
	\[u(r,\theta) = r^{\beta} \left[ (a_{0}  + a_{1}\log(r) + a_{2}(\log ^2(r) -\theta ^2))\sin (\beta \theta ) +  ( a_{1}\theta +  a_{2} 2\theta \log(r)  )\cos(\beta \theta )\right].  \]
If  $\sin(\beta\omega)\neq 0$, then 
	\begin{align*}
		a_{0} & =  -a_{1} \omega\frac{\cos(\beta \omega )}{ \sin(\beta \omega )},\\
		a_{1} & = \frac{c_{1}}{ \sin(\beta \omega )}, \\
		a_{2} & = 0.
	\end{align*}
		
If  $\sin(\beta\omega)= 0$, then 	
\begin{align*}
a_{0} & \quad\text{is arbitrary,} \\
a_{1} &=  0,\\
a_{2} &= \frac{c_{1}}{ 2\omega\cos(\beta \omega )}. 
\end{align*}
\end{prpAppndx}
We note that above coefficients solve the following system
\[\begin{pmatrix}\sin{(\beta\omega)} & \omega\cos{(\beta\omega) } & -\omega^{2}\sin{(\beta\omega)} \\
	0 & \sin{(\beta\omega)}  & 2 \omega \cos{(\beta\omega)}\\
    0 & 0 & \sin{(\beta\omega)}
\end{pmatrix} \begin{pmatrix} a_{0}\\ a_{1} \\ a_{2} \end{pmatrix} = \begin{pmatrix} 0 \\ c_{1} \\ 0  \end{pmatrix}.\]
In addition, it is easy to see that the sum of the last two solutions solves the following system 
\begin{equation}\label{SystPrp3}
	\begin{split}
	\Delta u & = 0 ,\qquad \quad\text{in } \Omega,\\
	u(r,0)  & = 0,\qquad\quad \forall \ r>0, \\
	u(r,\omega) & =   r^{\beta }(c_{0}  + c_{1} \log(r)), \quad \forall \ r>0,
	\end{split}
\end{equation}
in this case $\bm{c} = \left[c_{0}, c_{1},  0 \right]^{\top}$.
Now we generalized the last two propositions.
\begin{prpAppndx}\label{AppPrpn}
	Consider the following system
		 \begin{equation}\label{SystPrpn}
			\begin{split}
			\Delta u & = 0 ,\qquad \quad\text{in } \Omega,\\
			u(r,0)  & = 0,\qquad\quad \forall \ r>0, \\
			u(r,\omega) & =  r^{\beta }\sum_{i = 0 }^{k-1} c_{i}\log(r)^{i}, \quad \forall \ r>0,
			\end{split}
		\end{equation}
		where $\beta$ and $c_{i}$ are real numbers given for $i\in\{0,\ldots, k-1\}$. Then, the solution is   
		 \[u(r,\theta) =  r^{\beta} \sum_{j = 0}^{k} a_{j} \sum_{i = 0}^{j} C_{j-i}^{(i)}\frac{\df^{i} \sin (\beta \theta)}{\df \theta ^{i}}  \theta^{i} (\log(r))^{j-i}  ,\]
	with 
	\[C_{j}^{(i)} = \begin{cases}  1,&  i = 0, \\
									\frac{1}{i!\beta^{i}}\prod_{l = 1}^{i}( j + l), & i\geq 1.
		  	        \end{cases}\]
and the coefficients solve the following system
\[\bm{M}\bm{a} = \bm{c}, \]
where
\[\bm{M} =  \displaystyle{\begin{pmatrix}\sin{(\beta\omega)} & \omega\cos{(\beta\omega) } & \cdots & \cdots & \frac{\omega^{k}}{\beta^{k}} \frac{\df^{k} \sin{(\beta\theta)} }{\df \theta^{k}}|_{\theta= \omega} \\
	0 & \sin{(\beta\omega)}  & 2 \omega \cos{(\beta\omega)} & \cdots & \frac{k\omega^{k-1}}{\beta^{k-1}} \frac{\df^{k-1} \sin{(\beta\theta)} }{\df \theta^{k-1}}|_{\theta= \omega} \\
	0 & 0 & \sin{(\beta\omega)} & \cdots &  \frac{k(k-1)\omega^{k-2}}{2! \beta^{k-2}} \frac{\df^{k-2} \sin{(\beta\theta)} }{\df \theta^{k-2}}|_{\theta= \omega}  \\
    \vdots & \vdots & \vdots & \ddots & \vdots\\
	0 & 0 & 0 & \sin{(\beta\omega)} &  \frac{k \omega }{\beta} \frac{\df \sin{(\beta\theta)} }{\df \theta}|_{\theta= \omega}    \\
	0 & 0 & 0 & 0 & \sin{(\beta\omega)}
\end{pmatrix}}\]
or 
\[\bm{M} = \left[ m_{i,j}\right]_{i = 0, \ldots, k;  j = 0,\ldots, k } =  \begin{cases}
	\frac{1}{i! \beta^{j-i}} \frac{\df^{i}(\theta^{j})}{\df \theta^{i}}|_{\theta=\omega} \frac{\df^{j-i}\sin{(\beta\theta)}}{\df \theta^{j-i
	}}|_{\theta = \omega}, & i \leq j, \\
	0 ,&  i > j,
\end{cases}  \]
$\bm{a}= [a_{0},a_{1},\ldots,a_{k} ]^\top $ and  $\bm{c}= [c_{0},c_{1},\ldots,c_{k-1},0]^{\top}$.  In the case of $\sin{(\beta\omega)}\neq 0 $ we have that $a_{k}=0$  and  $a_{k-1} = \dfrac{c_{k-1}}{\sin{(\beta\omega)}}$. If $\sin{(\beta\omega)} = 0$ the system have infinite solutions with $a_{0}$ arbitrary, and $a_{k} = \dfrac{c_{k-1}}{k\omega\cos{(\beta\omega)}}$.  
\end{prpAppndx}
The proofs of the last three Propositions are similar to those by Mghazli~\cite[Proposition A1, A2 and  A3]{Mghazli1992} in the Appendix. In the present case we must take into account the Dirichlet boundary condition on $\theta = 0$ and $\theta=\omega$,  we also consider $\beta\in \R$, but this is not an issue. In addition, is not difficult to check the proposed solutions.


\bibliographystyle{abbrvnat}
\bibliography{BibliographyComp}   

\begin{thebibliography}{39}
\providecommand{\natexlab}[1]{#1}
\providecommand{\url}[1]{\texttt{#1}}
\expandafter\ifx\csname urlstyle\endcsname\relax
  \providecommand{\doi}[1]{doi: #1}\else
  \providecommand{\doi}{doi: \begingroup \urlstyle{rm}\Url}\fi

\bibitem[Antipov et~al.(2001)]{Antipov2001}
Y.~Antipov et~al.
\newblock Mathematical model of delamination cracks on imperfect interfaces.
\newblock \emph{International Journal of Solids and Structures}, 38:\penalty0
  6665--6697, 2001.

\bibitem[Babuška(1970)]{Babuska1970}
I.~Babuška.
\newblock Finite element method for domains with corners.
\newblock \emph{Computing}, 6:\penalty0 264--273, 1970.

\bibitem[Brenner and Carstensen(2004)]{BrennerCarstensen2004}
S.~C. Brenner and C.~Carstensen.
\newblock Finite element methods (ch. 4).
\newblock In E.~Stein, R.~de~Borst, and T.~J.~R. Hughes, editors,
  \emph{Encyclopedia of Computational Mechanics, Volume 1: Fundamentals}, pages
  73--114. Wiley, 2004.

\bibitem[Costabel and Dauge(1993)]{CostabelDauge1993}
M.~Costabel and M.~Dauge.
\newblock Construction of corner singularities for agmon-douglis-nirenberg
  elliptic systems.
\newblock \emph{Mathematische Nachrichten}, 162:\penalty0 209--237, 1993.

\bibitem[Costabel and Dauge(1996)]{CostabelDauge1996}
M.~Costabel and M.~Dauge.
\newblock A singularly mixed boundary value problem.
\newblock \emph{Communications in Partial Differential Equations}, 21:\penalty0
  1919--1949, 1996.

\bibitem[Dauge(1988)]{Dauge1988}
M.~Dauge.
\newblock \emph{Elliptic Boundary Value Problems on Corner Domains}.
\newblock Springer, Berlin, Heidelberg, 1988.

\bibitem[Grisvard(1985)]{Grisvard1985}
P.~Grisvard.
\newblock \emph{Elliptic problems in nonsmooth domains}.
\newblock Pitman Advanced Publishing Program, Boston, London, Melbourne, 1985.

\bibitem[Grisvard(1992)]{Grisvard1992}
P.~Grisvard.
\newblock \emph{Singularities in boundary value problems}.
\newblock Masson and Springer-Verlag, Paris, 1992.

\bibitem[Jiménez-Alfaro and Mantič(2023)]{Jimenez-Alfaro2023}
S.~Jiménez-Alfaro and V.~Mantič.
\newblock Crack tip solution for mode iii cracks in spring interfaces.
\newblock \emph{Engineering Fracture Mechanics}, 288:\penalty0 109293, 2023.

\bibitem[Jiménez-Alfaro et~al.(2020)Jiménez-Alfaro, Villalba, and
  Mantič]{Jimenez-Alfaro2020}
S.~Jiménez-Alfaro, V.~Villalba, and V.~Mantič.
\newblock Singular elastic solutions in corners with spring boundary conditions
  under anti-plane shear.
\newblock \emph{International Journal of Fracture}, 223:\penalty0 197--220,
  2020.

\bibitem[Kondratiev(1967)]{Kondratiev1967}
V.~A. Kondratiev.
\newblock Boundary value problems for elliptic equations in domains with
  conical or angular points.
\newblock \emph{Trans. Moscow Math. Soc.}, 16:\penalty0 227--313, 1967.

\bibitem[Kozlov et~al.(1997)Kozlov, Maz'ya, and
  Rossmann]{KozlovMazyaRossmann1997}
V.~A. Kozlov, V.~G. Maz'ya, and J.~Rossmann.
\newblock \emph{Elliptic Boundary Value Problems in Domains with Point
  Singularities}.
\newblock American Mathematical Society, Providence, Rhode Island, 1997.

\bibitem[Kozlov et~al.(2001)Kozlov, Maz'ya, and
  Rossmann]{KozlovMazyaRossmann2001}
V.~A. Kozlov, V.~G. Maz'ya, and J.~Rossmann.
\newblock \emph{Spectral Problems Associated with Corner Singularities of
  Solutions of Elliptic Equations}.
\newblock American Mathematical Society, Providence, Rhode Island, 2001.

\bibitem[Kufner and Sändig(1987)]{KufnerSandig1987}
A.~Kufner and A.-M. Sändig.
\newblock \emph{Some Applications of Weighted Sobolev Spaces}.
\newblock Teubner Verlag, Wiesbaden, 1987.

\bibitem[Leguillon and Sanchez-Palencia(1987)]{Leguillon1987}
D.~Leguillon and E.~Sanchez-Palencia.
\newblock \emph{Computation of singular solutions in elliptic problems and
  elasticity}.
\newblock Masson, Paris, 1987.

\bibitem[Lenci(2001)]{Lenci2001}
S.~Lenci.
\newblock Analysis of a crack at a weak interface.
\newblock \emph{International Journal of Fracture}, 108:\penalty0 275--290,
  2001.

\bibitem[Mantič et~al.(2024)]{Mantic2024}
V.~Mantič et~al.
\newblock A new crack-tip element for the logarithmic stress-singularity of
  mode-iii cracks in spring interfaces.
\newblock \emph{Computational Mechanics}, 74:\penalty0 641--660, 2024.

\bibitem[Maz'ya and Rossmann(1992)]{MazyaRossmann1992}
V.~Maz'ya and J.~Rossmann.
\newblock On a problem of babuška (stable asymptotics of the solution to the
  dirichlet problem for elliptic equations of second order in domains with
  angular points).
\newblock \emph{Mathematische Nachrichten}, 155:\penalty0 199--220, 1992.

\bibitem[Medková(2018)]{Medkova2018}
D.~Medková.
\newblock \emph{The Laplace Equation. Boundary Value Problems on Bounded and
  Unbounded Lipschitz Domains}.
\newblock Springer, Cham, Switzerland, 2018.

\bibitem[Mghazli(1992)]{Mghazli1992}
Z.~Mghazli.
\newblock Regularity of an elliptic problem with mixed dirichlet-robin boundary
  conditions in a polygonal domain.
\newblock \emph{CALCOLO}, 29:\penalty0 241--267, 1992.

\bibitem[Mishuris(2001)]{Mishuris2001}
G.~Mishuris.
\newblock Interface crack and nonideal interface concept (mode iii).
\newblock \emph{International Journal of Fracture}, 107, 2001.

\bibitem[Mishuris(1999)]{Mishuris1999}
G.~S. Mishuris.
\newblock Stress singularity at a crack tip for various intermediate zones in
  bimaterial structures (mode iii).
\newblock \emph{International Journal of Solids and Structures}, 36:\penalty0
  999--1015, 1999.

\bibitem[Mishuris and Kuhn(2001)]{MishurisKuhn2001}
G.~S. Mishuris and G.~Kuhn.
\newblock Asymptotic behaviour of the elastic solution near the tip of a crack
  situated at a nonideal interface.
\newblock \emph{ZAMM Zeitschrift für Angewandte Mathematik und Mechanik}, 81,
  2001.

\bibitem[Nazarov and Plamenevsky(1994)]{NazarovPlamenevsky1994}
S.~Nazarov and B.~Plamenevsky.
\newblock \emph{Elliptic Problems in Domains with Piecewise Smooth Boundaries}.
\newblock Walter de Gruyter, Berlin, 1994.

\bibitem[Nicaise(1993)]{Nicaise1993}
S.~Nicaise.
\newblock \emph{Polygonal Interface Problems}.
\newblock Peter Lang, 1993.

\bibitem[Sayas et~al.(2019)Sayas, Brown, and Hassell]{SayasVariational}
F.~Sayas, T.~Brown, and M.~Hassell.
\newblock \emph{Variational Techniques for Elliptic Partial Differential
  Equations: Theoretical Tools and Advanced Applications}.
\newblock CRC Press, 2019.

\bibitem[Sinclair(1996)]{Sinclair1996}
G.~B. Sinclair.
\newblock On the influence of cohesive stress-separation laws on elastic stress
  singularities.
\newblock \emph{Journal of Elasticity}, 44, 1996.

\bibitem[Sinclair(1999)]{Sinclair1999}
G.~B. Sinclair.
\newblock A note on the removal of further breakdowns in classical solutions of
  laplace’s equation on sectorial regions.
\newblock \emph{Journal of Elasticity}, 56:\penalty0 247--252, 1999.

\bibitem[Sinclair(2004)]{Sinclair2004}
G.~B. Sinclair.
\newblock Stress singularities in classical elasticity-i: Removal,
  interpretation, and analysis.
\newblock \emph{Applied Mechanics Reviews}, 57:\penalty0 251--298, 2004.

\bibitem[Sinclair(2009)]{Sinclair2009}
G.~B. Sinclair.
\newblock A note on the influence of cohesive stress-separation laws on elastic
  stress singularities in antiplane shear.
\newblock \emph{Journal of Elasticity}, 94:\penalty0 87--93, 2009.

\bibitem[Sinclair(2015)]{Sinclair2015}
G.~B. Sinclair.
\newblock On the influence of adhesive stress-separation laws on elastic stress
  singularities.
\newblock \emph{Journal of Elasticity}, 118:\penalty0 187--206, 2015.

\bibitem[Strang and Fix(1973)]{StrangFix1973}
G.~Strang and G.~J. Fix.
\newblock \emph{An analysis of the finite element method}.
\newblock Prentice-Hall, Inc., Englewood Cliffs, NJ, 1973.

\bibitem[Szabó and Babuška(1991)]{SzaboBabuska1991}
B.~Szabó and I.~Babuška.
\newblock \emph{Finite Element Analysis}.
\newblock Wiley, New York, 1991.

\bibitem[Ueda et~al.(2006)]{Ueda2006}
S.~Ueda et~al.
\newblock On the stiffness of spring model for closed crack.
\newblock \emph{International Journal of Engineering Science}, 44:\penalty0
  874--888, 2006.

\bibitem[Watanabe et~al.(2007)]{Watanabe2007}
K.~Watanabe et~al.
\newblock Closed interface crack with singular spring stiffness model.
\newblock \emph{International Journal of Engineering Science}, 45:\penalty0
  210--226, 2007.

\bibitem[Williams(1952)]{Williams1952}
M.~L. Williams.
\newblock Stress singularities resulting from various boundary conditions in
  angular corners of plates in extension.
\newblock \emph{Journal of Applied Mechanics}, 19:\penalty0 526--528, 1952.

\bibitem[Wolfram(1991)]{Wolfram1991}
S.~Wolfram.
\newblock \emph{Mathematica: A system for doing mathematics by computer}.
\newblock Addison-Wesley, Redwood City, CA, 2 edition, 1991.

\bibitem[Wu et~al.(2017)Wu, Zhang, and Wan]{Wu2017}
J.~Wu, L.~Zhang, and L.~Wan.
\newblock A mode-iii crack under adhesion studied by non-uniform linear spring
  models.
\newblock \emph{Acta Mechanica}, 228:\penalty0 1621--1629, 2017.

\bibitem[Yosibash(2012)]{Yosibash2012}
Z.~Yosibash.
\newblock \emph{Singularities in elliptic boundary value problems and
  elasticity and their connection with failure initiation}.
\newblock Springer, New York, 2012.

\end{thebibliography}
\end{document}